\documentclass[aop]{imsart}

\RequirePackage{amsthm,amsmath,amsfonts,amssymb}
\RequirePackage[numbers]{natbib}
\RequirePackage[colorlinks,citecolor=blue,urlcolor=blue]{hyperref}
\RequirePackage{graphicx}
\usepackage[utf8]{inputenc}
\usepackage{stmaryrd}
\usepackage{wasysym}
\usepackage{mathrsfs}
\usepackage{cleveref}
\usepackage{enumitem}
\usepackage{relsize} 
\usepackage{dsfont}
\usepackage{soul}
\usepackage{filecontents}
\usepackage{color}
\usepackage[all]{xy}
\usepackage{float}
\usepackage{bm}
\usepackage{mathtools}
\usepackage{thmtools}
\usepackage{comment}
\usepackage{tikz}
\usepackage{ytableau}
\usepackage{mathdots}
\usepackage[leftcaption]{sidecap} 

\startlocaldefs
\numberwithin{equation}{section}
\theoremstyle{plain}

\newtheorem{thm}{Theorem}[section]
\newtheorem{example}[thm]{Example}
\newtheorem{prop}[thm]{Proposition}

\newtheorem{lemma}[thm]{Lemma}
\newtheorem{cor}[thm]{Corollary}

\theoremstyle{definition}
\newtheorem{defi}{Definition}

\newtheorem{rmk}{Remark}

\DeclareSymbolFont{bbold}{U}{bbold}{m}{n}
\DeclareSymbolFontAlphabet{\mathbbold}{bbold}

\makeatletter
\newcommand{\subalign}[1]{%
  \vcenter{%
    \Let@ \restore@math@cr \default@tag
    \baselineskip\fontdimen10 \scriptfont\tw@
    \advance\baselineskip\fontdimen12 \scriptfont\tw@
    \lineskip\thr@@\fontdimen8 \scriptfont\thr@@
    \lineskiplimit\lineskip
    \ialign{\hfil$\m@th\scriptstyle##$&$\m@th\scriptstyle{}##$\hfil\crcr
      #1\crcr
    }%
  }%
}
\makeatother

\DeclarePairedDelimiter{\abs}{\lvert}{\rvert} 
\DeclarePairedDelimiter{\floor}{\lfloor}{\rfloor}

\newcommand{\R}{\mathbb{R}}
\newcommand{\Z}{\mathbb{Z}}
\newcommand{\Q}{\mathbb{Q}}
\newcommand{\N}{\mathbb{N}}
\newcommand{\C}{\mathbb{C}}
\newcommand{\F}{\mathbb{F}}

\newcommand{\E}{\mathbb{E}}

\newcommand{\bbone}{\mathbbold{1}}
\newcommand{\la}{\lambda}
\newcommand{\La}{\Lambda}
\newcommand{\hLa}{\hat{\Lambda}}
\newcommand{\eps}{\epsilon}

\newcommand{\tX}{\tilde{X}}
\newcommand{\tB}{\tilde{B}}
\newcommand{\ta}{\tilde{a}}

\newcommand{\tI}{\tilde{I}}
\newcommand{\tV}{\tilde{V}}
\newcommand{\tPi}{\tilde{\Pi}}

\newcommand{\inj}{\hookrightarrow}
\newcommand{\tth}{^{th}}

\newcommand{\M}{\mathbb{M}}

\renewcommand{\L}{\Lambda}
\newcommand{\tL}{\tilde{\Lambda}}
\newcommand{\tLa}{\tilde{\Lambda}}

\newcommand{\sB}{\mathscr{B}}
\newcommand{\sS}{\mathscr{S}}

\newcommand{\tP}{\tilde{P}}

\newcommand{\tA}{\tilde{A}}

\newcommand{\cM}{\mathcal{M}}
\newcommand{\tM}{\tilde{M}}

\renewcommand{\vec}[1]{\boldsymbol{#1}}

\DeclareMathOperator{\len}{len}

\DeclareMathOperator{\Span}{span}

\DeclareMathOperator{\coker}{coker}
\DeclareMathOperator{\Sig}{Sig}
\DeclareMathOperator{\bSig}{\overline{Sig}}
\DeclareMathOperator{\Proj}{Proj}

\DeclareMathOperator{\SN}{SN}

\DeclareMathOperator{\diag}{diag}

\DeclareMathOperator{\Mat}{Mat}
\DeclareMathOperator{\val}{val}
\DeclareMathOperator{\col}{col}
\DeclareMathOperator{\corank}{corank}
\DeclareMathOperator{\rank}{rank}
\DeclareMathOperator{\Stab}{Stab}
\DeclareMathOperator{\Law}{Law}

\newcommand{\sqbinom}[2]{\begin{bmatrix}#1\\ #2\end{bmatrix}}

\newcommand{\Pois}{\mathcal{S}}

\newcommand{\bZ}{\bar{\mathbb{Z}}}

\newcommand{\cU}{\mathcal{U}}
\newcommand{\bzero}{\mathbf{0}}
\newcommand{\tOmega}{\widetilde{\Omega}}
\DeclareMathOperator{\Poiss}{Poiss}
\newcommand{\tPois}{\tilde{\Pois}}

\DeclareMathOperator{\jumps}{jumps}

\newcommand{\ou}{o_{unif}}

\newcommand{\GL}{\mathrm{GL}}

\newcommand{\cL}{\mathcal{L}}
\newcommand{\sSig}{\overline{\Sig}}

\renewcommand{\l}{\lambda}

\endlocaldefs

\begin{document}

\begin{frontmatter}

\title{Reflecting Poisson walks and dynamical universality in $p$-adic random matrix theory}
\runtitle{Reflecting walks and universality in $p$-adic random matrix theory}

\begin{aug}
\author[A]{\fnms{Roger}~\snm{Van Peski}\ead[label=e1]{rvp@mit.edu}}

\address[A]{Department of Mathematics, KTH Royal Institute of Technology, Stockholm, Sweden\printead[presep={,\ }]{e1}}

\end{aug}

\begin{abstract}
We prove dynamical local limits for the singular numbers of $p$-adic random matrix products at both the bulk and edge. The limit object which we construct, the \emph{reflecting Poisson sea}, may thus be viewed as a $p$-adic analogue of line ensembles appearing in classical random matrix theory. However, in contrast to those it is a discrete space Poisson-type particle system with only local reflection interactions and no obvious determinantal structure. The limits hold for any $\mathrm{GL}_n(\mathbb{Z}_p)$-invariant matrix distributions under weak universality hypotheses, with no spatial rescaling. 
\end{abstract}



\end{frontmatter}

\maketitle

\tableofcontents

\section{Introduction}

\subsection{Preface} The goal of this work is to define a new interacting particle system on $\Z$, which we call the \emph{reflecting Poisson sea}, and show that it is a universal object in random matrix theory. Concretely, it is a bi-infinite collection of Poisson random walkers $(\Pois^{\mu,2\infty}_i(T))_{i \in \Z}$ in continuous time $T \in \R_{\geq 0}$, which remain ordered (meaning $\ldots \geq \Pois^{\mu,2\infty}_i(T) \geq \Pois^{\mu,2\infty}_{i+1}(T) \geq \ldots$) for all time $T$ due to certain local reflection interactions between the walkers $\Pois^{\mu,2\infty}_i(T)$.

\begin{figure}[htbp]
\centering
\includegraphics[scale=.4]{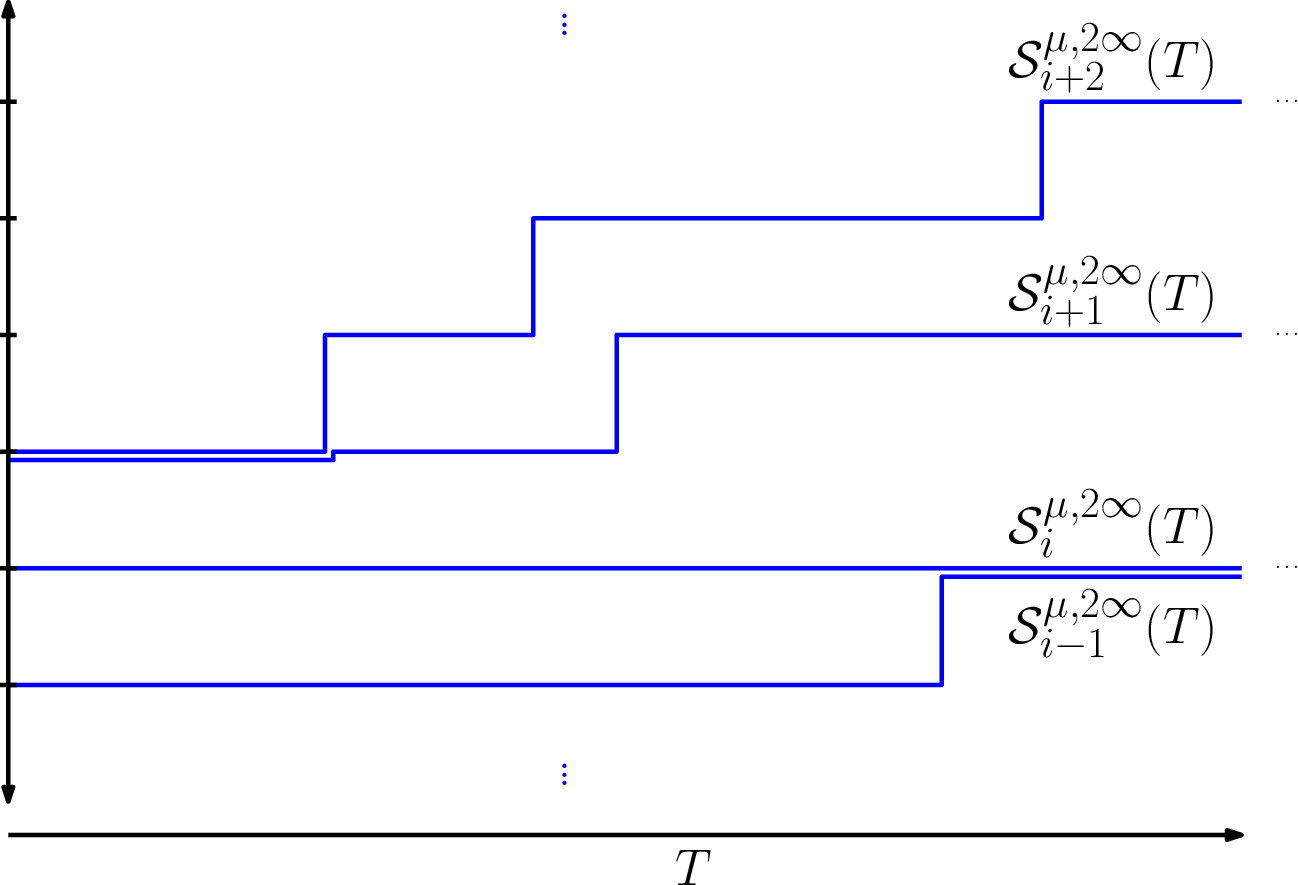}
\caption{A sample path trajectory of $(\Pois^{\mu,2\infty}_i(T))_{i \in \Z}$, where the vertical direction represents space and the horizontal direction represents time, and there are infinitely many paths above and below those pictured. When $\Pois^{\mu,2\infty}_j(T) = \Pois^{\mu,2\infty}_{j+1}(T)$ we draw the paths slightly shifted so both are visible.
}\label{fig:sea_cartoon}
\end{figure}


Other ordered collections of random walks, in both discrete and continuous space, have emerged in the last decade as universal limits for a variety of other random processes. These include discrete tiling models \cite{cohn1996local,johansson2002non,sheffield2005random}, polynuclear growth and polymer models \cite{ferrari2005one,corwin2014brownian,corwin2016kpz,dauvergne2023uniform}, and random matrices over $\R$ and $\C$. 

In the random matrix results, one considers the singular values of a complex matrix or eigenvalues of a Hermitian matrix of size $N$, which form a random ordered $N$-tuple $\la_1 \geq \ldots \geq \la_N$ of real numbers. A now-vast literature is devoted to studying various probabilistic limits of these tuples as $N \to \infty$. Introducing Markov dynamics on the appropriate spaces of matrices, one then obtains Markov processes on such tuples, equivalently collections of interacting random walks on $\R$. For instance, the eigenvalues of Brownian motion on the space of $N \times N$ Hermitian matrices are described by $N$ independent Brownian motions conditioned not to intersect for all time, the so-called Dyson Brownian motion \cite{dyson1962brownian}.

There are also natural discrete-time Markov processes. For example, one may view products of independent matrices $A_1, A_2A_1, A_3A_2A_1,\ldots$ as a Markov process $A_\tau \cdots A_1, \tau = 0,1,2,\ldots$ with the number of products $\tau$ playing the role of time, see for example \cite{bellman1954limit,furstenberg1960products,gol1989lyapunov,crisanti2012products,forrester2016singular,borodin2018product,akemann2019integrable,gorin2018gaussian,liu2018lyapunov,ahn2022extremal} and the references therein. One may obtain discrete-time stochastic processes in other ways, from sums of random matrices \cite{ahn2023airy} or successive submatrices of a larger ambient matrix \cite{baryshnikov2001gues,johansson2006eigenvalues,gorin2020universal,huang2022eigenvalues}. 

Our results concern the matrix product Markov chain, but over the $p$-adic numbers\footnote{Basic background on the $p$-adic numbers and linear algebra over them is given in \Cref{sec:p-adic}.}. To any nonsingular $p$-adic matrix $A \in \Mat_N(\Z_p)$ is associated a tuple of \emph{singular numbers}, nonnegative integers $\SN(A)_1 \geq \ldots \geq \SN(A)_N \geq 0$ for which there exist $U,V \in \GL_N(\Z_p)$ such that
\begin{equation}
A = U\diag_{N \times N}(p^{\SN(A)_1},\ldots,p^{\SN(A)_N})V.
\end{equation}
Structurally these are analogous to singular values of a complex matrix, but in the $p$-adic context they have the additional meaning of parametrizing the matrix's cokernel, a finite abelian $p$-group
\begin{equation}\label{eq:cok_intro}
\coker(A) := \Z_p^N/A\Z_p^N \cong \bigoplus_{i=1}^N \Z/p^{\SN(A)_i}\Z.
\end{equation}
For various random matrices $A$, the resulting random groups have been studied since the 1980s \cite{friedman-washington} in number theory (the so-called Cohen-Lenstra heuristics \cite{cohen-lenstra}), and more recently in connection with random graphs \cite{clancy2015cohen,wood2017distribution,meszaros2020distribution,nguyen2022local} and random simplicial complexes \cite{kahle2014topology,kahle2020cohen} as well. See the ICM notes \cite{wood2023probability} for a fuller account and bibliography.

We consider the singular numbers of products $A_1, A_2A_1, \ldots,A_\tau \cdots A_2 A_1,\ldots$ of iid matrices $A_i \in \Mat_N(\Z_p)$ as a stochastic process on the space
\begin{equation}
\Sig_N := \{(\la_1,\ldots,\la_N) \in \Z^N: \la_1 \geq \ldots \geq \la_N\}
\end{equation}
in discrete time $\tau$, as was done in \cite{van2020limits,nguyen2022universality,van2023local}. Each singular number $\SN(A_\tau \cdots A_1)_i$ is a (random) nondecreasing function of $\tau$ which begins at $\SN(I)_i = 0$, see \Cref{fig:matrix_paths}. The trajectory of this process looks similar to \Cref{fig:sea_cartoon}, except that time is discrete, there is a top and bottom path, and the upward jumps may have size greater than $1$. 

\begin{figure}[htbp]
\centering
\includegraphics[scale=.4]{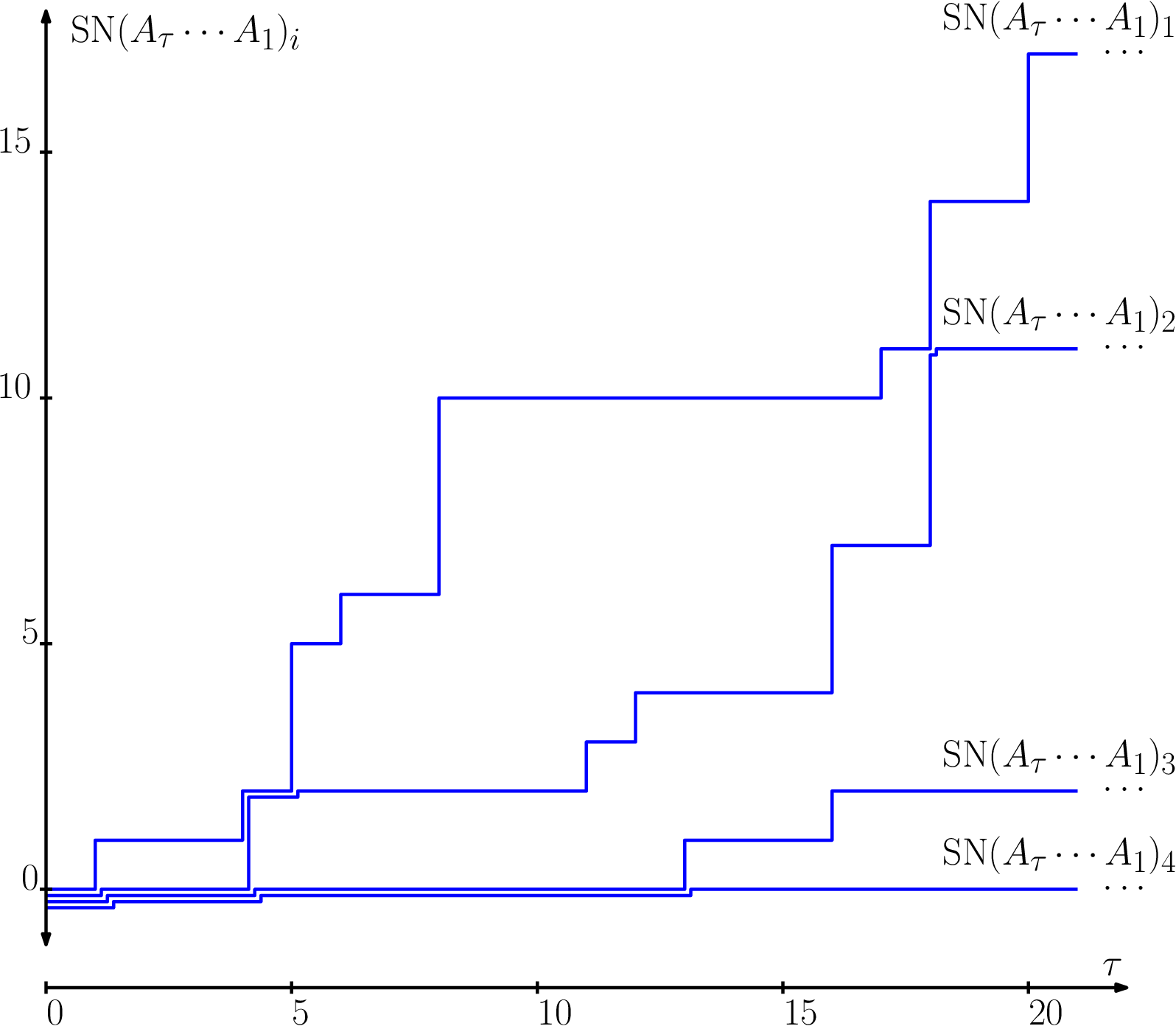}
\caption{Plot of a realization of the paths $\SN(A_\tau \cdots A_1)_i, i = 1,2,3,4$, depicted as piecewise-constant functions on $\R_{\geq 0}$, where the matrices $A_j \in \Mat_4(\Z_2)$ have iid entries distributed by the additive Haar measure on $\Z_2$ (simulated on SAGE, data as in \cite[Figure 1]{van2020limits}). As in \Cref{fig:sea_cartoon} we show equal singular numbers by paths slightly below one another.
}\label{fig:matrix_paths}
\end{figure}

For different choices of distribution on the matrices $A_j$, this stochastic process will in general behave very differently. The surprising observation which precipitated this work is that for many example distributions on the matrices $A_j$, we found that the evolution of the singular numbers $\SN(A_\tau \cdots A_1)_i$ for large $i \gg 1$ (in other words, singular numbers far away from the largest one) converged to the same nontrivial limit as the matrix size $N$ goes to $\infty$.

Slightly more precisely, given some `bulk observation points' $(r_N)_{N \geq 1}$ where $1 \ll r_N \ll N$, and iid $N \times N$ matrices $A_i$, we study the joint evolution of singular numbers
\begin{equation}
(\ldots,\SN(A_\tau \cdots A_1)_{r_N-1},\SN(A_\tau \cdots A_1)_{r_N},\SN(A_\tau \cdots A_1)_{r_N+1},\ldots)
\end{equation}
close to the $r_N\tth$ one. Our results \Cref{thm:uniform_matrix_intro} and \Cref{thm:universality_intro} are precise versions of the imprecise statement that the evolution of this tuple in discrete time $\tau$ converges, in joint distribution across multiple times $\tau$, to the reflecting Poisson sea as $N \to \infty$ and $\tau$ is scaled appropriately with $N$. These are Poisson-type limits: we do not rescale the singular numbers $\SN(A_\tau \cdots A_1)_i$ at all, but we do have to rescale the discrete time $\tau$ to a continuous parameter $T$ to arrive at a meaningful limit. A very similar statement is true for the \emph{edge limit} of the joint evolution of singular numbers $\SN(A_\tau \cdots A_1)_{N-i}, i = 0,1,\ldots$ close to the last one, with the same Poisson-type dynamics, see \Cref{thm:rmt_edge_intro}. 

This work is part of the same program as \cite{van2023local}, which established bulk local limits as above at a single time $\tau$ for certain choices of distribution on the $A_i$. The results here extend to multiple times, and also hold more universally: we establish dynamical convergence to the reflecting Poisson sea using minimal assumptions on the matrix distribution. The results of \cite{van2023local} and the earlier related work \cite{van2020limits} used techniques from integrable probability, which are powerful but apply only to particular matrix ensembles. The proof techniques we introduce here are more explicit and robust, apply more broadly, and are essentially disjoint from the ones we used in those works. Both yield different results and neither currently supersedes the other; in the cases treated in \cite{van2023local} where both apply, combining them yields stronger results which we state as \Cref{thm:uniform_matrix_intro} and \Cref{thm:haar_corner_intro} below. 

Structurally, from the perspective of the complex random matrix literature, the reflecting Poisson sea may be viewed as an analogue of the infinite Dyson process studied in \cite{katori2010non,osada2012infinite,osada2013interacting,osada2020infinite,spohn1987interacting,tsai2016infinite}. This similarly is a particle system with a bi-infinite ordered collection of particles and arises in random matrix bulk limits. The `edge' version of the reflecting Poisson sea, which governs the evolution of the smallest singular numbers in \Cref{thm:rmt_edge_intro}, may likewise be seen as a $p$-adic analogue of the Airy line ensemble of \cite{corwin2014brownian} (or, more properly, a slightly different line ensemble studied in \cite{ahn2022extremal} in the context of complex matrix products). Probabilistically, however, the reflecting Poisson sea behaves quite differently from these structural cousins. They are all examples of \emph{Gibbsian line ensembles}, and most works on them rely on determinantal point process structure, or following \cite{corwin2014brownian} on Gibbs resampling properties. Both of these come from conditioning the lines to never intersect\footnote{While the above examples of Gibbsian line ensembles are in continuous space $\R$, discrete-space examples also exist, see e.g. \cite{borodin2013markov}.}, making their influence on one another in the dynamics highly nonlocal. By contrast, the reflecting Poisson sea is an interacting particle system with only local interactions. 

From the other perspective of \eqref{eq:cok_intro} and the random groups literature, the $p$-adic matrix product process defines a growth process on finite abelian $p$-groups. From this angle, the reflecting Poisson sea defines a continuous-time growth process $\bigoplus_{i \in \Z} \Z/p^{\Pois^{\mu,2\infty}_i(T)}\Z$ on the space of infinitely-generated $p^\infty$-torsion abelian groups. \Cref{thm:uniform_matrix_intro} and our results below are then Poisson-type limit theorems for this process, showing that certain discrete-time stochastic processes on abelian $p$-groups coming from matrix products converge to it. We note that for a fixed number $\tau$ of matrix products, the $N \to \infty$ limit process on cokernels was studied and given a group-theoretic interpretation in \cite{nguyen2022universality}, and we expect the reflecting Poisson sea to describe the further $\tau \to \infty$ limit of that process. It will be interesting to see if a more group-theoretic description or interpretation of the reflecting Poisson sea itself can be found.

Let us now describe these results in more detail.

\subsection{The reflecting Poisson sea} \label{subsec:sea} It is helpful to first speak of the finite version $\Pois^{\nu,n}(T) = (\Pois^{\nu,n}_1(T),\ldots,\Pois^{\nu,n}_n(T))$. This is a collection of $n$ Poisson random walks on $\Z$, started at integer positions $\nu_1 \geq \nu_2 \geq \ldots \geq \nu_n$ at time $T=0$, which interact as follows (see \Cref{fig:finite_poisson_walkers}):
\begin{enumerate}[label=(\Roman*)]
\item Each walk $\Pois^{\nu,n}_i(T)$ has an independent exponential clock of rate $t^i$, where $t \in (0,1)$ is a fixed parameter not to be confused with time.
\item When the clock belonging to $\Pois^{\nu,n}_i$ rings at some time $T_0$, it increases by $1$, \emph{unless} $\Pois^{\nu,n}_i(T_0) = \Pois^{\nu,n}_{i-1}(T_0)$ at that time.
\item In the latter case, the `next available walker' takes the jump instead: when one has $\Pois^{\nu,n}_i(T_0) = \ldots = \Pois^{\nu,n}_{i-d}(T_0) < \Pois^{\nu,n}_{i-d-1}(T_0)$ for some $d$ (formally taking $\Pois^{\nu,n}_0 = \infty$), the walker $\Pois^{\nu,n}_{i-d}$ will instead jump at $T_0$.\label{item:interact}
\end{enumerate}

\begin{figure}[htbp]
\centering
\includegraphics[scale=.5]{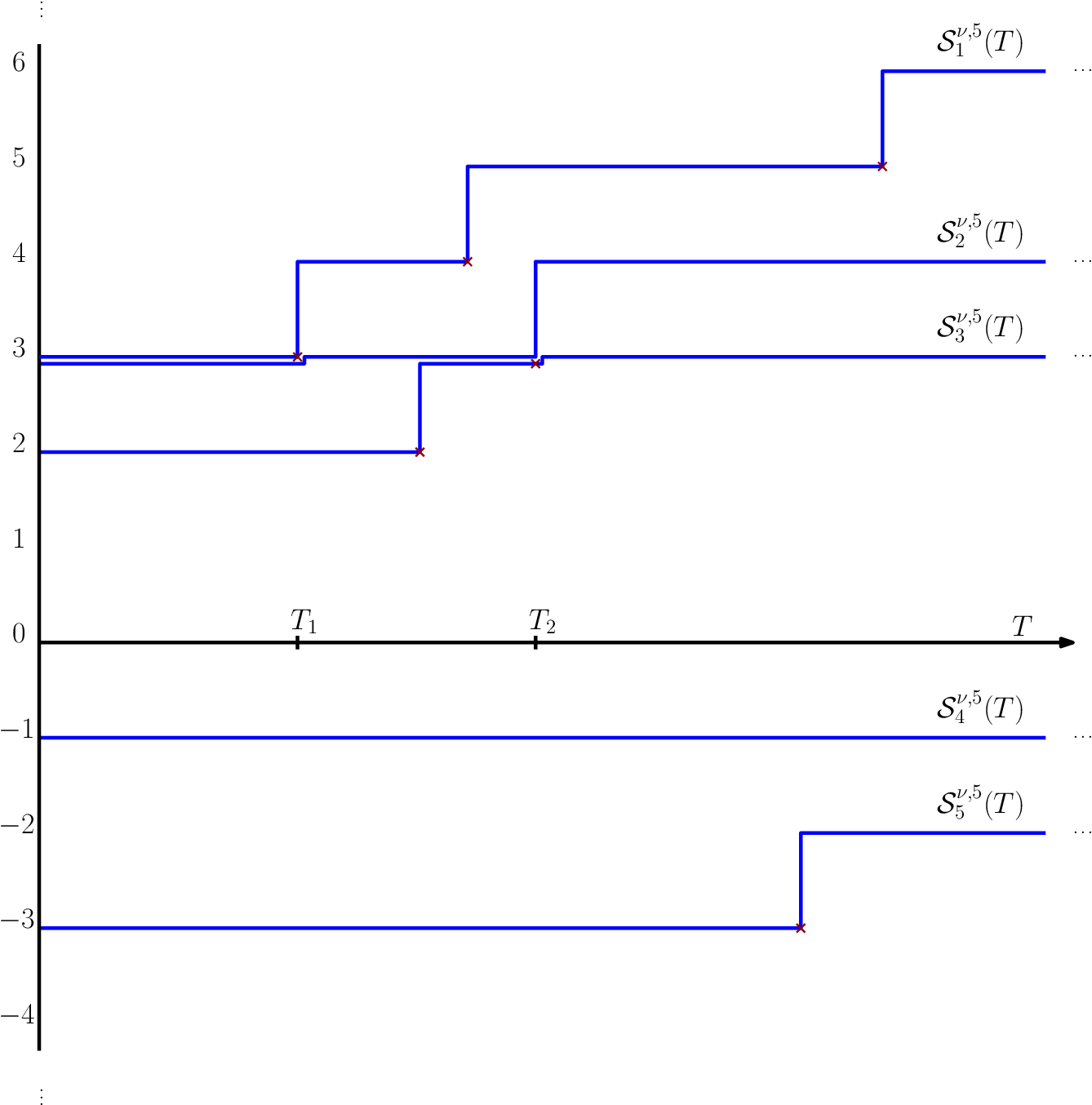}
\caption{A sample trajectory of $\Pois^{\nu,n}(T)$ as $T \geq 0$ varies, for $n=5$ and $\nu = (3,3,2,-1,-3)$. We indicate when a path's clock rings by a cross on the path, and draw two paths at the same position slightly below one another. The top two paths $\Pois^{\nu,n}_1$ and $\Pois^{\nu,n}_2$ both begin at position $3$ and remain there until $\Pois^{\nu,n}_1$'s clock rings at time $T_1 > 0$; later, $\Pois^{\nu,n}_3$ jumps to position $3$ and then has its clock ring again at time $T_2 > T_1$, but because it is blocked by $\Pois^{\nu,n}_2$, the latter path jumps instead by condition \ref{item:interact}.
}\label{fig:finite_poisson_walkers}
\end{figure}

The process $\Pois^{\nu,n}(T) = (\Pois^{\nu,n}_1(T),\ldots,\Pois^{\nu,n}_n(T))$ may equivalently be viewed as a vector of independent Poisson walks which reflects off the walls of the positive type $A$ Weyl chamber $\{(x_1,\ldots,x_n) \in \R^n: x_1 \geq \ldots \geq x_n\}$. The reflecting Poisson sea is a version of this process with a bi-infinite collection of paths, hence the name. While it is not immediate how to define the bi-infinite version, it is however clear how to take $n=\infty$ and define a process $\Pois^{\nu,\infty}(T) = (\Pois^{\nu,\infty}_1(T),\Pois^{\nu,\infty}_2(T),\ldots)$ started at $\nu = (\nu_1,\nu_2,\ldots)$ with a top path and infinitely many paths below it, having jump rates $t,t^2,\ldots$. The sum of the jump rates is still finite and the reflection interactions \ref{item:interact} still make sense because there is always a `next available walker' at any location. This process shares with $\Pois^{\nu,n}$ the natural \emph{Markovian projection property} that for any $d \in \Z$, the truncated process 
\begin{equation}
F_d(\Pois^{\nu,n}(T)) := (\min(d,\Pois^{\nu,n}_1(T)),\ldots,\min(d,\Pois^{\nu,n}_n(T)))
\end{equation} 
is also a Markov process, because the paths at positions $\geq d$ do not influence the lower ones. 

Heuristically, $\Pois^{\mu,2\infty}(T)$ is the same process, but with a $\Z$-indexed collection of paths living in state space
\begin{equation}
\Sig_{2\infty} := \{(\la_n)_{n \in \Z} \in \Z^\Z: \la_{n+1} \leq \la_{n} \text{ for all }n \in \Z\},
\end{equation} 
started at initial condition $\mu = (\mu_i)_{i \in \Z} \in \Sig_{2\infty}$. The difficulty appears when making sense of the above dynamics for initial conditions such as $(\ldots,0,0,\ldots) \in \Sig_{2\infty}$ which have no top path at a given location, but by suitably doing so we show the following.

\begin{thm}\label{thm:sea_intro}
For any $\mu \in \Sig_{2\infty}$, there exists a Markov process $\Pois^{\mu,2\infty}(T), T \in \R_{\geq 0}$ on $\Sig_{2\infty}$ with initial condition $\Pois^{\mu,2\infty}(0) = \mu$, which satisfies the following Markovian projection property. For any $d \in \Z$, the process $F_d(\Pois^{\mu,2\infty}(T)) := (\min(d,\Pois^{\mu,2\infty}_i(T)))_{i \in \Z}$ is Markov. Furthermore, if $\mu_k \geq d$ for some $k$, then the process $(\min(d,\Pois^{\mu,2\infty}_{k+i}(T))_{i \in \Z_{\geq 1}}$, obtained by throwing away the coordinates which are equal to $d$ for all time from $F_d(\Pois^{\mu,2\infty}(T))$, is equal in multi-time distribution to $F_d(\Pois^{(\mu_{k+1},\mu_{k+2},\ldots),\infty}(t^k T))$. 
\end{thm}

The Markovian projection to $F_d(\Pois^{(\mu_{k+1},\mu_{k+2},\ldots),\infty}(t^k T))$ is extremely useful for proving convergence to the reflecting Poisson sea, as one may prove convergence of projections $F_d$ for each $d$, and these are easier objects to understand. Our proofs rely heavily on this.

\begin{rmk}
The Airy line ensemble \cite{corwin2014brownian}, bulk sine process \cite{tsai2016infinite}, the line ensemble of \cite{ahn2022extremal}, and the Bessel line ensemble \cite{wu2023bessel} were all constructed using certain continuous-time processes given by eigenvalues or singular values of Brownian motions on certain spaces of matrices. We similarly construct the reflecting Poisson sea from the continuous-time jump process $\Pois^{\nu,n}(T)$, which may be viewed as a $p$-adic analogue of these processes---see \cite{van2023p}---though in contrast to them it has only local interactions.
\end{rmk}

\begin{rmk}
The dynamics of $\Pois^{\nu,n}(T)$ and the reflecting Poisson sea are similar to some interacting particle systems in the literature, but not the same\footnote{Though we note that the $t \to 1$ asymptotics of $\Pois^{\nu,\infty}(T)$ were treated from an interacting particle system perspective in \cite{van2022q}.}. One of these is the particle system PushTASEP of \cite{borodin2008large}, which similarly features particles with independent clocks which can push one another. There, however, the particle whose clock rings always jumps (possibly also pushing others), while in our case the particle whose clock rings may donate its jump by \ref{item:interact} while not itself jumping. 

Another similar interacting particle system in the literature is the totally-asymmetric zero-range process (TAZRP) and other zero-range processes, see e.g. \cite{kipnis1998scaling} for a textbook treatment. The dynamics described above would be the same as TAZRP if our jump rates were identically $1$ rather than powers of $t$, but the fact that different paths/particles have different jump rates means that our process is not a zero-range process.

\end{rmk}

\subsection{First limit theorem} \label{subsec:limits} The first random matrix result we will state concerns matrices with iid entries coming from the additive Haar probability measure on $\Z_p$. This is a very natural measure, studied in the first work \cite{friedman-washington} on cokernels of $p$-adic random matrices. In particular, taking each entry modulo $p^k$ yields the uniform measure on the finite set $\Mat_N(\Z/p^k\Z)$.

Our limits will take place in the state space of \emph{extended} bi-infinite integer signatures
\begin{equation}
\bSig_{2\infty} := \{(\mu_n)_{n \in \Z} \in (\Z \cup \{\pm \infty\})^\Z: \mu_{n+1} \leq \mu_{n} \text{ for all }n \in \Z\}
\end{equation}
where we allow entries $\pm \infty$; it will soon be clear why $\pm \infty$ entries are technically convenient. We wish to speak of limits of the $N$-tuple of singular numbers $\SN(A) = (\SN(A)_1,\ldots,\SN(A)_N)$, which are integer signatures\footnote{Later we allow matrices which are not full rank by appropriately modifying the space of singular numbers, but we ignore this for now and consider only nonsingular matrices.} living in
\begin{equation}
\Sig_N := \{(\la_1,\ldots,\la_N) \in \Z^N: \la_1 \geq \ldots \geq \la_N\}.
\end{equation}
It is desirable to embed such signatures into $\bSig_{2\infty}$ so limits take place there, which we do by the map $\iota: \Sig_N \to \Sig_{2\infty}$ defined by
\begin{equation}
\iota((\la_1,\ldots,\la_N))_i := \begin{cases}
\infty & i \leq 0 \\ 
\la_i & 1 \leq i \leq N \\ 
-\infty & i > N
\end{cases}
\end{equation}
Because we wish to speak of limits of singular numbers $\SN(A_\tau \cdots A_1)_i$ for $i$ close to an observation point $r_N$, we define the shift map
\begin{align}
\begin{split}
s: \sSig_{2\infty} &\to \sSig_{2\infty} \\
(\mu_n)_{n \in \Z} &\mapsto (\mu_{n+1})_{n \in \Z}
\end{split}
\end{align}
Then $ s^{r_N} \circ \iota(\SN(A^{(N)}_{\tau} \cdots A^{(N)}_1)), \tau = 0,1,2,\ldots$ defines a discrete-time stochastic process on the subset of $\bSig_{2\infty}$ which only has nontrivial parts at indices between $1-r_N$ and $N-r_N$, and these endpoints go to $-\infty$ and $+\infty$ respectively under our hypotheses. There is substantial freedom in the result below regarding how the observation points $r_N$ go to $\infty$, though the number of matrix products ($\propto p^{r_N}$) we must take there to see a nontrivial limit depends on this choice.

\begin{thm}\label{thm:uniform_matrix_intro}
Fix $p$ prime. For each $N \in \Z_{\geq 1}$, let $A^{(N)}_i, i \geq 1$ be iid matrices with iid entries distributed by the additive Haar probability measure on $\Z_p$, and let $(r_N)_{N \geq 1}$ be any integer sequence such that $r_N \to \infty$ and $N-r_N \to \infty$ as $N \to \infty$. Define 
\begin{equation}
\Lambda^{(N)}(T) := s^{r_N} \circ \iota(\SN(A^{(N)}_{\floor{p^{r_N} T}} \cdots A^{(N)}_1)),T \in \R_{\geq 0}
\end{equation}
Then
\begin{equation}\label{eq:la_to_pois_intro_uniform}
\La^{(N)}(T) \to \Pois^{\mu,2\infty}(T)
\end{equation}
in finite-dimensional distribution\footnote{When we speak of convergence in distribution on $\bSig_{2\infty}$ here and elsewhere, we mean with respect to the topology defined in \Cref{subsec:measure_conv}, which is that of convergence in joint distribution of all finite collections of coordinates. We have made no effort to consider other notions of convergence than finite-dimensional distributions, though given that the paths are on $\Z$ and nondecreasing we suspect this can be upgraded if desired.}, where the parameter $t$ of $\Pois^{\mu,2\infty}$ is set to $1/p$ and $\mu = \bzero := (0)_{i \in \Z}$.
\end{thm}

It is worth noting that in the special case $\mu = \bzero$ of this theorem, the process $\Pois^{\mu,2\infty}(T)$ enjoys additional properties. One is the shift-stationarity property
\begin{equation}\label{eq:shift_stationarity_intro}
(\Pois^{\bzero,2\infty}_{i-1}(t^{-1}T))_{i \in \Z} = \Pois^{\bzero,2\infty}(T) \quad \quad \quad \quad \text{in multi-time distribution.}
\end{equation}
We prove this in \Cref{thm:time_symmetry}, but it is not hard to see heuristically: the jump rates of each path are in geometric progression, and shifting the indices while rescaling time leaves this geometric progression invariant. In addition, there are explicit formulas for its fixed-time marginals coming from \cite{van2023local} and given in \Cref{thm:sea_cL}, such as the following. 

\begin{example}\label{ex:explicit_formula}
When $\mu = (0)_{i \in \Z}$, the quantity $X:= \max\{i: \Pois^{\bzero,2\infty}_i(T) > 0\}$---the index of the lowest path at a position $>0$ at time $T$---has law given by
\begin{equation}
\Pr(X = n) = \frac{1}{\prod_{i \geq 1} (1-t^i)} \sum_{m \geq 0} e^{-\frac{T}{1-t} t^{n-m+1}} \frac{(-1)^m t^{\binom{m}{2}}}{\prod_{i=1}^m (1-t^i)}
\end{equation}
for any $n \in \Z$. At the level of this and other one-point marginals, the shift-stationarity is readily apparent, and was noted in \cite[(1.20)]{van2023local}. 
\end{example}
In particular this formula shows that indeed $\Pois^{\bzero,2\infty}$ is not the constant $\bzero$.

\subsection{Universal bulk limit} Another natural probability measure on $\Mat_N(\Z_p)$ is the Haar probability measure on $\GL_N(\Z_p)$. The singular numbers of this measure are all $0$, but if one instead considers an $N \times N$ corner of a Haar-distributed element of $\GL_{N+D}(\Z_p)$, the singular numbers are nontrivial, and their distribution is different from the above iid Haar case. We prove in \Cref{thm:haar_corner_intro} later that essentially the same limit as in \Cref{thm:uniform_matrix_intro} holds also for these ensembles, after suitably adjusting the time-change $p^{r_N}$. 

Both the additive Haar measure of \Cref{thm:uniform_matrix_intro} and the measure of \Cref{thm:haar_corner_intro} have the very useful property that they are invariant under multiplication by $\GL_N(\Z_p)$. Our next result shows that very little beyond this property is necessary to obtain bulk convergence to the reflecting Poisson sea in a dynamical sense.

\begin{thm}\label{thm:universality_intro}
Let $t=1/p$ and $\mu \in \bSig_{2 \infty}$ be any signature with all parts nonnegative and $\mu_{-n} \rightarrow \infty$ as $n \rightarrow \infty$. For each $N \in \mathbb{N}$, let $A^{(N)}$ be a random matrix in $\operatorname{Mat}_{N}\left(\mathbb{Z}_{p}\right)$ with distribution invariant under left-multiplication by $\GL_N(\Z_p)$, and let $r_{N}$ be a `bulk observation point,' such that
\begin{enumerate}[label=(\roman*)]
\item The singular numbers are nontrivial: $\operatorname{Pr}\left(A^{(N)} \in \mathrm{GL}_{N}\left(\mathbb{Z}_{p}\right)\right)<1$ for every $N$, \label{item:univ_intro_i}
\item $r_{N} \rightarrow \infty$ and $N-r_{N} \rightarrow \infty$ as $N \rightarrow \infty$, and \label{item:univ_intro_ii}
\item The coranks $X_{N}:=\operatorname{corank}\left(A^{(N)}\pmod{p}\right)$ have exponential tails decaying faster than powers of $t$, i.e., there exists $C>0$ such that \label{item:univ_intro_iii}
\begin{equation}\label{eq:X_N_hyp_intro}
\operatorname{Pr}\left(X_{N} \geq k | X_{N}>0\right) < C t^{1.0001 k} \quad \quad \text{ for all $k>0$.}
\end{equation}
\end{enumerate}
Let $A_i^{(N)},i \geq 1$ be iid copies of $A^{(N)}$, and let $B^{(N)} \in \operatorname{Mat}_{N}\left(\mathbb{Z}_{p}\right), N \geq 1$ be left-$\mathrm{GL}_{N}\left(\mathbb{Z}_{p}\right)$-invariant `initial condition' matrices with fixed singular numbers prescribed by $\mu$ via
\begin{equation}
\mathrm{SN}\left(B^{(N)}\right)_{i}=\mu_{i-r_{N}}
\end{equation}
for all $1 \leq i \leq N$, and define the matrix product process with initial condition
\begin{equation}\label{eq:def_pi_n_intro}
\Pi^{(N)}(\tau):=\operatorname{SN}\left(A_{\tau}^{(N)} \cdots A_{1}^{(N)} B^{(N)}\right), \tau \in \mathbb{Z}_{\geq 0}.
\end{equation}
Define the time-scaling
\begin{equation}\label{eq:def_cN_intro}
c_{N}:=\frac{t^{-r_{N}}}{\mathbb{E}\left[\bbone(X_{N} \leq r_{N})\left(t^{-X_{N}}-1\right)\right]},
\end{equation}
and let
\begin{equation}
    \Lambda^{(N)}(T) := s^{r_N} \circ \iota(\Pi^{(N)}(\floor{c_N T})),T \in \R_{\geq 0}.
\end{equation}
Then we have convergence
    \begin{equation}\label{eq:intro_bulk_conv}
        \Lambda^{(N)}(T) \xrightarrow{N \to \infty} \Pois^{\mu,2\infty}(T)
    \end{equation}
    in finite-dimensional distribution.
\end{thm} 

The conditions of \Cref{thm:universality_intro} are easy to check in explicit cases such as iid additive Haar matrices and Haar $\GL_{N+D}(\Z_p)$ corners, as we do in \Cref{sec:deducing}. However, they apply to many others, including singular matrices. Matrices $A_i^{(N)}$ with first column $0$ and all other entries distributed iid from the additive Haar measure, for instance, are perfectly valid. For singular matrices, Smith normal form still gives a decomposition $A=U \diag_{N \times N}(p^{\SN(A)_1},\ldots,p^{\SN(A)_N})V$ where the $\SN(A)_i$ are allowed to be equal to $+\infty$, and one takes $p^\infty = 0$ since $p^n \to 0$ in the $p$-adic norm as $n\to\infty$. The initial condition $\mu$ may similarly have entries equal to $\infty$, making $B^{(N)}$ singular. One may also simply take matrices with deterministic singular numbers, $A_i^{(N)} = U^{(i)}_N D_N$ where $U^{(i)}_N \in \GL_N(\Z_p)$ are iid Haar distributed and $D_N$ is some deterministic matrix with $\corank(D_N \pmod{p})$ satisfying the tail decay condition.

The only reason \Cref{thm:uniform_matrix_intro} is not a special case of \Cref{thm:universality_intro} is the condition $\mu_{-n} \rightarrow \infty$ as $n \rightarrow \infty$, which precludes initial conditions such as the zero signature (in which case the law of $B^{(N)}$ is the Haar measure on the group $\GL_N(\Z_p)$). However, \Cref{thm:universality_intro} does still apply to singular numbers of $A^{(N)}_\tau \cdots A^{(N)}_1$ with no initial condition $B^{(N)}$, in the following sense: one may condition on the matrix $B^{(N)} := A^{(N)}_{s} \cdots A^{(N)}_1$, and apply the theorem to $A^{(N)}_{s+\tau} \cdots A^{(N)}_{s+1}B^{(N)}$. This is why we refer to it as dynamical universality: it shows universality of the dynamics, but not of the fixed-time distribution.

In the additive Haar case of \Cref{thm:uniform_matrix_intro}, we nonetheless were able to prove convergence without any initial condition $B^{(N)}$. This is because the `single-time' bulk limit of $A^{(N)}_{s} \cdots A^{(N)}_1$ is known by integrable probability methods \cite{van2023local}, and the above idea lets us bootstrap to multiple times using \Cref{thm:universality_intro}. We believe that the general version \Cref{thm:universality_intro} is true without the restriction $\mu_{-n} \rightarrow \infty$, but establishing the necessary single-time input in this generality seems quite nontrivial.




\subsection{The edge} In classical random matrix theory, local limits of eigenvalues or singular values far away from the smallest and largest one are usually referred to as \emph{bulk limits}. The \emph{edge limits} close to the smallest or largest are different objects, with different scalings in the limit theorems. At the level convergence of line ensembles for matrix products, a limit of this type for complex matrix products was shown recently in \cite{ahn2022extremal}. In our setting, we find in \Cref{thm:rmt_edge_intro} that essentially the same result as in the bulk holds at the lower edge: the joint evolution of $(\ldots,\SN(A_\tau \cdots A_1)_{N-1},\SN(A_\tau \cdots A_1)_{N})$ converges to a version of $\Pois^{\mu,2\infty}(T)$ which has a lowest path but the same local dynamics, again with no rescaling of the singular numbers. The theorem and proof are essentially identical to \Cref{thm:universality_intro}, and indeed we prove both simultaneously via the general result \Cref{thm:gen_limit} later. From the perspective of classical random matrix theory this similarity between the bulk and edge is quite surprising, but here it arises naturally from our proofs.

The limit process $\Pois^{\mu,edge}(T)$ for the edge version lives on
\begin{equation}\label{eq:sig_edge}
\bSig_{\text {edge }}:=\left\{\left(\mu_{n}\right)_{n \in \mathbb{Z}_{\leq 0}} \in (\Z \cup \{\pm\infty\})^{\mathbb{Z}_{\leq 0}}: \mu_{n+1} \leq \mu_{n} \text { for all } n \in \mathbb{Z}_{<0}\right\}
\end{equation}
because there is a smallest singular number $\operatorname{SN}\left(A_{\tau}^{(N)} \cdots A_{1}^{(N)}\right)_{N}$. The limit object, $\Pois^{\mu, e d g e}(T)$, is constructed the same way as $\Pois^{\mu, 2 \infty}(T)$ and has the same local Poisson/reflection dynamics, see \Cref{def:pois_edge} for details. The result we will now state is exactly the same as \Cref{thm:universality_intro}, except $r_{N}$, $\operatorname{Sig}_{2 \infty}$, and $\Pois^{\mu,2\infty}$ are replaced by $N$, $\operatorname{Sig}_{\text {edge }}$, and $\Pois^{\mu,edge}$.

\begin{thm}\label{thm:rmt_edge_intro}
Let $t=1/p$ and let $\mu \in \bSig_{\text {edge }}$ have $\mu_{0} \geq 0$ and $\mu_{-n} \rightarrow \infty$ as $n \rightarrow \infty$. For each $N \in \mathbb{N}$, let $A^{(N)}$ be a random matrix in $\operatorname{Mat}_{N}\left(\mathbb{Z}_{p}\right)$ with distribution invariant under left-multiplication by $\GL_N(\Z_p)$ and satisfying conditions (i) and (iii) of \Cref{thm:universality_intro}. Let $A_i^{(N)},i \geq 1$ be iid copies of $A^{(N)}$, and let $B^{(N)} \in \operatorname{Mat}_{N}\left(\mathbb{Z}_{p}\right), N \geq 1$ be left-$\mathrm{GL}_{N}\left(\mathbb{Z}_{p}\right)$-invariant `initial condition' matrices with fixed singular numbers
\begin{equation}
\mathrm{SN}\left(B^{(N)}\right)_{i}=\mu_{i-N}
\end{equation}
for all $1 \leq i \leq N$. Let $\Lambda^{(N)}(T)$ be as in \Cref{thm:universality_intro} with $r_N = N$ throughout. Then 
    \begin{equation}\label{eq:intro_edge_conv}
        \Lambda^{(N)}(T) \xrightarrow{N \to \infty} \Pois^{\mu,edge}(T)
    \end{equation}
    in finite-dimensional distribution.
\end{thm}

Hints of \Cref{thm:rmt_edge_intro} appeared in \cite[Theorem 1.2]{van2020limits}, where it was shown that the $\tau \to \infty$ law of large numbers of the smallest singular numbers converged as $N \to \infty$ to a geometric progression with common ratio $p$, which corresponds to the geometric progression of jump rates in the reflecting Poisson sea. \Cref{thm:rmt_edge_intro} applies much more broadly and gives more detailed information, however. Note also that we do not have a result like \Cref{thm:uniform_matrix_intro} at the edge, because the input of \cite{van2023local} was only shown in the bulk.

\begin{rmk}
Many previous works \cite{wood2017distribution,wood2018cohen,wood2015random,meszaros2020distribution,cheong2021cohen,cheong2022generalizations,cheong2023cokernel} prove a different form of universality for singular numbers/cokernels, namely that for a single $N \times N$ matrix with iid entries (or a product of finitely many such matrices \cite{nguyen2022universality}), the large $N$ limit distribution of singular numbers is universal for many choices of entry distribution and in particular agrees with the additive Haar case. In \Cref{thm:universality_intro} and \Cref{thm:rmt_edge_intro}, by contrast, we do not require that the distribution of singular numbers of a single matrix $A_i^{(N)}$ converges. The mixing of many matrices in the product still manages to average out the nonuniversal behavior of each individual matrix, yielding a form of universality for products which holds even in settings where universality for a single matrix breaks down.
\end{rmk}

\begin{rmk}
Though we have stated our results over $\Z_p$ for simplicity, one may obtain results over $\Q_p$ with only marginally more effort. We also suspect that the results may be proven over any other ring of integers of a non-Archimedean local field with finite residue field via the same methods.
\end{rmk}

\subsection{Outline of the proofs and rest of the paper}\label{subsec:methods}

\Cref{sec:p-adic} gives some basic background on $p$-adic numbers and matrices over them, as well as a less-standard variational characterization of singular numbers and some useful consequences. \Cref{sec:constructing} constructs $\Pois^{\mu,2\infty}(T)$ and proves some useful properties of it. The basic idea of the construction is to couple many processes $\Pois^{\nu,\infty}(T)$ together on the same probability space, which corresponds to `histories of clock ring times for $\Pois^{\mu,2\infty}(T)$ for all time,' and take a suitable limit.

Apart from \Cref{sec:deducing}, the rest of the text is devoted to proving \Cref{thm:gen_limit}, which combines the similar results of \Cref{thm:universality_intro} and \Cref{thm:rmt_edge_intro} into one statement so they can be proven simultaneously. That result also assumes weaker hypotheses than the ones assumed \Cref{thm:universality_intro}, but they are somewhat less transparent to state, so we have given the above hypotheses for the sake of exposition. 

To prove \Cref{thm:gen_limit}, we show convergence of the truncated processes $F_d(\La^{(N)}(T))$ to $F_d(\Pois^{\mu,2\infty}(T))$, which are both still Markov; on the matrix side this simply corresponds to taking all matrices modulo $p^d$. The process $\Pois^{\mu,2\infty}(T)$ is a complicated object, with infinitely many jumps occurring on any time interval, but this truncation simplifies it. Namely, \Cref{thm:sea_intro} guarantees that \emph{provided that $\mu$ has a part at least $d$}, such a truncation is given by a process $\Pois^{\nu,\infty}(t^k T)$, which has only finitely many jumps on any time interval and is determined by its Markov generator (which is explicit, see \Cref{thm:compute_Q}). This reduction is why the hypothesis $\lim_{n \to \infty} \mu_{-n} = \infty$ is so necessary. In \Cref{sec:markov_preliminaries}, we show in this manner that the proof reduces to showing that the Markov generator of the matrix product process converges to that of an appropriate process $\Pois^{\nu,\infty}(t^k T)$.

We show this convergence of Markov generators in \Cref{sec:transition_asymptotics}, using explicit nonasymptotic bounds on the transition probabilities of singular numbers which we establish in \Cref{sec:nonasymp_lin_alg}. These crucially use the $\GL_N(\Z_p)$-invariance of the matrices $A_i^{(N)}$, as it yields that 
\begin{equation}
\SN(A^{(N)}_{\tau+1} \cdots A_1^{(N)}) = \SN(\diag_{N \times N}(p^{\SN(A^{(N)}_{\tau+1})}) U \diag_{N \times N}(p^{\SN(A^{(N)}_{\tau} \cdots A_1^{(N)})}))
\end{equation}
in distribution, where $U \in \GL_N(\Z_p)$ is Haar-distributed and independent of the singular numbers. Hence it suffices to understand the singular numbers of matrices $\diag_{N \times N}(p^\la)U\diag_{N \times N}(p^\mu)$ for fixed $\la,\mu \in \Sig_N$ and $U$ Haar-distributed, which we do by explicit linear-algebraic manipulations using the explicit characterization of the Haar measure and variational characterization of the singular numbers. The reason we are able to establish \Cref{thm:universality_intro} and \Cref{thm:rmt_edge_intro} in the generality that we do is that these bounds are very robust and use no special structure beyond the $\GL_N(\Z_p)$-invariance of the matrix distributions involved. 

Once \Cref{thm:universality_intro} is established, we show \Cref{thm:uniform_matrix_intro}, and a related result \Cref{thm:haar_corner_intro} for $\GL_{N+D}(\Z_p)$-corners, in \Cref{sec:deducing}. As mentioned, these would follow directly from \Cref{thm:universality_intro} if the hypothesis $\lim_{n \to \infty} \mu_{-n}$ on the initial condition were not present. In particular, this means that if one can establish the limit \eqref{eq:la_to_pois_intro_uniform} of \Cref{thm:uniform_matrix_intro} at a single time $T_1$, one can bootstrap the single-time limit for $A_{\floor{p^{r_N}T_1}}^{(N)} \cdots A_1^{(N)}$ to a multi-time limit by conditioning on the matrix $A_{\floor{p^{r_N}T_1}}^{(N)} \cdots A_1^{(N)}$, and then applying \Cref{thm:universality_intro} with $A_{\floor{p^{r_N}T_1}}^{(N)} \cdots A_1^{(N)}$ playing the role of $B^{(N)}$. This yields asymptotics for the subsequent evolution of the singular numbers at time $T>T_1$. 


To show the convergence at $T_1$ we use results from the related work \cite{van2023local}, which rely on asymptotic analysis of Hall-Littlewood processes and hold for a very special class of examples. These gave explicit formulas similar to \Cref{ex:explicit_formula} for the random matrix limit, as well as for limits of $\Pois^{\nu,\infty}(T)$ and hence for $\Pois^{\mu,2\infty}(T)$ via our coupling construction, see \Cref{thm:use_other_paper}. It does not matter what these formulas are, and they may be treated as a black box. All that matters here is that the formulas agree for $\Pois^{\mu,2\infty}(T)$ and for random matrices, which establishes the single-time limit.

We find it interesting that the only way we know to arrive at \Cref{thm:uniform_matrix_intro} uses entirely disjoint techniques for the single-time limit and the bootstrap to multiple times. For Theorems \ref{thm:universality_intro} and \ref{thm:rmt_edge_intro}, we find that the matrix product computations in the proof give a satisfying conceptual reason as to why the result is true. The proofs of single-time convergence and limit formulas in \cite{van2023local}, however, are still quite mysterious to us. We hope that future efforts can bring the ideas of this work to bear on single-time convergence, both to illuminate the formulas in \cite{van2023local} and to enlarge the scope of their universality class to different matrix distributions. It is also worth noting that while we do not use the techniques of \cite{van2020limits,van2023local} in the proofs of \Cref{thm:universality_intro} and \Cref{thm:rmt_edge_intro}, we would never have guessed such results without the explicit examples those techniques afforded.

\textbf{Remark on notation.} In formulas for probabilities in terms of the prime $p$ below, we typically instead use the variable $t=1/p$, to match the parameter in $\Pois^{\mu,2\infty}$ (which does not have to be the inverse of a prime, in general). We have used $t$ for this parameter to be consistent with notation for the related Hall-Littlewood polynomials, see \cite[Proposition 3.13]{van2023local} for the relation.

\section{Preliminaries on $p$-adic random matrices and singular numbers}\label{sec:p-adic}

The following is a condensed version of the exposition in \cite[Section 2]{evans2002elementary}, to which we refer any reader desiring a more detailed introduction to $p$-adic numbers geared toward a probabilistic viewpoint. Fix a prime $p$. Any nonzero rational number $r \in \Q^\times$ may be written as $r=p^k (a/b)$ with $k \in \Z$ and $a,b$ coprime to $p$. Define $|\cdot|: \Q \to \R$ by setting $|r|_p = p^{-k}$ for $r$ as before, and $|0|_p=0$. Then $|\cdot|_p$ defines a norm on $\Q$ and $d_p(x,y) :=|x-y|_p$ defines a metric. We additionally define $\val_p(r)=k$ for $r$ as above and $\val_p(0) = \infty$, so $|r|_p = p^{-\val_p(r)}$. We define the \emph{field of $p$-adic numbers} $\Q_p$ to be the completion of $\Q$ with respect to this metric, and the \emph{$p$-adic integers} $\Z_p$ to be the unit ball $\{x \in \Q_p : |x|_p \leq 1\}$. It is not hard to check that $\Z_p$ is a subring of $\Q_p$. We remark that $\Z_p$ may be alternatively defined as the inverse limit of the system $\ldots \to \Z/p^{n+1}\Z \to \Z/p^n \Z \to \cdots \to \Z/p\Z \to 0$, and that $\Z$ naturally includes into $\Z_p$. 

$\Q_p$ is noncompact but is equipped with a left- and right-invariant (additive) Haar measure; this measure is unique if we normalize so that the compact subgroup $\Z_p$ has measure $1$. The restriction of this measure to $\Z_p$ is the unique Haar probability measure on $\Z_p$, and is explicitly characterized by the fact that its pushforward under any map $r_n:\Z_p \to \Z/p^n\Z$ is the uniform probability measure. For concreteness, it is often useful to view elements of $\Z_p$ as `power series in $p$' $a_0 + a_1 p + a_2 p^2 + \ldots$, with $a_i \in \{0,\ldots,p-1\}$; clearly these specify a coherent sequence of elements of $\Z/p^n\Z$ for each $n$. The Haar probability measure then has the alternate explicit description that each $a_i$ is iid uniformly random from $\{0,\ldots,p-1\}$. Additionally, $\Q_p$ is isomorphic to the ring of Laurent series in $p$, defined in exactly the same way.

Similarly, $\GL_N(\Q_p)$ has a unique left- and right-invariant measure for which the total mass of the compact subgroup $\GL_N(\Z_p)$ is $1$. We denote this measure by $\M$. The restriction of $\M$ to $\GL_N(\Z_p)$ pushes forward to $\GL_N(\Z/p^n\Z)$; these measures are the uniform measures on the finite groups $\GL_N(\Z/p^n\Z)$. This gives an alternative characterization of the measure.

The following standard result is sometimes known as Smith normal form and holds also for more general rings.

\begin{prop}\label{thm:smith}
Let $n \leq m$. For any $A \in M_{n \times m}(\Q_p)$, there exist $U \in \GL_n(\Z_p), V \in \GL_m(\Z_p)$ such that $UAV = \diag_{n \times m}(p^{\l_1},\ldots,p^{\l_n})$ where $\l$ is a weakly decreasing $n$-tuple of integers when $A$ is nonsingular, when $A$ is singular we formally allow parts of $\l$ to equal $\infty$ and define $p^\infty = 0$. Furthermore, there is a unique such $n$-tuple $\l$.
\end{prop}

\begin{defi}\label{def:finite_sigs}
We denote the tuple $\la$ of \Cref{thm:smith} by $\SN(A)$, and refer to its elements $\la_1,\ldots,\la_n$ as the \emph{singular numbers} of $A$. We call such tuples (extended) integer signatures, and write the set of such signatures as follows:
\begin{equation}
\bSig_n := \{(\la_1,\ldots,\la_n) \in (\Z \cup \{\infty\})^n: \la_1 \geq \ldots \geq \la_n\}.
\end{equation}
Given $\la \in \bSig_n$, we set $|\la| := \la_1+\ldots+\la_n$, interpreting this as $\infty$ if any of the parts are infinite. We additionally write $m_i(\la) = \#\{j: \la_j = i\}$.
\end{defi}

Similarly to eigenvalues and singular values, singular numbers have a variational characterization. We first recall the version for singular values, one version of which states that for $A \in \Mat_{n \times m}(\C)$ (assume without loss of generality $n \leq m$) with singular values $a_1 \geq \ldots \geq a_n$,
\begin{equation}\label{eq:sv_minmax}
\prod_{i=1}^k a_i = \sup_{\substack{V \subset \C^m: \dim(V) = k \\ W \subset \C^n: \dim(W) = k}} |\det(\Proj_W \circ A|_V)|
\end{equation}
where $\Proj$ is the orthogonal projection and $A|_V$ is the restriction of the linear operator $A$ to the subspace $V$. Equation \eqref{eq:sv_minmax} holds because the right hand side is unchanged by multiplying $A$ by unitary matrices, hence $A$ may be taken to be diagonal with singular values on the diagonal by singular value decomposition, at which point the result is easy to see. For a slightly different version which picks out the $k\tth$ largest singular value rather than the product of the $k$ largest, see \cite[Section 5]{fulton2000eigenvalues}.

For $p$-adic matrices, we state the result slightly differently to avoid referring to orthogonal projection, the reason being that unlike $U(n)$, $\GL_n(\Z_p)$ does not preserve a reasonable inner product, only the norm.

\begin{prop}\label{thm:minmax}
Let $1 \leq n \leq m$ be integers and $A \in \Mat_{n \times m}(\Q_p)$ with $\SN(A) = (\la_1,\ldots,\la_n)$. Then for any $1 \leq k \leq n$,
\begin{align}\label{eq:raleigh}
\la_n+\ldots+\la_{n-k+1} &= \inf_{\substack{P:\Q_p^n \to \Q_p^n \text{ rank $k$ projection}\\ W \subset \Q_p^m: \dim W = k}} \val_p(\det(PA|_W)) 
\end{align}
\end{prop}
\begin{proof}
If $U_1 \in \GL_n(\Z_p),U_2 \in \GL_m(\Z_p)$, then for any a rank $k$ projection $P$ the matrix $U_1PU_1^{-1}$ is also a rank $k$ projection, and similarly for any $W$ as above $U_2 W$ is also a dimension $k$ subspace. Hence
\begin{equation}
\inf_{\substack{P:\Q_p^n \to \Q_p^n \text{ rank $k$ projection}\\ W \subset \Q_p^m: \dim W = k}} \val_p(\det(PA|_W)) = \inf_{\substack{P:\Q_p^n \to \Q_p^n \text{ rank $k$ projection}\\ W \subset \Q_p^m: \dim W = k}} \val_p(\det(P(U_1AU_2)|_W)).
\end{equation}
By Smith normal form we may choose $U_1,U_2$ so that $U_1AU_2 = \diag_{n \times m}(p^{\la_1},\ldots,p^{\la_n})$, hence
\begin{equation}
\text{RHS\eqref{eq:raleigh}} = \inf_{\substack{P:\Q_p^n \to \Q_p^n \text{ rank $k$ projection}\\ W \subset \Q_p^m: \dim W = k}} \val_p(\det(P \diag_{n \times m}(p^{\la_1},\ldots,p^{\la_n})|_W)).
\end{equation}
The infimum on the right hand side is clearly achieved by taking $W = \Span(\vec{e_{n-k+1}},\ldots,\vec{e_n})$ (where $\vec{e_i}$ are the standard basis vectors) and $P$ to be the projection onto $\Span(\vec{e_{n-k+1}},\ldots,\vec{e_n})$. This proves \eqref{eq:raleigh}. 
\end{proof}

We record a few corollaries of \Cref{thm:minmax} which will be useful later.

\begin{cor}
\label{thm:multiply_smaller_dimension}
Let $n \leq m$, $A \in \Mat_{n \times m}(\Q_p)$, and $\kappa \in \Sig_m$. Then 
\begin{equation}
|\SN(\diag_{m \times m}(p^{\kappa_1},\ldots,p^{\kappa_n}) A)| = |\SN(A)| + |\kappa|.
\end{equation}
\end{cor}
\begin{proof}
Follows immediately from \Cref{thm:minmax} with $k=n$ and multiplicativity of the determinant.
\end{proof}

\begin{cor}
\label{thm:minor_increase_sns}
If $d \leq m$ and $\ell \leq n$ are nonnegative integers, $A \in \Mat_{m \times n}(\Q_p)$, and $B$ is any $d \times \ell$ submatrix of $A$, then the $j$ smallest singular numbers satisfy
\begin{equation}
    \label{eq:minor_increase_sns}
    \sum_{j=1}^k \SN(B)_{\min(d,\ell)-j+1} \geq \sum_{j=1}^k \SN(A)_{\min(m,n) - j+1}
\end{equation}
for any $1 \leq k \leq \min(d,\ell)$.
\end{cor}
\begin{proof}
By \Cref{thm:minmax} both sides of \eqref{eq:minor_increase_sns} may be expressed as an infimum, and the left hand side is an infimum over a smaller set.
\end{proof}

We will often write $\diag_{n \times N}(p^\l)$ for $\diag_{n \times N}(p^{\l_1},\ldots,p^{\l_n})$, and also omit the dimensions $n \times N$ when they are clear from context. We note also that for any $\l \in \Sig_N$, the double coset $\GL_N(\Z_p) \diag(p^\l) \GL_N(\Z_p)$ is compact. The restriction of the additive Haar measure $\M$ to such a double coset, normalized to be a probability measure, is the unique $\GL_N(\Z_p) \times \GL_N(\Z_p)$-invariant probability measure on $\GL_N(\Q_p)$ with singular numbers $\l$, and all $\GL_N(\Z_p) \times \GL_N(\Z_p)$-probability measures are convex combinations of these for different $\l$. These measures may be equivalently described as $U \diag(p^{\l_1},\ldots,p^{\l_N}) V$ where $U,V$ are independently distributed by the Haar probability measure on $\GL_N(\Z_p)$. More generally, if $n \leq m$ and $U \in \GL_n(\Z_p), V \in \GL_m(\Z_p)$ are Haar distributed and $\mu \in \bSig_n$, then $U \diag_{n \times m}(p^\mu) V$ is invariant under $\GL_n(\Z_p) \times \GL_m(\Z_p)$ acting on the left and right, and is the unique such bi-invariant measure with singular numbers given by $\mu$. 

The Haar measure on $\GL_N(\Z_p)$ also has an explicit characterization which is well-known and will be very useful in \Cref{sec:nonasymp_lin_alg}.

\begin{prop}[{\cite[Proposition 2.1]{van2023p}}]\label{thm:haar_sampling}
Let 
\begin{equation}
A \in \Mat_N(\Z_p)
\end{equation}
be a random matrix with distribution given as follows: sample its columns $v_N,v_{N-1},\ldots,v_1$ from right to left, where the conditional distribution of $v_i$ given $v_{i+1},\ldots,v_N$ is that of a random column vector with additive Haar distribution conditioned on the event 
\begin{equation}\label{eq:notinspan}
v_i \pmod{p} \not \in \Span(v_{i+1} \pmod{p},\ldots,v_N \pmod{p}) \subset \F_p^N,
\end{equation}
where in the case $i=N$ we take the span in \eqref{eq:notinspan} to be the $0$ subspace. Then $A$ is distributed by the Haar measure on $\GL_N(\Z_p)$.
 \end{prop}

\section{Constructing $\Pois^{\mu,2\infty}$}\label{sec:constructing}

In this section we construct the bulk and edge limit processes mentioned in the Introduction, by coupling together many copies of the process $\Pois^{\mu,\infty}(T)$ discussed previously. We will give a uniform construction with general initial condition, and to set up this formalism we define an extended version of earlier signature notation. Throughout this section, $t \in (0,1)$ is a fixed real parameter.

Our goal is to speak of limits of the tuple of singular numbers of a matrix, which is a finite signature, to bi-infinite signatures. A reasonable way to formalize this is to embed all sets $\bSig_n$ into the set of bi-infinite signatures. It is technically convenient to allow these to include $-\infty$ entries, as we will see later, even though it does not make sense to have $-\infty$ as a singular number.

\begin{defi}\label{def:bi-infinite_sig}
Let $\bZ = \Z \cup \{\pm \infty\}$, and define
\begin{equation}\Sig_{\infty} := \{(\mu_n)_{n \in \Z_{\geq 1}} \in \Z^{\Z_{\geq 1}}: \mu_{n+1} \leq \mu_{n} \text{ for all }n \in \Z_{\geq 1}\}
\end{equation}
and
\begin{equation}\bSig_{\infty} := \{(\mu_n)_{n \in \Z} \in \bZ^{\Z_{\geq 1}}: \mu_{n+1} \leq \mu_{n} \text{ for all }n \in \Z_{\geq 1}\},
\end{equation}
and the bi-infinite versions
\begin{equation}\Sig_{2\infty} := \{(\mu_n)_{n \in \Z} \in \Z^\Z: \mu_{n+1} \leq \mu_{n} \text{ for all }n \in \Z\}
\end{equation}
and
\begin{equation}\bSig_{2\infty} := \{(\mu_n)_{n \in \Z} \in \bZ^\Z: \mu_{n+1} \leq \mu_{n} \text{ for all }n \in \Z\}.
\end{equation}
For $x \in \bZ$, we write $(x[2\infty]) = (x)_{n \in \Z} \in \bSig_{2\infty}$. We refer to the elements $\la_n,\mu_n$ above as \emph{parts}, as is standard terminology with integer partitions, and use notation $m_i(\la)$ as in \Cref{def:finite_sigs}. Finally, we use $\Sig_n^+,\sSig_n^+,\Sig_{2\infty}^+,\sSig_{2\infty}^+$ to denote the subsets where all parts lie in $\Z_{\geq 0} \cup \{\infty\}$.
\end{defi}

\begin{defi}\label{def:pi}
    For any finite set $I \subset \Z$, define
\begin{align*}
\pi_I: \sSig_{2\infty} & \to \bZ^I \\ 
\mu & \mapsto (\mu_i)_{i \in I}.
\end{align*}
When it is convenient to do so we will often abuse notation and identify the image of $\pi_I$ in $\bZ^I$, which is just weakly decreasing $n$-tuples, with $\bSig_{|I|}$. Similarly, for a half-infinite interval $I = [a,\infty)$ we define $\pi_I: \sSig_{2\infty} \to \bZ^I$ in the same way, and often abuse notation and identify the image in $\bZ^I$ with $\bSig_\infty$. 
\end{defi}

\begin{defi}
    \label{def:conjugate_parts}
    Given $\mu = (\mu_n)_{n \in \Z} \in \sSig_{2\infty}$, we define $\mu' = (\mu_n')_{n \in \Z} \in \sSig_{2\infty}$ by
    \begin{equation}
\mu_i' = \begin{cases}
    \text{the unique index $j$ such that $\mu_j \geq i, \mu_{j+1} < i$} & \text{ if }\lim_{n \to \infty} \mu_{-n} \geq i > \lim_{n \to \infty} \mu_n \\
    -\infty & \text{ if }i > \lim_{n \to \infty} \mu_{-n} \\
    \infty & \text{ if }i \leq \lim_{n \to \infty} \mu_n
\end{cases}
    \end{equation}
\end{defi}


\begin{defi}
    \label{def:poisson_walks}
    For a fixed parameter $t \in (0,1)$, length $n \in \N \cup \{\infty\}$, and initial condition $\nu \in \Sig_{n}$, we define the stochastic process $\Pois^{\nu,n}(T) = (\Pois^{\nu,n}_1(T),\ldots,\Pois^{\nu,n}_n(T))$ on $\Sig_n$ as follows. For each $1 \leq i \leq n$, $\Pois^{\nu,n}_i$ has an exponential clock of rate $t^i$, and when $\Pois^{\nu,n}_i$'s clock rings, $\Pois^{\nu,n}_i$ increases by $1$ if the resulting $n$-tuple is still weakly decreasing. If not, then $\Pois^{\nu,n}_{i-d}$ increases by $1$ instead and $\Pois^{\nu,n}_i$ remains the same, where $d \geq 0$ is the smallest index so that the resulting tuple is weakly decreasing. In the case of trivial initial condition we will often write $\Pois^{(n)}$ for $\Pois^{(0[n]),n}$. 
\end{defi}

Strictly speaking, the description in \Cref{def:poisson_walks} only makes sense if the set of clock ringing times is discrete. This is simple to show, and we do so in \Cref{thm:tomega_full_measure} once we have set up the relevant probability space. To couple many processes $\Pois^{\mu,\infty}(T)$ together, it is helpful to define notation for certain shifted versions. 

\begin{defi}\label{def:half_infinite_process}
For $\mu \in \sSig_{\infty}$ and $t \in (0,1)$, we define the stochastic process
\begin{equation}\label{eq:def_tpois}
\tPois^{\mu,n}(T) = (\tPois^{\mu,n}_{-n}(T),\tPois^{\mu,n}_{-n+1}(T),\ldots) = (\Pois_1^{\mu,\infty}(t^{-n-1}T),\Pois_2^{\mu,\infty}(t^{-n-1}T),\ldots).
\end{equation}
\end{defi}

We also emphasize that $\tPois^{\mu,n}(T)$ is merely a notational shift of $\Pois^{\mu,\infty}(T)$ as defined in \Cref{def:poisson_walks}, where we make the indices start at $-n$ rather than $1$, and speed up time by a factor of $t^{-n-1}$ so that $\tPois_1^{\mu,n}(T)$ has jump rate $t$, similarly to $\Pois^{\mu,n}_1(T)$ and $\Pois^{\mu,\infty}_1(T)$. 

\begin{defi}\label{def:omega}
Define the probability space 
\begin{equation}
\Omega := \prod_{i \in \Z} \R_{\geq 0}^\N 
\end{equation}
with the product Borel $\sigma$-algebra. Define the probability measure 
\begin{equation}
\Poiss := \prod_{i \in \Z} \Poiss_{t^i} \in \cM(\Omega)
\end{equation}
where $\Poiss_r \in \cM(\R_{\geq 0}^\N)$ is the product over the $\N$ factors of the distributions of rate-$r$ exponential variables.
\end{defi}

Clearly $\Poiss_r$ may be identified with the law of a rate $r$ Poisson jump process on time $T \geq 0$ by viewing each $\R_{\geq 0}$ factor as specifying the waiting time between adjacent jumps (or in the case of the first factor, the waiting time between time $T=0$ and the first jump). Heuristically, $\Pois^{\mu,2\infty}(T)$ is defined by giving each $\Pois_i^{\mu,2\infty}(T)$ an independent exponential clock with rate $t^i$, and having $\Pois_i^{\mu,2\infty}(T)$ jump when its clock rings; here, $\Omega$ is exactly the space of possible sequences of ring times of all of the $\Z$-many clocks, and the measure $\Poiss$ is exactly the desired Poisson measure on the ring times. The main difficulty consists in making sense of this when $\lim_{n \to -\infty} \mu_n$ is finite, i.e. when infinitely many particles with rates in increasing geometric progression are all located at a single point and so infinitely many of their clocks ring on any time interval. 

However, we first make formal the above claim that with probability $1$ only finitely many clocks with indices belonging to any half-infinite interval $[i,\infty)$ ring on a given time interval, which was necessary for \Cref{def:half_infinite_process} to make sense. First define notation 
\begin{align}
\begin{split}
\jumps: \R_{\geq 0} \times \left(\R_{\geq 0}^\N\right) &\to \Z_{\geq 0} \\
(T,(a_1,a_2,\ldots)) &\mapsto \sup\left(\right\{n \geq 0: \sum_{i=1}^n a_i \leq T\left\}\right)
\end{split}
\end{align}
i.e. $\jumps(T,\cdot)$ tells how many times the clock parametrized by the element of $\R_{\geq 0}^\N$ has rung by time $T$. 

\begin{defi}\label{def:tomega}
Denote
\begin{equation}
\tOmega := \{\omega \in \Omega: \sum_{j=i}^\infty \jumps(T, \pi_j(\omega)) < \infty \text{ holds for every $T \geq 0$ and $i \in \Z$}\}.
\end{equation}
\end{defi}

\begin{lemma}\label{thm:tomega_full_measure}
The set $\tOmega \subset \Omega$ has full measure. 
\end{lemma}
\begin{proof}
It is an elementary computation with exponential random variables that for any $T$ and $i$,
\begin{equation}
    \label{eq:finite_omega}
    \sum_{j=i}^\infty \jumps(T, \pi_j(\omega)) < \infty 
\end{equation}
with probability $1$. 
Hence the set of $\omega \in \Omega$ for which \eqref{eq:finite_omega} holds for all $T' \in [0,T]$ is full measure, and the complement $\Omega \setminus \tOmega$ is therefore a union over $i \in \Z, T \in \N$ of measure $0$ sets. It therefore has measure $0$, so $\tOmega$ has full measure.
\end{proof}

We may couple the processes $\tPois^{\pi_{[-n,\infty)}(\mu),n}(T)$ on the probability space $\tOmega$ as follows. Simply note that any sequence of clock ring times for $\tPois_{-n}^{\pi_{[-n,\infty)}(\mu),n},\tPois_{-n+1}^{\pi_{[-n,\infty)}(\mu),n},\ldots$, viewed as an element of $\prod_{i=-n}^\infty \R_{\geq 0}^\N$, determines $(\tPois_{-n}^{\pi_{[-n,\infty)}(\mu),n}(T),\tPois_{-n+1}^{\pi_{[-n,\infty)}(\mu),n}(T),\ldots)$ for all $T \geq 0$ by the jump rules of \Cref{def:poisson_walks}. The random variable $\tPois^{\pi_{[-n,\infty)}(\mu),n}(T)$ is then a function on this probability space,
\begin{equation}
\tPois^{\pi_{[-n,\infty)}(\mu),n}(T): \prod_{i=-n}^\infty \R_{\geq 0}^\N \to \sSig_{\infty}
\end{equation}
for any $T \geq 0$. Therefore
\begin{equation}
\prod_{n \geq 1} \tPois^{\pi_{[-n,\infty)}(\mu),n}(T) \circ \Proj_{[-n,\infty)}: \tOmega \to \prod_{n \geq 1} \sSig_{\infty}
\end{equation}
defines a coupling of all random variables $\{\tPois^{\pi_{[-n,\infty)}(\mu),n}(T): n \geq 1\}$ on $\tOmega$, where $\Proj_{[-n,\infty)}$ denotes projection onto coordinates $-n,-n+1,\ldots$. For each $\omega \in \tOmega$ we denote by $(\Pois^{\pi_{[-n,\infty)}(\mu),n}_i(T))(\omega) \in \bZ$ the corresponding coordinate of $\Pois^{\pi_{[-n,\infty)}(\mu),n}(T)$ under $\omega$. Finally, we may define the desired object.


\begin{defi}\label{def:limit_process}
For any $\mu \in \sSig_{2\infty}$, we define the continuous-time stochastic process $\Pois^{\mu,2\infty}_T, T \geq 0$ on $\sSig_{2\infty}$ by setting
\begin{align}
    \label{eq:define_infinite_dynamics}
    \begin{split}
    \Pois^{\mu,2\infty}(T): \tOmega &\to \sSig_{2\infty} \\
    \omega &\mapsto \left(\lim_{n \to \infty} (\tPois^{\pi_{[-n,\infty)(\mu),n}}_i(T))(\omega) \right)_{i \in \Z}
    \end{split}
\end{align}
for each $T \geq 0$.
\end{defi}

We note that the limit must be taken along $n \in \Z_{\geq -i}$, as $\tPois^{\pi_{[-n,\infty)(\mu),n}}_i(T))(\omega)$ is only well-defined if $n \geq -i$.

\begin{prop}\label{thm:construct_limit}
For any $T \geq 0$ and $\omega \in \tOmega$, the limit \eqref{eq:define_infinite_dynamics} exists and defines a $\sSig_{2\infty}$-valued random variable\footnote{i.e. it is measurable in the $\sigma$-algebra on $\sSig_{2\infty} \subset \bZ^\Z$ inherited from the product $\sigma$-algebra, where each $\bZ$ factor has the discrete $\sigma$-algebra.}. Furthermore, the resulting stochastic process in $T \geq 0$ is Markovian.
\end{prop}

We first establish a preparatory lemma:

\begin{lemma}
    \label{thm:monotonicity}
    For every $n \in \Z_{\geq 1}, \omega \in \tOmega, T \in \R_{\geq 0}, i \in \Z_{\geq -n}$, the following inequality holds:
\begin{equation}
    \label{eq:monotonicity}
    (\tPois_i^{\pi_{[-n,\infty)}(\mu),n}(T))(\omega) \geq (\tPois_i^{\pi_{[-n-1,\infty)}(\mu),n+1}(T))(\omega).
\end{equation}
\end{lemma}

\Cref{thm:monotonicity} is a purely deterministic/combinatorial fact, and the idea behind it is that $\tPois^{\cdot,n+1}_T$ has an extra particle in front compared to $\tPois^{\cdot,n}_T$, which may block the others but will never bring them further ahead. It holds for the half-infinite processes $\tPois$ but not for the finite $n$ approximations $\Pois^{\pi_{[-n,n]}(\mu),n}$, as these do not account for pushing by higher-indexed particles. This is the main reason we use the former process rather than the latter in our construction. 

\begin{proof}[Proof of \Cref{thm:monotonicity}]
Since $\omega \in \tOmega$, the clocks $-n-1,-n,-n+1,\ldots$ only ring a finite number of times in any interval $[0,T]$. Additionally, the lemma clearly holds at time $T=0$. Hence it suffices to show that if \eqref{eq:monotonicity} is true for each $i$ before a given clock rings, then it is also true for each $i$ after that clock rings, for then we may induct on the (finite, by above) number of rings. Let $T,\eps \geq 0$ be such that \eqref{eq:monotonicity} holds at time $T$, and under the event $\omega$ exactly one clock rings on the interval $[T,T+\eps]$.

If the strict inequality case of \eqref{eq:monotonicity} holds for a given $i$ before the clock rings (i.e. at time $T$), then clearly \eqref{eq:monotonicity} still holds after at time $T+\eps$ because the $\tPois_i^{\pi_{[-n,\infty)}(\mu),n}$ can change by at most $1$ when any clock rings. So it remains to consider the equality case
\begin{equation}
    \label{eq:pre-jump}
    (\tPois_i^{\pi_{[-n,\infty)}(\mu),n}(T))(\omega) = (\tPois_i^{\pi_{[-n-1,\infty)}(\mu),n+1}(T))(\omega)
\end{equation}
of \eqref{eq:monotonicity} holds for some index $i$ at time $T$, and the $(n+1)\tth$ approximation has a jump at the same index,
\begin{equation}\label{eq:i_jump}
    (\tPois_i^{\pi_{[-n-1,\infty)}(\mu),n+1}(T+\eps))(\omega) = (\tPois_i^{\pi_{[-n-1,\infty)}(\mu),n+1}(T))(\omega)+1.
\end{equation}
To show that \eqref{eq:monotonicity} continues to hold at time $T+\eps$, we must show that this jump occurs at the same location for the $n\tth$ approximation,
\begin{equation}
\label{eq:i_jump_wts}
 (\tPois_i^{\pi_{[-n,\infty)}(\mu),n}(T+\eps))(\omega) = (\tPois_i^{\pi_{[-n,\infty)}(\mu),n}(T))(\omega)+1
\end{equation}

The clock that rings to induce the jump \eqref{eq:i_jump} must be the $j\tth$ clock, for some $j \geq i$ for which $(\tPois_j^{\pi_{[-n-1,\infty)}(\mu),n+1}(T))(\omega) = (\tPois_i^{\pi_{[-n-1,\infty)}(\mu),n+1}(T))(\omega)$, by the definition of our dynamics. Since \eqref{eq:monotonicity} held before the jump, we have 
\begin{align}
\begin{split}
(\tPois_i^{\pi_{[-n,\infty)}(\mu),n}(T))(\omega) &\geq (\tPois_j^{\pi_{[-n,\infty)}(\mu),n}(T))(\omega) \\ 
&\geq (\tPois_j^{\pi_{[-n-1,\infty)}(\mu),n+1}(T))(\omega) \\
&= (\tPois_i^{\pi_{[-n-1,\infty)}(\mu),n+1}(T))(\omega) \\ 
&= (\tPois_i^{\pi_{[-n,\infty)}(\mu),n}(T))(\omega)
\end{split}
\end{align}
(using \eqref{eq:pre-jump}), so all above inequalities must be equalities. It follows that the particle of $\tPois^{\pi_{[-n(\mu),\infty)},n}_T$ which jumps on $[T,T+\eps]$ began at position $(\tPois_i^{\pi_{[-n,\infty)}(\mu),n}(T))(\omega)$ rather than some other one. Hence one of the following must be true: (a) \eqref{eq:i_jump_wts} holds,
or (b) $i>-n$ and $(\tPois_i^{\pi_{[-n,\infty)}(\mu),n}(T))(\omega) = (\tPois_{i-1}^{\pi_{[-n,\infty)}(\mu),n}(T))(\omega)$ (for then $\tPois_i^{\pi_{[-n,\infty)}(\mu),n}$ is blocked by $\tPois_{i-1}^{\pi_{[-n,\infty)}(\mu),n}$). 

Suppose for the sake of contradiction that (b) holds. Then since \eqref{eq:monotonicity} holds for $i-1$ at time $T$ by inductive hypothesis,
\begin{align}
\begin{split}
(\tPois_i^{\pi_{[-n,\infty)}(\mu),n}(T))(\omega) &= (\tPois_{i-1}^{\pi_{[-n,\infty)}(\mu),n}(T))(\omega) \\ 
&\geq (\tPois_{i-1}^{\pi_{[-n-1,\infty)}(\mu),n+1}(T))(\omega) \\ 
&\geq (\tPois_i^{\pi_{[-n-1,\infty)}(\mu),n+1}(T))(\omega) \\ 
&= (\tPois_i^{\pi_{[-n,\infty)}(\mu),n}(T))(\omega),
\end{split}
\end{align}
so again all inequalities must be equalities and 
\begin{equation}\label{eq:n+1_still_equal}
(\tPois_{i-1}^{\pi_{[-n-1,\infty)}(\mu),n+1}(T))(\omega) = (\tPois_i^{\pi_{[-n-1,\infty)}(\mu),n+1}(T))(\omega).
\end{equation}
Since only one jump occurs on the interval $[T,T+\eps]$, \eqref{eq:i_jump} and \eqref{eq:n+1_still_equal} imply that 
\begin{equation}
(\tPois_i^{\pi_{[-n-1,\infty)}(\mu),n+1}(T+\eps))(\omega) = (\tPois_{i-1}^{\pi_{[-n-1,\infty)}(\mu),n+1}(T+\eps))(\omega) + 1,
\end{equation}
which violates the weakly decreasing order. Hence (b) cannot hold, so \eqref{eq:i_jump_wts} holds, which completes the proof.
\end{proof}

\begin{proof}[Proof of \Cref{thm:construct_limit}]
We show that for any $\omega \in \tOmega,i \in \Z,T \in \R_{\geq 0}$, the limit
\begin{equation}
    \label{eq:want_limit_exists}
    \lim_{n \to \infty} (\tPois^{\pi_{[-n,\infty)(\mu),n}}_i(T))(\omega) 
\end{equation}
exists.

The sequence $((\tPois_i^{\pi_{[-n,\infty)}(\mu),n}(T))(\omega))_{n \geq -i}$ is bounded below by $(\tPois_i^{\pi_{[-n,\infty)}(\mu),n}(0))(\omega)$ (which is independent of $n \geq -i$), because coordinates of $\tPois^{\cdot,n}_T$ are nondecreasing in time. Since $((\tPois_i^{\pi_{[-n,\infty)}(\mu),n}(T))(\omega))_{n \geq -i}$ is also decreasing in $n$ by \Cref{thm:monotonicity}, it is immediate that the limit \eqref{eq:want_limit_exists} exists. Hence $\Pois^{\mu,2\infty}(T)$ is well-defined. Furthermore, each coordinate $\Pois^{\mu,2\infty}_i(T)$ is a limit of measurable functions $\tPois_i^{\pi_{[-n,\infty)}(\mu),n}(T): \tOmega \to \bZ$ and hence measurable, so $\Pois^{\mu,2\infty}(T)$ is measurable with respect to the product $\sigma$-algebra on $\bZ^\Z$.

We now show $\Pois^{\mu,2\infty}(T)$ is Markovian, which holds by the following facts:
\begin{itemize}
    \item For any fixed $T \geq 0$, $\Pois^{\mu,2\infty}(T)$ is determined by $(\tPois^{\pi_{[-n,\infty)}(\mu),n})_{n \geq 1}(T)$ by the above.
    \item 
    For $s \geq 0$ and for each $n \geq 1$, $\tPois^{\pi_{[-n,\infty)}(\mu),n}({T+s})$ is determined by $\tPois^{\pi_{[-n,\infty)}(\mu),n}(T)$ together with the complete data of which clocks ring when on the interval $[T,T+s]$, by definition. 
    \item The complete data of which clocks ring when on the interval $[T,T+s]$ is independent of everything earlier, by the memoryless property of exponential distributions.
\end{itemize}
This completes the proof.
\end{proof}

Some properties of $\Pois^{\mu,2\infty}(T)$ will be useful later.

\begin{defi}
    \label{def:trunc_map}
    For any $d \in \Z$ we define $F_d: \sSig_{2\infty} \to \sSig_{2\infty}$ by 
    \begin{equation}
F_d((\mu_n)_{n \in \Z}) = (\min(\mu_n,d))_{n \in \Z}.
    \end{equation}
    We define $F_d$ on $\sSig_\infty$ and $\sSig_n$ in exactly the same way.
\end{defi}

\begin{prop}\label{thm:markovian_projections}
For any $d \in \Z$ and $\mu \in \sSig_\infty$, $F_d(\Pois^{\mu,2\infty}(T))$ is a Markov process.
\end{prop}
\begin{proof}
It is clear from \Cref{def:poisson_walks} and \Cref{def:half_infinite_process} that $F_d(\tPois^{\nu,n}(T))$ is a Markov process for any $\nu \in \sSig_{\infty}$. Clearly $F_d(\Pois^{\mu,2\infty}(T))$ is a limit of $F_d(\tPois^{\pi_{[-n,\infty]}(\mu),n}(T))$, by the same proof as \Cref{thm:construct_limit}, and the Markov property is inherited by the limit as in that proof.
\end{proof}

\begin{prop}\label{thm:describe_projections}
Fix $d \in \Z$. If $\mu \in \sSig_{2\infty}$ has a part $\mu_k \geq d$ for some $k$, then the process $(\min(d,\Pois^{\mu,2\infty}_{k+i}(T))_{i \in \Z_{\geq 1}}$, obtained by throwing away the coordinates which are equal to $d$ for all time from $F_d(\Pois^{\mu,2\infty}(T))$, is equal in multi-time distribution to $F_d(\Pois^{(\mu_{k+1},\mu_{k+2},\ldots),\infty}(t^k T))$. 
\end{prop}
\begin{proof}
By our construction,
\begin{equation}
(\min(d,\Pois^{\mu,2\infty}_{k+i}(T))_{i \in \Z_{\geq 1}} = \left(\lim_{n \to \infty} F_d(\tPois^{\pi_{[-n,\infty)(\mu),n}}(T))_{k+i} \right)_{i \in \Z_{\geq 1}}.
\end{equation}
For all $n \geq -k$, 
\begin{equation}
(F_d(\tPois^{\pi_{[-n,\infty)(\mu),n}}(T))_{k+i})_{i \in \Z_{\geq 1}} = F_d(\Pois^{(\mu_{k+1},\mu_{k+2},\ldots),\infty}_i(t^k T))_{i \in \Z_{\geq 1}})
\end{equation}
in multi-time distribution by the explicit description of the dynamics \Cref{def:poisson_walks} and \Cref{def:half_infinite_process}. This completes the proof.
\end{proof}

For completeness, we relate the above discussion to what was stated in the Introduction.

\begin{proof}[Proof of \Cref{thm:sea_intro}]
The construction was done in \Cref{thm:construct_limit}, the Markovian projection property is \Cref{thm:markovian_projections}, and the explicit description of these projections is given in \Cref{thm:describe_projections}.
\end{proof}

We now prove that our construction is the bulk limit of the processes $\Pois^{\nu,n}$, which we will not use in further proofs, but reassures us that the above coupling did indeed capture the notion of a bi-infinite limiting version of $\Pois^{\nu,n}$.

\begin{prop}\label{thm:construct_limit_intro_slight_generalization}
For any $\mu \in \sSig_{2\infty}$, there exists a stochastic process $\Pois^{\mu,2\infty}(T),T \geq 0$, with $\Pois^{\mu,2\infty}(0)=\mu$, which is a bulk limit of the processes $\Pois^{\nu,n}$ above in the following sense. The processes $\Pois^{(\mu_{1-r_n},\ldots,\mu_{n-r_n}),n}(T), n \geq 1$ may be coupled on $\tOmega$ such that for any $D \in \N$, $T_1 \in \R_{\geq 0}$ and sequence of `bulk observation points' $r_n, n \geq 1$ with $r_n \to \infty$ and $n-r_n \to \infty$,
\begin{equation}
 (\Pois^{(\mu_{1-r_n},\ldots,\mu_{n-r_n}),n}_{r_n-D}(t^{-r_n}T),\ldots,\Pois^{(\mu_{1-r_n},\ldots,\mu_{n-r_n}),n}_{r_n+D}(t^{-r_n}T)) \to (\Pois^{\mu,2\infty}_{-D}(T),\ldots,\Pois^{\mu,2\infty}_{D}(T)) 
\end{equation}
almost surely for all $0 \leq T \leq T_1$. 
\end{prop}
\begin{proof}
We couple $\Pois^{(\mu_{1-r_n},\ldots,\mu_{n-r_n}),n}(t^{-r_n}T), n \geq 1$ on $\tOmega$ in the obvious way, namely by defining
\begin{equation}
\Pois^{(\mu_{1-r_n},\ldots,\mu_{n-r_n}),n}(t^{-r_n}T): \pi_{[1-r_n,n-r_n]}(\tOmega) \to \sSig_n
\end{equation}
by identifying the $n$ coordinates of $\pi_{[1-r_n,n-r_n]}(\tOmega)$ with the clock times\footnote{Note that the $n$ clocks corresponding to these $n$ coordinates $1-r_n,\ldots,n-r_n$ in $\tOmega$ have rates $t^{1-r_n},\ldots,t^{n-r_n}$, while the clocks of $\Pois^{(\mu_{1-r_n},\ldots,\mu_{n-r_n}),n}$ have rates $t,\ldots,t^n$, and the time change $t^{-r_n}$ exactly accounts for this difference.} of the $n$ particles of $\Pois^{(\mu_{1-r_n},\ldots,\mu_{n-r_n}),n}$. Similarly, we have the coupling
\begin{equation}
\tPois^{(\mu_{1-r_n},\mu_{2-r_n},\ldots),r_n-1}(T): \pi_{[1-r_n,\infty)}(\tOmega) \to \sSig_\infty.
\end{equation}
For each $\omega \in \tOmega$, there exists an index $j_0$ such that clocks $j_0,j_0+1,\ldots$ do not ring on the interval $[0,T_1]$. Hence as long as $n-r_n \geq j_0$, 
\begin{equation}\label{eq:finite_tpois_omega_match}
\Pois^{(\mu_{1-r_n},\ldots,\mu_{n-r_n}),n}(T)(\omega) = \pi_{[1,n]}\left(\tPois^{(\mu_{1-r_n},\mu_{2-r_n},\ldots),r_n-1}(T)(\omega)\right)
\end{equation}
for any $T \in [0,T_1]$. Because $n-r_n \to \infty$, this is true for all sufficiently large $n$. Since $r_n \to \infty$,
\begin{equation}\label{eq:r_n_limit}
\lim_{n \to \infty} \tPois^{(\mu_{1-r_n},\mu_{2-r_n},\ldots),r_n-1}_i(T) = \lim_{n \to \infty} \tPois^{(\mu_{-n},\mu_{-n+1},\ldots),n}_i(T) = \Pois^{2\infty,\mu}(T).
\end{equation}
Combining \eqref{eq:r_n_limit} with \eqref{eq:finite_tpois_omega_match} completes the proof.
\end{proof}

The `edge version' we saw earlier in \Cref{thm:rmt_edge_intro} is derived easily from the above:

\begin{defi}\label{def:pois_edge}
One may identify the set $\bSig_{edge}$ of \eqref{eq:sig_edge} with
\begin{equation}
\{\nu \in \sSig_{2\infty}: \nu_i \in \Z \text{ for }i \leq 0 \text{ and }\nu_1 = \nu_2 = \ldots = -\infty\}.
\end{equation} 
For any $\mu \in \bSig_{edge}$, letting $\hat{\mu} \in \sSig_{2\infty}$ be its image under the above map, we define
\begin{equation}
\Pois^{\mu,edge}(T) = \Pois^{\hat{\mu},2\infty}(T).
\end{equation}
\end{defi}

Note that if $\mu$ has a part $\mu_i \geq d$, then only the parts $F_d(\Pois^{\mu,2\infty})_j, j > i$ can evolve, leading to a much simpler process because the sum of their jump rates is finite. The following result, which we have stated in terms of Markov generators because we will need this later, says informally that if $\mu$ has a part $\geq d$, then $F_d(\Pois^{\mu,2\infty})$ evolves by the same reflecting Poisson dynamics as the prelimit process. This will be extremely useful for random matrix results, as for such $\mu$ we may check convergence to $F_d(\Pois^{\mu,2\infty})$ by taking asymptotics of generators/transition matrices.

\begin{prop}\label{thm:compute_Q}
Let $d \in \N$ and let $\mu \in \sSig_{2\infty}$ be such that 
\begin{equation}
\label{eq:def_i_0}
i_0 := \mu_d'+1 = \inf(\{i \in \Z: \mu_i < d\}) > -\infty,
\end{equation}
and let
\begin{equation}
N := \begin{cases}
\infty & \mu_i > -\infty \text{ for all }i \\ 
\max(\{i: \mu_i > -\infty\}) & \text{else}
\end{cases}
\end{equation}
Let $Q: F_d(\sSig_{2\infty}) \to F_d(\sSig_{2\infty})$ be the matrix defined by 
\begin{equation}
\label{eq:explicit_Q}
Q(\nu,\kappa) = 
\begin{cases}
-\frac{t^{\nu_d'+1} - t^{N+1}}{1-t} & \kappa = \nu \\
 \frac{t^i(1-t^{m_{\nu_i}(\nu)})}{1-t} & \text{there exists $i \in \Z$ such that }\kappa_j = \nu_j + \bbone(j = i) \text{ for all }j \in \Z \\ 
0 & \text{else}
\end{cases}
\end{equation}
for all $\nu,\kappa \in F_d(\sSig_{2\infty})$, where we recall the definition of $m_\ell(\nu)$ from \Cref{def:finite_sigs}. Then the matrix exponential $e^{TQ}: F_d(\sSig_{2\infty}) \to F_d(\sSig_{2\infty})$ is well-defined, and
\begin{equation}\label{eq:gen_Q_def}
\Pr(F_d(\Pois^{\mu,2\infty}(T+T_0)) = \kappa| F_d(\Pois^{\mu,2\infty}(T_0))=\nu) = (e^{TQ})(\nu,\kappa).
\end{equation} 
\end{prop}

\begin{proof}

For any $n > i_0$, by \Cref{thm:describe_projections} and \Cref{def:half_infinite_process} $F_d(\Pois^{\mu,2\infty}(T))$ is equal in (multi-time joint) distribution to a copy of $F_d(\tPois^{\pi_{[-n,\infty)}(\mu),n}(T))$, padded with infinitely many entries $d$ on the left. Note that for any $\mu$ as in the statement, all entries $F_d(\tPois^{\pi_{[-n,\infty)}(\mu),n}(T))_j, j \leq \nu_d'$ never change because they are already equal to $d$. Additionally, if $N$ is finite, the entries $F_d(\Pois^{\mu,2\infty}(T))_j, j > N$ do not change because they are equal to $-\infty$. Meanwhile, all entries $F_d(\tPois^{\pi_{[-n,\infty)}(\mu),n}(T))_j,  i_0 \leq j \leq N$ jump according to Poisson clocks of rate $t^j$ as before, until they reach $d$, at which point they jump no longer. 

With this explicit description, we may compute the transition rates $Q(\nu,\kappa)$ for $F_d(\Pois^{\mu,2\infty}(T))$. By this description, the transition rate out of a state $\nu \in F_d(\bSig_{2\infty})$ is given by the negative sum of the jump rates of particles which are at position (a) strictly less than $d$, and (b) greater than $-\infty$, since these are exactly the particles for which having their clock ring will cause the process $F_d(\Pois^{\mu,2\infty}(T))$ to leave $\nu$. This is the sum of jump rates of the particles of index $\nu_d'+1,\nu_d'+2,\ldots,N$, so this transition rate is 
\begin{equation}
    \label{eq:diagonal_gen_entry}
    Q(\nu,\nu) = -(t^{\nu_d'+1}+\ldots+t^N) = - \frac{t^{\nu_d'+1} - t^{N+1}}{1-t}.
\end{equation}
Now we compute the jump rate $Q(\nu,\kappa)$ where $\kappa \neq \nu$. From the above description it is clear that this is $0$ unless the condition on the second line of \eqref{eq:explicit_Q} is met, i.e. only one particle needs to jump in order to get from $\nu$ to $\kappa$. As in \eqref{eq:explicit_Q}, let us take $i$ to be the index of that particle, so $\kappa_i=\nu_i+1$. Due to the jump-donation condition \ref{item:interact} from the Introduction, the transition rate from $\nu$ to $\kappa$ is then just the sum of jump rates of all particles at position $\nu_i$. By the definition of $m_{\nu_i}(\nu)$ (\Cref{def:bi-infinite_sig}), the particles at position $\nu_i$ are exactly the ones of indices $i,i+1,\ldots,i+m_{\nu_i}(\nu)-1$. Hence the transition rate from $\nu$ to $\kappa$ is 
\begin{equation}
    \label{eq:off_diagonal_gen_entry}
    Q(\nu,\kappa) = t^i + \ldots + t^{i+m_{\nu_i}(\nu)-1} = \frac{t^i(1-t^{m_{\nu_i}(\nu)})}{1-t}.
\end{equation}
Combining \eqref{eq:diagonal_gen_entry}, \eqref{eq:off_diagonal_gen_entry}, and the above discussion that these are the only nonzero entries of the transition matrix completes the proof.
\end{proof}

\begin{defi}
    \label{def:shift}
    Define the \emph{forward shift map}
\begin{align*}
s: \sSig_{2\infty} &\to \sSig_{2\infty} \\
(\mu_n)_{n \in \Z} &\mapsto (\mu_{n+1})_{n \in \Z}
\end{align*}
\end{defi}

Because the $i\tth$ coordinate $\mu_i(T)$ of $\Pois^{\mu,2\infty}(T)$ behaves as a Poisson jump process with rate $t^i$ (neglecting interactions with the other coordinates), the $i\tth$ coordinate of $s(\Pois^{\mu,2\infty}(T))$ has rate $t^{i+1} = t \cdot t^i$, i.e. $s$ has the effect of slowing down each jump rate by a factor of $t$. Heuristically this justifies the following.

\begin{prop}\label{thm:time_symmetry}
If $a \in \Z$ and $\mu = (a)_{n \in \Z}$, then 
\begin{equation}
    \label{eq:time_sym}
    s(\Pois^{\mu,2\infty}(t^{-1} \cdot T)) = \Pois^{\mu,2\infty}(T)
\end{equation}
in distribution as stochastic processes.
\end{prop}
\begin{proof}
Define a map 
\begin{align*}
    \xi:\Omega & \to \Omega \\
    ((a_{n,i})_{i \in \N})_{n \in \Z} & \mapsto ((t \cdot a_{n+1,i})_{i \in \N})_{n \in \Z}
\end{align*}
The map $\xi$ scales the waiting times $a_{n,i}$ by $t$ and shifts which coordinate $\mu_n$ they correspond to. Since these waiting times are exponential variables with rates in geometric progression with common ratio $t$ under the measure $\Poiss \in \cM(\Omega)$ defined in the proof of \Cref{thm:construct_limit}, it follows that 
\begin{equation}
    \label{eq:omega_meas_inv}
    \xi_*(\Poiss) = \Poiss.
\end{equation} 
It is also immediate from the definition of $\xi$ that for any $T \geq 0$ and $\omega \in \tOmega$,
\begin{equation}\label{eq:xi_exact_equality}
(\tPois_{i}^{\pi_{[-n,\infty)}(\mu),n}(T))(\omega) = (\tPois_{i-1}^{\pi_{[-n-1,\infty)}(\mu),n+1}(t^{-1}T))(\xi(\omega)).
\end{equation}
Hence clearly
\begin{equation}
    \label{eq:limits_equal2}
    \lim_{n \to \infty} (\tPois_i^{\pi_{[-n,\infty)}(\mu),n}(T))(\omega) = \lim_{n \to \infty} (\tPois_{i-1}^{\pi_{[-n-1,\infty)}(\mu),n+1}(t^{-1}T))(\xi(\omega)),
\end{equation}
and in view of the construction in \Cref{def:limit_process} this implies \eqref{eq:time_sym}.
\end{proof}



\subsection{Convergence of measures on $\sSig_{2\infty}$} \label{subsec:measure_conv} Having constructed the putative universal object $\Pois^{\mu,2\infty}$ and shown some basic properties, we now set up what is needed to prove convergence to it. To speak of weak convergence of $\sSig_{2\infty}$-valued random variables, we must at minimum equip $\sSig_{2\infty}$ with a topology. We first equip $\bZ$ with the toplogy where open sets are generated by the basis of sets
\begin{equation}\label{eq:set_basis}
\sB := \{\{n\}: n \in \Z \cup \{-\infty\}\} \cup \{[n,\infty]: n \in \Z\} \cup \{\emptyset\} \subset \mathcal{P}(\bZ)
\end{equation}
(here $\mathcal{P}$ denotes the power set). In other words, open sets are either finite subsets of $\Z \cup \{-\infty\}$ or unions of these with intervals $[n,\infty]$, where here and below we use square braces to indicate that the interval includes the infinite endpoint. For concreteness later we note that the closed sets in this topology are those which, if they contain arbitrarily large positive finite integers, they also contain $\infty$. 

\begin{rmk}
The reason for this topology, which treats $\infty$ and $-\infty$ on unequal footing, comes from the topology of $\Q_p$. Since the elements of $\bZ$ will correspond to singular numbers, we wish the map $\Z \cup \{\infty\} \to \Q_p$ given by $n \mapsto p^n$ to be continuous in the $p$-adic norm topology, where as usual we set $p^\infty = 0$. The reason we include $-\infty$ entries in $\bSig_{2\infty}$, even though $p^{-\infty}$ is not defined, is more a notational convenience, as it allows us to embed infinite signatures with a last element $\mu_i$ into $\bSig_{2\infty}$ as $(\ldots,\mu_{i-1},\mu_i,-\infty,\ldots)$.
\end{rmk}

We now give $\sSig_{2\infty}$ the topology it inherits from the product topology on $\bZ^\Z$, where each $\bZ$ factor has the topology above. Equivalently, this is the topology of pointwise convergence on $\bZ^\Z$. When we speak of measures on $\sSig_{2\infty}$, we will always mean measures with respect to the Borel $\sigma$-algebra determined by this topology. Note that this is just the product $\sigma$-algebra of the discrete $\sigma$-algebras on each $\bZ$ factor, which is the one we took earlier in \Cref{thm:construct_limit}.

The space $\bZ$ is second-countable and separable, hence metrizable by Urysohn's theorem, hence the product $\bZ^\Z$ (and therefore $\sSig_{2\infty}$) is metrizable as well. This makes the following two definitions of weak convergence equivalent by the portmanteau theorem.

\begin{defi}\label{def:weak}
A sequence of probability measures $(M_n)_{n \geq 1}$ on $\sSig_{2\infty}$ converges weakly to $M$ if, for every $S \subset \sSig_{2\infty}$ which is a continuity set of $M$ (i.e. $M(\partial S)=0$), $M_n(S) \to M(S)$ as $n \to \infty$. Equivalently, for every bounded, continuous $f: \sSig_{2\infty} \to \R$, 
\begin{equation}
\int_{\bZ^\Z} f dM_n \to \int_{\bZ^\Z} f dM.
\end{equation}
\end{defi}

We reduce weak convergence to more checkable, combinatorial conditions, which are what we will actually show. 

\begin{defi}\label{def:sS_cU}
For $I \subset \Z$ let $\pi_I: \bZ^\Z \to \bZ^I$ be the projection. For any finite $I \subset \Z$, let
\begin{equation}
\sS_I := \{\prod_{i \in I} A_i \subset \bZ^I: A_i \in \sB \text{ for all }i\} \subset \mathcal{P}(\bZ^I),
\end{equation}
where $\sB$ is as defined in \eqref{eq:set_basis}. Furthermore, let
\begin{equation}
\cU := \bigcup_{\substack{I \subset \Z \\ I \text{ finite}}} \pi_I^{-1}(\sS_I) \subset \mathcal{P}(\bSig_{2\infty}).
\end{equation}
\end{defi}

\begin{lemma}\label{thm:weak_eq}
A sequence of probability measures $(M_n)_{n \geq 1}$ on $\sSig_{2\infty}$ converges weakly to a probability measure $M$ if, for every finite $I \subset \Z$ and $A \in \sS_I$,
\begin{equation}\label{eq:pushforward_converge}
((\pi_I)_*(M_n))(A) \to ((\pi_I)_*(M))(A).
\end{equation}
More generally, for any $k \in \Z_{\geq 1}$, a sequence of probability measures $(M_n)_{n \geq 1}$ on $\sSig_{2\infty}^k$ converges weakly to a probability measure $M$ if, for every collection of finite sets $\prod_{i=1}^k I_i \subset \Z^k$, and sets $A_i \in \sS_{I_i}, 1 \leq i \leq k$,
\begin{equation} \label{eq:k_pushforward_cvg}
\left(\prod_{i=1}^k \pi_{I_i}\right)_*(M_n)\left(\prod_{i=1}^k A_i\right) \to \left(\prod_{i=1}^k \pi_{I_i}\right)_*(M)\left(\prod_{i=1}^k A_i\right).
\end{equation}
\end{lemma}

\begin{proof}[Proof of \Cref{thm:weak_eq}]
We prove the $k=1$ case first. Note that (i) $\cU$ is closed under finite intersections, and (ii) every open set in $\bZ^\Z$ is a countable union of elements of $\cU$, which follows since $\bZ$ is countable. By \cite[Theorem 2.2]{bil1968convergence}, the two properties (i), (ii) imply that for weak convergence $M_n \to M$, it suffices to check that 
\begin{equation}
    \label{eq:check_cylinder}
    M_n(U) \to M(U)
\end{equation}
for every $U \in \cU$. This follows immediately from the hypothesis \eqref{eq:pushforward_converge} by the definition of $\cU$, completing the proof for $k=1$. For general $k$, we simply note that $\prod_{i=1}^k \cU \in \mathcal{P}(\sSig_{2\infty}^k)$ is also closed under finite intersections and its countable unions yield all open sets in $(\bZ^\Z)^k$ for the same reason as above, so \cite[Theorem 2.2]{bil1968convergence} applies.
\end{proof}

\begin{lemma}\label{thm:F_d_suffices}
Let $M$ be a probability measure on $(\sSig_{2\infty})^k$, and $(M_n)_{n \geq 1}$ be a sequence of probability measures on $(\sSig_{2\infty})^k$ such that for any $d \in \Z$, the sequence $((F_d^k)_*(M_n))_{n \geq 1}$ obeys the condition \eqref{eq:k_pushforward_cvg} with respect to $(F_d^k)_*(M)$. Then $M_n$ converges weakly to $M$.
\end{lemma}
\begin{proof}
By \Cref{thm:weak_eq} it suffices to check, for all finite sets $I_1,\ldots,I_k \subset \Z$ and all $A_1 \in \sS_{I_1},\ldots,A_k \in \sS_{I_k}$ (in the notation of \Cref{def:sS_cU}), that
\begin{equation}\label{eq:joint_meas_wts}
\left(\prod_{i=1}^k \pi_{I_i}\right)_*(M_n)\left(\prod_{i=1}^k A_i\right) \to \left(\prod_{i=1}^k \pi_{I_i}\right)_*(M)\left(\prod_{i=1}^k A_i\right).
\end{equation}
Each set $A_i$ is a product over $j \in I_i$ of sets $\{b_{i,j}\}$ or $[b_{i,j},\infty]$ with $b_{i,j} \in \Z \cup \{-\infty\}$. It is trivial that for $d > b_{i,j}$, $x \in \{b_{i,j}\}$ if and only if\footnote{Here we slightly abuse notation and view $F_d: \bSig_1 \to \bSig_1$ as a map on $\bZ$.} $F_d(x) \in F_d(\{b_{i,j}\})$, and $x \in [b_{i,j},\infty]$ if and only if $F_d(x) \in F_d([b_{i,j},\infty])$ (of course, both forward directions are automatic, but the backward directions would not be true if $\{b_{i,j}\}$ and $[b_{i,j},\infty]$ are replaced with arbitrary subsets of $\bZ$). Hence for any
\begin{equation}
d > \sup_{\substack{1 \leq i \leq k}} \sup_{j \in I_i} b_{i,j},
\end{equation}
we have that
\begin{equation}\label{eq:use_F_d_meas}
\left(\prod_{i=1}^k \pi_{I_i} \circ F_d\right)_*(M_n)\left(\prod_{i=1}^k F_d(A_i)\right) = \left(\prod_{i=1}^k \pi_{I_i}\right)_*(M_n)\left(\prod_{i=1}^k A_i\right)
\end{equation}
and similarly with $M_n$ replaced by $M$. By hypothesis, 
\begin{equation}
\left(\prod_{i=1}^k \pi_{I_i} \circ F_d\right)_*(M_n)\left(\prod_{i=1}^k F_d(A_i)\right) \to \left(\prod_{i=1}^k \pi_{I_i} \circ F_d\right)_*(M)\left(\prod_{i=1}^k F_d(A_i)\right),
\end{equation}
and applying \eqref{eq:use_F_d_meas} to both sides yields \eqref{eq:joint_meas_wts} and completes the proof.
\end{proof}

It will be useful in \Cref{sec:deducing} to use the following variant as well.

\begin{lemma}\label{thm:finite_sets_to_conj_parts}
Let $M$ be a probability measure on $(\sSig_{2\infty}^+)^k$ and $(\mu(1),\ldots,\mu(k)) \sim M$. Let $(M_n)_{n \geq 1}$ be a sequence of probability measures on $(\sSig_{2\infty}^+)^k$ such that if $(\mu(n,1),\ldots,\mu(n,k)) \sim M_n$ for each $n \in \Z_{\geq 1}$, then for any $d \in \Z_{\geq 1}$
\begin{equation}\label{eq:conj_part_convergence}
(\mu(n,i)_j')_{\substack{1 \leq i \leq k \\ 1 \leq j \leq d}} \to (\mu(i)_j')_{\substack{1 \leq i \leq k \\ 1 \leq j \leq d}}\quad \quad \quad \text{in distribution as $n \to \infty$.}
\end{equation}
Then $M_n \to M$ weakly as $n \to \infty$. 
\end{lemma}
\begin{proof}
For $\mu \in \sSig_{2\infty}^+$, $F_d(\mu)$ is uniquely determined by $\mu_1',\ldots,\mu_d'$ and vice versa, so the result follows immediately from \Cref{thm:F_d_suffices}.
\end{proof}

\section{Main theorem statement and comments}\label{sec:state_rmt_result}

We wish to talk about random finite-length signatures---singular numbers of the matrix product process---converging to random elements of $\sSig_{2\infty}$, so it is convenient to define an embedding of $\sSig_N$ into $\sSig_{2\infty}$.

\begin{defi}\label{def:shift_and_embedding}
For $\la \in \sSig_N$ define 
\begin{equation}
\iota_n(\la) = \begin{cases}
\infty & n \leq 0 \\
\la_n & 1 \leq n \leq N \\
-\infty & n > N
\end{cases},
\end{equation}
and let
\begin{align*}
\iota: &\sSig_N \inj \sSig_{2 \infty} \\
& (\la_1,\ldots,\la_N) \mapsto (\iota_n(\la))_{n \in \Z}.
\end{align*}
\end{defi}

We are now able to state the main dynamical result, which in the bulk case we will later augment to include the single-time marginal as well. It applies to both the bulk and edge: the sequence $(r_N)_{N \geq 1}$, which represents `observation points' of the matrix product process, may be taken such that $0 \ll r_N \ll N$ for a bulk limit, or $r_N = N-k$ for fixed $k$ for an edge limit. 

\begin{thm}\label{thm:gen_limit}
Fix $p$ prime and let $t=1/p$. For each $N \in \N$, let $A^{(N)}_i, i \geq 1$ be an iid sequence of left-$\GL_N(\Z_p)$-invariant random matrices in $\Mat_N(\Z_p)$, and let $r_N$ be a `bulk observation point' such that $r_N$ and the random variable
    \begin{equation}
    X_N := \corank(A^{(N)}_i \pmod{p})
    \end{equation}
    satisfy
    \begin{enumerate}
        \item[(i)] \label{eq:rN_bulk} $r_N \to \infty$ as $N \to \infty$,
        \item[(ii)]\label{eq:not_always_0} $\Pr(X_N = 0) < 1$ for every $N$, and 
        \item[(iii)] $X_N$ is far away from $r_N$ with high probability, in the sense that for every $j \in \Z$, 
        \begin{equation}\label{eq:lambda_not_large}
        \Pr(X_N > r_N+j) = o(c_N^{-1}) \quad \quad \quad \quad \text{as $N \to \infty$}
        \end{equation}
        where 
        \begin{equation}\label{eq:def_cN}
    c_N := \frac{t^{-r_N}}{\E [\bbone(X_N \leq r_N)(t^{-X_N} - 1)]} \quad \quad N =1,2,\ldots 
    \end{equation}
    \end{enumerate}
    Let $\mu \in \sSig_{2\infty}^+$ be such that $\lim_{n \to \infty} \mu_{-n} = \infty$, and let $B^{(N)} \in \Mat_N(\Z_p), N \geq 1$ be any left-$\GL_N(\Z_p)$-invariant matrices with singular numbers around $r_N$ matching $\mu$, i.e. for every $i \in \Z$
    \begin{equation}\label{eq:B_mu_limit}
        (s^{r_N} \circ \iota(\SN(B^{(N)})))_i = \mu_i
    \end{equation}
    for all sufficiently large $N$. Define the prelimit matrix product process $\Pi^{(N)}(\tau) = \SN(A^{(N)}_\tau \cdots A^{(N)}_1 B^{(N)})$ for $\tau \in \Z_{\geq 0}$, and the shifted version on $\sSig_{2\infty}$
    \begin{equation}
    \Lambda^{(N)}(T) := s^{r_N} \circ \iota(\Pi^{(N)}(\floor{c_N T})),T \in \R_{\geq 0}.
    \end{equation}
    Then we have convergence
    \begin{equation}\label{eq:gen_limit_conv}
        \Lambda^{(N)}(T) \xrightarrow{N \to \infty} \Pois^{\mu,2\infty}(T)
    \end{equation}
    in finite-dimensional distribution.
\end{thm}

Many remarks on \Cref{thm:gen_limit} are in order. First of all, the hypothesis (iii) is not the same as the more simply stated one given in the Introduction. That one is in fact a consequence of (iii) above, as \Cref{thm:prob_hypothesis} below shows. We gave the simpler-to-state version in \Cref{thm:universality_intro}, but the weaker one above is the condition that is naturally needed in the proof. Let us explain where this hypothesis comes from.

Without the indicator, $\E[t^{-X_N} - 1]$ is just the expected size of the set of nonzero vectors in the kernel of $A^{(N)}_i \pmod{p}$, which is a measure of how much multiplication by $A^{(N)}_i \pmod{p}$ should change the corank of a matrix. Somewhat surprisingly, this is also the only relevant feature of $A^{(N)}_i$ to determine how the singular numbers change asymptotically. 

The linear-algebraic lemmas used later in the proof rests on the assumption that the observation point $r_N$ is larger than $X_N = \corank(A^{(N)}_i \pmod{p})$, and sometimes much larger. However, for many reasonable matrix distributions such as the additive Haar measure, $X_N$ can be arbitrarily large, though this is a low-probability event. Because all of our lemmas break down on this event, we need that it is rare enough that it will not happen for any of the matrices $A^{(N)}_i$ that we consider. The number of such matrices is on the order of $c_N$ thanks to the above-mentioned fact that $c_N$ controls how much each matrix affects the singular numbers of the product. Hence, we need that $c_N$ times the probability of this bad event is $o(1)$, which is exactly \eqref{eq:lambda_not_large}. The hypothesis is just saying that the only important contribution to the change in corank of the matrix product comes from higher-probability events that $A_i^{(N)}$ have low (but nonzero) coranks, rather than lower-probability rare events of extremely high corank.

We now show that this hypothesis is implied by the one stated in \Cref{thm:universality_intro}, which takes the simpler form of a tail bound on the coranks.

\begin{prop}\label{thm:prob_hypothesis}
Let $(r_N)_{N \in \N}$ be a sequence with $r_N \to \infty$. Let $X_N, N \geq 1$ be any $\Z_{\geq 0}$-valued random variables satisfying a uniform tail bound
\begin{equation}
    \label{eq:intro_hyp}
    \Pr(X_N \geq k| X_N > 0) < C t^{1.0001 k}
\end{equation}
for some $C$ independent of $N$, where $t \in (0,1)$ as usual. Then 
\begin{equation}\label{eq:little_o_tfac}
\Pr(X_N \geq r_N + j) = o(\E[\bbone(X_N \leq r_N)(t^{r_N-X_N}-t^{r_N})])
\end{equation}
for every $j \in \Z$.
\end{prop}
\begin{proof}

Let us bound the right-hand side of \eqref{eq:little_o_tfac} below. First note that 
\begin{align}
    \begin{split}
        \bbone(X_N \leq r_N)(t^{r_N-X_N}-t^{r_N}) &= \bbone(1 \leq X_N \leq r_N)(t^{r_N-X_N}-t^{r_N}) \\ 
        & \geq t^{r_N} (t^{-1}-1)\bbone(1 \leq X_N \leq r_N) \\ 
        &= t^{r_N}(t^{-1}-1)(\bbone(X_N \geq 1) - \bbone(X_N > r_N)),
    \end{split}
\end{align}
where we used that $t^{-n} -1 \geq t^{-1}-1$ for $n \geq 1$. Taking expectations,
\begin{align}\label{eq:finishes_prop4.2}
\begin{split}
        \E[ \bbone(X_N \leq r_N)(t^{r_N-X_N}-t^{r_N})] &\geq t^{r_N}(t^{-1}-1)(\Pr(X_N \geq 1) - \Pr(X_N > r_N)) \\ 
        &\geq  t^{r_N} (t^{-1}-1)\left(\Pr(X_N \geq 1) - \Pr(X_N \geq 1) \cdot C t^{1.0001r_N}\right) \\ 
        &\geq t^{r_N}(t^{-1}-1)\Pr(X_N \geq 1) (1+o(1)),
\end{split}
\end{align}
where in the second inequality we have used the rewritten form 
\begin{equation}
    \Pr(X_N \geq r_N) < \Pr(X_N \geq 1) \cdot C t^{1.0001 r_N}
\end{equation}
of \eqref{eq:intro_hyp}, and in the the third simply used that $r_N \to \infty$. Now just note that for any $j \in \Z$,
\begin{equation}
    \Pr(X_N \geq r_N+j) \leq \Pr(X_N \geq 1) \cdot C t^{1.0001(r_N+j)} = o(\Pr(X_N \geq 1) t^{r_N}),
\end{equation}
and in combination with \eqref{eq:finishes_prop4.2} this proves \eqref{eq:little_o_tfac}.
\end{proof}

\begin{proof}[Proof of \Cref{thm:universality_intro} from \Cref{thm:gen_limit}]
By \Cref{thm:prob_hypothesis}, the hypotheses in \Cref{thm:universality_intro} imply those in \Cref{thm:gen_limit}, and the result follows.
\end{proof}

\begin{proof}[Proof of \Cref{thm:rmt_edge_intro}]
Exactly as for \Cref{thm:universality_intro}, taking $r_N = N$ in \Cref{thm:gen_limit} and using the natural inclusion $\bSig_{edge}  \hookrightarrow \sSig_{2\infty} $ taking $(\mu_i)_{i \in \Z_{\leq 0}}$ to $(\ldots,\mu_{-1},\mu_0,-\infty,-\infty,\ldots)$.
\end{proof}

One might also wonder where the definition of $c_N$ came from; why $\bbone(X_N \leq r_N)$ rather than, say, $\bbone(X_N \leq r_N - 1)$? We show that this is simply a matter of convenience and our hypothesis guarantee that any cutoff near $r_N$ will give the same result.

\begin{prop}\label{thm:hyp_cutoff_indep}
Suppose $r_N$ and $X_N$ are such that for every $j \in \Z$,
\begin{equation}\label{eq:prob_o}
\Pr(X_N \geq r_N + j) = o(\E[\bbone(X_N \leq r_N)(t^{r_N-X_N}-t^{r_N})]) \quad \quad \quad \text{as $N \to \infty$}.
\end{equation}
Then for every $j \in \Z$,
\begin{equation}
\E[\bbone(X_N \leq r_N+j)(t^{r_N-X_N}-t^{r_N})] = (1+o(1))\E[\bbone(X_N \leq r_N)(t^{r_N-X_N}-t^{r_N})].
\end{equation}
\end{prop}
\begin{proof}
We will prove the case $j > 0$, as the case $j < 0$ is the same after replacing $r_N$ by $r_N-j$. It suffices to show
\begin{equation}\label{eq:cutoff_wts}
\E[\bbone(r_N < X_N \leq r_N + j)(t^{r_N-X_N}-t^{r_N})] = o(1) \E[\bbone(X_N \leq r_N)(t^{r_N-X_N}-t^{r_N})].
\end{equation}
Since 
\begin{equation}
\E[\bbone(r_N < X_N \leq r_N + j)(t^{r_N-X_N}-t^{r_N})] \leq t^{-j} \Pr(r_N < X_N \leq r_N + j) \leq t^{-j} \Pr(X_N > r_N), 
\end{equation}
which is $o(\E[\bbone(X_N \leq r_N)(t^{r_N-X_N}-t^{r_N})])$ by \eqref{eq:prob_o}, \eqref{eq:cutoff_wts} follows.
\end{proof}

\section{Reducing {\Cref{thm:gen_limit}} to Markov generator asymptotics}\label{sec:markov_preliminaries}

Our goal is to understand the asymptotic dynamics of singular numbers $\Pi^{(N)}(\tau) = \SN(A_\tau \cdots A_1)$ (in the notation of \Cref{thm:gen_limit}) under matrix products $A_1,A_2, \ldots \in \Mat_N(\Z_p)$ in an `observation window' around some $r_N$, i.e. $\Pi^{(N)}_i(\tau)$ where $i = r_N + const$. It is helpful to view the $\Pi^{(N)}_i(\tau)$ as a collection of particles on $\Z$, which may inhabit the same location, and which `jump' in discrete time $\tau$ by $\Pi^{(N)}_i(\tau+1) - \Pi^{(N)}_i(\tau)$ at each `time step' $\tau \mapsto \tau+1$. To establish a continuous-time Poisson-type limit of this evolution, we show the following: 
\begin{enumerate}[label=(\arabic*)]
    \item  \label{item:nochange} With probability $1-O(p^{-r_N})$, none of the singular numbers $\Pi^{(N)}_i(\tau), i \approx r_N$ change under the time step $\tau \mapsto \tau+1$ (and in fact, we see this is true for all $i \geq  r_N$ as well).
    \item For each $i \approx r_N$, we show the probability $\Pi^{(N)}_i(\tau)$ jumps at a given time step is $c p^{-i} + O(p^{-2r_N})$ for $c$ independent of $i$ which we explicitly compute, in the case when $\Pi^{(N)}_i(\tau)$ is not equal to any other part of $\la(\tau)$, and otherwise is given by a slightly different formula since multiple parts may push one another. This leads to the jump rates of the continuous-time process seen in \Cref{thm:gen_limit}. \label{item:1jump}
    \item We show that the probability that more than one jump occurs among $i \approx r_N$ is $O(p^{-2r_N})$ and hence may be discounted.\label{item:2jump}
\end{enumerate}

This section contains the general Markov process theory portion of the proof of \Cref{thm:gen_limit}. We first state three lemmas about random matrices, which correspond to \ref{item:nochange}, \ref{item:1jump}, and \ref{item:2jump} of the above sketch and contain all of the needed hard computations, and then show how they imply \Cref{thm:gen_limit}. The proofs of the lemmas themselves will be deferred to \Cref{sec:transition_asymptotics}. Below we as usual fix a prime $p$ and let $t=1/p$, and will use the following convenient notations:

\begin{defi}\label{def:order_and_skew}
Define a partial order $\subset$ on $\sSig_{2\infty}$ by 
\begin{equation}
    \label{eq:partial_order}
    \nu \subset \kappa \quad \iff \quad \nu_i \leq \kappa_i \quad \text{ for all $i$.}
\end{equation}
For $\nu \subset \kappa$, we define
\begin{equation}
|\kappa/\nu| = \sum_{i \in \Z} \kappa_i-\nu_i \in \Z_{\geq 0} \cup \{\infty\},
\end{equation}
where in the sum we take the convention that $\infty - \infty = (-\infty) - (-\infty) = 0$ and $\infty - n = n - (-\infty) = \infty$ for all $n \in \Z$. 
\end{defi}

\begin{defi}\label{def:finite_trans_mat}
Let $d \in \N$ and let $A$ be a random element of $\Mat_N(\Z_p)$. Then we define the Markov transition matrix on pairs $\kappa,\nu \in F_d(\Sig_N^+)$ by
\begin{equation}
P_{A,d}(\nu,\kappa) := \Pr(F_d(\SN(A B)) = \kappa),
\end{equation}
where $B$ is any matrix with independent of $A$ with left-$\GL_N(\Z_p)$-invariant distribution and deterministic singular numbers $\SN(B) = \nu$. Here $F_d$ is as in \Cref{def:trunc_map}.
\end{defi}

\begin{lemma}\label{thm:kappa=nu_case}
Let $(r_N)_{N \in \N}$ be a sequence with $r_N \leq N$ and $r_N \to \infty$, and for each $N \in \N$ let $A^{(N)}$ be a random matrix in $\Mat_N(\Z_p)$ with $\Pr( A^{(N)} \in \GL_N(\Z_p)) < 1$ and
\begin{equation}
\Pr(\corank(A^{(N)} \pmod{p}) \geq r_N - j) = o(c_N^{-1}) \quad \quad \text{ for all $j \geq 0$}
\end{equation}
where
\begin{equation}
c_N^{-1} := \E[\bbone(\corank(A^{(N)} \pmod{p}) \leq r_N)(t^{r_N-\corank(A^{(N)}\pmod{p})} - t^{r_N})].
\end{equation}
Fix $d \in \N$ and let $P_{A^{(N)},d}$ be as in \Cref{def:finite_trans_mat}. Then as $N \to \infty$,
\begin{equation}
P_{A^{(N)},d}(\nu^{(N)},\nu^{(N)}) = 1 - \frac{t^{(\nu^{(N)})'_d-r_N+1} - t^{N-r_N+1}}{1-t}c_N^{-1} + o(c_N^{-1}).
\end{equation}
Furthermore, for any $L \in \Z$ the implied constant in $o(c_N^{-1})$ is uniform over all $\nu^{(N)} \in F_d(\Sig_N^+)$ with $(\nu^{N})'_d \geq L+r_N$.
\end{lemma}

\begin{rmk}
Note that we do not require the asymptotic in \Cref{thm:kappa=nu_case} and below to be uniform over choices of the distribution of $A^{(N)}$ or the sequence $(r_N)_{N \in \N}$ which we assume to be fixed. Also, we will not need the uniformity of implied constants for \Cref{thm:gen_limit}, but we show it because we believe it may be useful for later work.
\end{rmk}

\begin{lemma}\label{thm:kappa-nu=1_case}
Assume the same setup as \Cref{thm:kappa=nu_case}. Then for any sequence of pairs $\nu^{(N)},\kappa^{(N)} \in F_d(\Sig_N^+)$ with $\nu^{(N)}\prec \kappa^{(N)}$ and $|\kappa^{(N)}/\nu^{(N)}| = 1$,
\begin{equation}\label{eq:1box_asymp_wts}
P_{A^{(N)},d}(\nu^{(N)},\kappa^{(N)}) = t^{j(N)} \frac{1-t^{m_{\nu_{j(N)}}(\nu)}}{1-t} c_N^{-1} + o(c_N^{-1})
\end{equation}
where $j(N)$ is the unique index such that $\kappa^{(N)}_{j(N)+r_N} = \nu^{(N)}_{j(N)+r_N}+1$. Furthermore, for any $L \in \Z$ the implied constant is uniform over all $\nu^{(N)} \in F_d(\Sig_N^+)$ with $(\nu^{N})'_d \geq L+r_N$.
\end{lemma}

\begin{lemma}\label{thm:kappa-nu_geq_2_case}
Assume the same setup as \Cref{thm:kappa=nu_case}. Then
\begin{equation}
P_{A^{(N)},d}(\nu^{(N)},\{\kappa \in F_d(\sSig_N): \kappa \supset \nu^{(N)} \text{ and }|\kappa/\nu^{(N)}| \geq 2\}) = o(c_N^{-1}).
\end{equation}
Furthermore, for any $L \in \Z$ the implied constant is uniform over all $\nu^{(N)} \in F_d(\Sig_N^+)$ with $(\nu^{N})'_d \geq L+r_N$.
\end{lemma}

Now we show \Cref{thm:gen_limit} conditional on the above lemmas. As a technical convenience, we work with slightly different prelimit processes.

\begin{defi}\label{def:modified_prelimit_proc}
In the setting of \Cref{thm:gen_limit}, define the (shifted) discrete-time singular number process $\tPi^{(N)}(\tau) = (\tPi^{(N)}(\tau)_i)_{i \in \Z}, \tau \in \Z_{\geq 0}$ on $\sSig_{2\infty}$ by
\begin{equation}
\tPi^{(N)}(\tau)_i = \begin{cases}
    \infty & i < 1-r_N \\
    \Pi^{(N)}(\tau)_{i+r_N} & 1-r_N \leq i \leq N-r_N \\
    \mu_i & i > N-r_N
\end{cases}.
\end{equation}
Define the continuous-time version by
\begin{equation}
\tL^{(N)}(T) = (\tL_i^{(N)}(T))_{i \in \Z} = \tPi^{(N)}(\floor{c_N T}).
\end{equation} 
\end{defi}

In other words, $\tPi^{(N)}$ agrees with $\Pi^{(N)}$ on all coordinates $i \leq N-r_N$, and all later coordinates are the same as those of $\mu$ and do not change as time $T$ increases. The process $\tL^{(N)}(T)$ has the advantage that $F_d(\tL^{(N)}(0)) = F_d(\mu)$ whenever $N$ is large enough so that $\SN(B^{(N)})$ has at least one part $\geq d$ (of course, the fact that this is true for large $N$ requires the hypothesis $\lim_{n \to \infty} \mu_{-n} = \infty$ of \Cref{thm:gen_limit}), hence $F_d(\tL^{(N)}(T))$ and $F_d(\Pois^{\mu,2\infty}(T))$ have the same initial condition. 

We now wish to prove the desired convergence in finite-dimensional distribution by analyzing the transition matrix and generator, respectively, of the discrete-time process $F_d(\tPi^{(N)}(\tau)), \tau =0,1,\ldots$ and the continuous-time process $F_d(\Pois^{\mu,2\infty}(T))$. For these considerations it is natural to consider a restricted state space, $\Sigma(d,\mu)$, which we now define. 

\begin{defi}\label{def:Sigma}
Recalling the notation $\nu/\mu$ and $\nu \supset \mu$ from \Cref{def:order_and_skew}, for any $\mu \in \sSig_{2\infty}$ and $d \geq 1$ we define
\begin{equation}
\Sigma(d,\mu) := \{\nu \in F_d(\sSig_{2\infty}): \nu \supset F_d(\mu), |\nu/F_d(\mu)| < \infty\}.
\end{equation}
\end{defi}

\begin{lemma}\label{thm:pois_on_Sigma}
For $\mu \in \sSig_{2\infty}$ with $\lim_{n \to \infty} \mu_{-n} = \infty$ and any $d \in \N$, the Markov process $F_d(\Pois^{\mu,2\infty}(T))$ remains on $\Sigma(d,\mu)$ for all time with probability $1$, and its transition matrix $Q$ (restricted to $\Sigma(d,\mu)$) is upper-triangular with respect to the partial order $\subset$ of \Cref{def:order_and_skew}.
\end{lemma}
\begin{proof}
By \Cref{thm:compute_Q}, the sum of transition rates of $F_d(\Pois^{\mu,2\infty})$ out of any state is bounded above by the sum of transition rates out of state $\mu$, which is $(t^{i_0}-t^{N+1})/(1-t)$ and hence finite. Therefore with probability $1$, $F_d(\Pois^{\mu,2\infty}(T))$ stays on $\Sigma(d,\mu)$. Upper-triangularity follows from the explicit definition in \Cref{thm:compute_Q}. 
\end{proof}

\begin{lemma}\label{thm:F_d(Lambda)_markov}
In the setting of \Cref{def:modified_prelimit_proc}, for any $d \in \N$, $F_d(\tPi^{(N)}(\tau))$ is a Markov chain. Furthermore, the set $\sSig_{2\infty} \setminus \Sigma(d,\mu)$ is absorbing for this process, so it projects to a Markov process on $\Sigma(d,\mu) \cup \{\aleph\}$ by identifying all states in $\sSig_{2\infty} \setminus \Sigma(d,\mu)$ with $\aleph$. Finally, the transition matrix $\tP_N$ of this Markov process is upper-triangular with respect to the partial order $\subset$ of \Cref{def:order_and_skew}.
\end{lemma}
\begin{proof}
The diagonal entries of the Smith normal form of any $\tA \in \Mat_N(\Z/p^d\Z)$ will lie in $\{1,p,\ldots,p^{d-1},0\}$, and so we define $\SN(\tA) \in \Sig_N$ to have all parts in $\{0,1,\ldots,d\}$, where all $0$ entries in the diagonal of the Smith normal form correspond to parts $d$. It is then clear that for any $A \in \Mat_N(\Z_p)$, 
\begin{equation}
    \label{eq:sn_mod}
    F_d(\SN(A)) = \SN(A \pmod{p^d}).
\end{equation}
Since $A^{(N)}_{\floor{c_N T}} \cdots A^{(N)}_1 B^{(N)} \pmod{p^d}$ is a product of independent matrices over $\Z/p^d\Z$, $F_d(\tPi^{(N)}(\tau))$ 
\begin{equation}
 F_d(\tPi^{(N)}(\tau)) = s^{r_N} \circ \iota(\SN(A^{(N)}_{\tau} \cdots A^{(N)}_1 B^{(N)} \pmod{p^d}))
\end{equation}
is a Markov process. Because $\tPi^{(N)}(\tau)$ is the same as above except on coordinates $i > N-r_N$, which do not evolve in time under either process, $\tPi^{(N)}(\tau)$ is also a Markov process. Clearly $F_d(\tPi^{(N)}(\tau))$ has upper-triangular transition matrix with respect to $\supset$, since multiplying matrices over $\Z_p$ can only increase their singular numbers. Hence if it ever leaves $\Sigma(d,\mu)$, it will not return, so it projects to a Markov process on $\Sigma(d,\mu) \cup \{\aleph\}$.
\end{proof}

Note that if $d > \lim_{n \to \infty} \mu_{-n}$, then $F_d(\tPi^{(N)}(\tau))_i = d > \mu_i$ for all sufficiently large negative $i$, hence in fact $F_d(\tPi^{(N)}(\tau))$ lives on $\sSig_{2\infty} \setminus \Sigma(d,\mu)$. When $d \leq \lim_{n \to \infty} \mu_{-n}$, however, it is not hard to see that $F_d(\tPi^{(N)}(\tau))$ will remain on $\Sigma(d,\mu)$ with probability $1$, though we find this fact as a consequence of later statements rather than explicitly deriving it. Finally, we prove the desired convergence for $F_d(\tL^{(N)}(T_i))$, from which \Cref{thm:gen_limit} is an easy consequence.

\begin{prop}\label{thm:gen_limit_F_d}
In the setting and notation of \Cref{thm:gen_limit},
\begin{equation}
    F_d(\tL^{(N)}(T)) \xrightarrow{N \to \infty} F_d(\Pois^{\mu,2\infty}(T))
\end{equation}
in finite-dimensional distribution.
\end{prop}

\begin{proof}[Proof of \Cref{thm:gen_limit_F_d}, assuming \Cref{thm:kappa=nu_case}, \Cref{thm:kappa-nu=1_case}, and \Cref{thm:kappa-nu_geq_2_case}]
We must show for any sequence of times $0 \leq T_1 < \ldots < T_k$ that 
\begin{equation}\label{eq:genlimit_wts}
    (F_d(\tL^{(N)}(T_i))_{1 \leq i \leq k} \to (F_d(\Pois^{\mu,2\infty}(T_i)))_{1 \leq i \leq k}
\end{equation}
weakly as $N \to \infty$. It follows from \Cref{thm:compute_Q} that
\begin{multline}
    \label{eq:joint_Q_exp}
    \Pr(\Pois^{\mu,2\infty}(T_i) = \nu^{(i)} \text{ for all }1 \leq i \leq k) \\ 
    = (e^{T_1Q})(F_d(\mu),\nu^{(1)}) (e^{(T_2-T_1)Q})(\nu^{(1)},\nu^{(2)}) \cdots (e^{(T_k-T_{k-1})Q})(\nu^{(k-1)},\nu^{(k)})
\end{multline}
when all $\nu^{(i)}$ lie in $\Sigma(d,\mu)$, and \eqref{eq:joint_Q_exp} is $0$ otherwise. Hence to show \eqref{eq:genlimit_wts}, by \eqref{eq:joint_Q_exp} we must show 
\begin{equation}
\label{eq:trans_to_semigroup}
\tP_N^{\floor{c_N T_i} - \floor{c_N T_{i-1}}}(\nu^{(i-1)},\nu^{(i)}) \xrightarrow{N \to \infty} (e^{(T_i-T_{i-1})Q})(\nu^{(i-1)},\nu^{(i)}),
\end{equation}
where $\tP_N$ is the transition matrix as defined in \Cref{thm:F_d(Lambda)_markov}. 

For technical reasons we wish to modify the state space to make it finite, so let $\tP_{N,i}$ and $Q_i$ be the restrictions of $\tP_N$ and $Q$ to the finite poset interval $[\nu^{(i-1)},\nu^{(i)}] \subset \Sigma(d,\mu)$. Then by upper-triangularity (see \Cref{thm:F_d(Lambda)_markov} and \Cref{thm:pois_on_Sigma} respectively), 
\begin{align}
\tP_{N,i}^{\floor{c_N T_i} - \floor{c_N T_{i-1}}}(\nu^{(i-1)},\nu^{(i)}) &= \tP_N^{\floor{c_N T_i} - \floor{c_N T_{i-1}}}(\nu^{(i-1)},\nu^{(i)}) \\ 
(e^{(T_i-T_{i-1})Q_i})(\nu^{(i-1)},\nu^{(i)}) &= (e^{(T_i-T_{i-1})Q})(\nu^{(i-1)},\nu^{(i)}).
\end{align} 
This implies that in order to show \eqref{eq:trans_to_semigroup}, it suffices to show 
\begin{equation}
\label{eq:trans_to_semigroup_finite}
\tP_{N,i}^{\floor{c_N T_i} - \floor{c_N T_{i-1}}}(\nu^{(i-1)},\nu^{(i)}) \xrightarrow{N \to \infty} (e^{(T_i-T_{i-1})Q_i})(\nu^{(i-1)},\nu^{(i)}). 
\end{equation}
The latter is an equality of finite matrices, and because they are finite it suffices to show 
\begin{equation}
    \label{eq:id_plus_small}
    \tP_{N,i} = I + c_N^{-1}Q_i + o(c_N^{-1}).
\end{equation}
(In fact, our reason for restricting to finite poset intervals was to be able to reduce to proving \eqref{eq:id_plus_small}). For any $\eta,\kappa \in [\nu^{(i-1)},\nu^{(i)}] \subset \Sigma(d,\mu)$, we have
\begin{equation}\label{eq:rephrase_transition_probs}
\tP_{N,i}(\eta,\kappa) = \tP_N(\eta,\kappa) = P_{A^{(N)}_i,d}((\eta_i)_{1-r_N \leq i \leq N-r_N},(\kappa_i)_{1-r_N \leq i \leq N-r_N})
\end{equation}
and $Q_i(\eta,\kappa) = Q(\eta,\kappa)$. We recall from \Cref{thm:compute_Q} that\footnote{Note that the $N$ in \Cref{thm:compute_Q} is not the same as the matrix size $N$ as we have used it elsewhere in this proof, but rather corresponds to the limit of $N-r_N$ in the notation of this proof.}
\begin{equation}\label{eq:recall_Q}
Q(\eta,\kappa) = 
\begin{cases}
-\frac{t^{\eta_d'+1} - t^{N+1}}{1-t} & \kappa = \eta \\
 \frac{t^i(1-t^{m_{\eta_i}(\eta)})}{1-t} & \text{there exists $i \in \Z$ such that }\kappa_j = \eta_j + \bbone(j = i) \text{ for all }j \in \Z \\ 
0 & \text{else}
\end{cases}.
\end{equation}
The asymptotics of \Cref{thm:kappa=nu_case}, \Cref{thm:kappa-nu=1_case}, and \Cref{thm:kappa-nu_geq_2_case} for \eqref{eq:rephrase_transition_probs}, which correspond to the three cases of \eqref{eq:recall_Q}, yield \eqref{eq:id_plus_small} in these three cases and hence complete the proof.
\end{proof}

\begin{proof}[Proof of \Cref{thm:gen_limit} from \Cref{thm:gen_limit_F_d}]
We claim that for any finite collection $I$ of coordinates, we have equality (in joint distribution) of projections
\begin{equation}\label{eq:pipi}
\pi_I(\tL^{(N)}(T)) = \pi_I(\La^{(N)}(T)).
\end{equation}
Note that $\tL^{(N)}$ and $\La^{(N)}$ only potentially differ at coordinates $i < 1-r_N$ and $i > N-r_N$. In the first case, any given finite set of coordinates will all be $>1-r_N$ for all sufficiently large $N$, since $r_N \to \infty$. In the second case, when $N - r_N \to \infty$ the same argument suffices. For choices of $r_N$ for which $N - r_N$ does not go to $\infty$, we must have $N - r_N = k$ for some $k$ and all sufficiently large $N$, and $\mu_{k+1} = -\infty$, in order to satisfy \eqref{eq:B_mu_limit}, and in this case $\tL^{(N)}$ and $\La^{(N)}$ agree on all coordinates $i > N-r_N$. This shows \eqref{eq:pipi}. It follows by \Cref{thm:weak_eq} that to show \Cref{thm:gen_limit}, it suffices to show the same statement with $\La^{(N)}$ replaced by $\tL^{(N)}$. This follows from \Cref{thm:gen_limit_F_d} by \Cref{thm:F_d_suffices}.
\end{proof}

\section{Nonasymptotic linear-algebraic bounds}\label{sec:nonasymp_lin_alg}

The purpose of this section is to prove three nonasymptotic statements about random matrix products, \Cref{thm:parts_don't_change}, \Cref{thm:single_box}, and \Cref{thm:two_unlikely}. In the next section we will use these to prove \Cref{thm:kappa=nu_case}, \Cref{thm:kappa-nu=1_case}, and \Cref{thm:kappa-nu_geq_2_case} respectively.

For the remainder of this section, we fix the following notation: Let $N \in \Z_{\geq 1}$ and $\mu,\la \in \sSig_N^+$ be fixed extended signatures with all parts nonnegative, let $A = (a_{ij})_{1 \leq i,j \leq N}$ be a Haar-distributed element of $\GL_N(\Z_p)$, and let $\nu = \SN(\diag(p^\la) A \diag(p^\mu))$ (a random element of $\sSig_N^+$). We write $\col_j(A) = (a_{ij})_{1 \leq i \leq N} \in \Z_p^N$ and similarly for other matrices. We also use $t=1/p$ without comment as usual, and recall the convention that $p^\infty := 0$ in the case when $\mu,\la$ have some entries equal to $\infty$. Finally, for any $\kappa \in \sSig_N^+$ we write 
\begin{equation}
    \len(\kappa) := \#\{1 \leq i \leq N: \kappa_i > 0\}
\end{equation}
with the usual convention $\infty > 0$.

\begin{lemma}\label{thm:parts_don't_change}
In the setting of this section, for any $1 \leq r \leq N$
\begin{equation}\label{eq:1box_lin_lemma}
\Pr(\nu_j = \mu_j \text{ for all }j \geq r) \geq \prod_{j=r}^N \frac{1 - t^{j-\len(\la)}}{1 - t^j} = \frac{(t;t)_{N-\len(\la)} (t;t)_{r-1}}{(t;t)_{r-1-\len(\la)} (t;t)_N}
\end{equation}
with equality if $\mu_{r-1} > \mu_{r}$. 
\end{lemma}

We remark that the right hand side of \eqref{eq:1box_lin_lemma} is $0$ when $r \leq \len(\la)$, in which case the lemma gives no information. 

\begin{lemma}\label{thm:single_box}
In the setting of this section, if $\len(\la) +1 \leq r \leq N$ is such that $\mu_{r-1} > \mu_r$ and $m_{\mu_r}(\mu) = m$, then 
\begin{align}\label{eq:single_box}
    \begin{split}
    (1-t^{r-\len(\la)})  C(r,N,m,\la) \leq \Pr(\nu_r = \mu_r+1 \text{ and } \nu_j = \mu_j \text{ for all } j > r) \leq C(r,N,m,\la) 
    \end{split}
\end{align}
where 
\begin{equation}
C(r,N,m,\la) := (t^{r-\len(\la)} -t^r) \frac{1-t^{m}}{1-t} \frac{(t;t)_{r-1}(t;t)_{N-\len(\la)}}{(t;t)_{N}(t;t)_{r-\len(\la)}}
\end{equation}
\end{lemma}

\begin{lemma}\label{thm:two_unlikely}
In the setting of this section, for any $r$ such that $\len(\la) + 1 \leq r \leq N$, 
\begin{equation}
\Pr\left(\sum_{j=r}^N \nu_j - \mu_j \geq 2\right) \leq 1 - \frac{(t;t)_{r-1}(t;t)_{N-\len(\la)}}{(t;t)_{N}(t;t)_{r-\len(\la)}}\left(1 -t^{r-\len(\la)} + t^{r-\len(\la)}(1-t^{N-r+1}) \frac{1-t^{\len(\la)}}{1-t}\right).
\end{equation}
\end{lemma}

Proving \Cref{thm:parts_don't_change}, \Cref{thm:single_box}, and \Cref{thm:two_unlikely} requires many auxiliary steps, which we begin proving. The following fact will be useful in proving \Cref{thm:parts_don't_change} and later.

\begin{lemma}\label{thm:nonsingular_submatrix}
In the setting of this section, let 
\begin{equation}
v_j = (a_{i,j})_{\len(\la) < i \leq N} \pmod{p} \in \F_p^{N-\len(\la)}.
\end{equation}
Then for any $\len(\la) < r \leq N$, the following implication and partial converse hold:
\begin{enumerate}
    \item If the set $\{v_j: r \leq j \leq N\}$ is linearly independent, then $\nu_j = \mu_j$ for all $j \geq r$.
    \item Suppose that additionally $\mu_{r-1} > \mu_r$. If $\nu_j = \mu_j$ for all $j \geq r$, then $\{v_j: r \leq j \leq N\}$ is linearly independent.
\end{enumerate}
\end{lemma}
\begin{proof}
Let
\begin{equation}
    A' := \diag(p^\la) A \diag(p^\mu) = (p^{\la_i+\mu_j}a_{i,j})_{1 \leq i,j \leq N}.
\end{equation}
First, suppose that $\{v_j: r \leq j \leq N\}$ is a linearly independent set. Then $\col_N(A')$ has an entry $p^{\mu_N} a_{i,N}$ with valuation $\mu_N$ (equivalently, $a_{i,N} \in \Z_p^\times$), and all other entries of $A'$ have valuation at least $\mu_N$. Hence by row and column operations we may cancel all entries in the same row and column as $p^{\mu_N}a_{i,N}$ and multiply its row by $a_{i,N}^{-1} \in \Z_p$, obtaining a matrix 
\begin{equation}\label{eq:matrix_after_first_op}
    \begin{pmatrix}
    p^{\la_1+\mu_1}\tilde{a}_{1,1} & \cdots &  p^{\la_1+\mu_{N-1}}\tilde{a}_{1,N-1} &  0 \\
    \vdots & \ddots & \vdots & \vdots  \\
    p^{\la_{N-1}+\mu_1}\tilde{a}_{N-1,1} & \cdots & p^{\la_{N-1}+\mu_{N-1}}\tilde{a}_{N-1,N-1} & 0 \\
    0 & \cdots & 0 & p^{\mu_N} 
    \end{pmatrix}
\end{equation}
for some $\tilde{a}_{i,j} \in \Z_p$.

By the linear independence assumption we may then find an entry $p^{\mu_{N-1}} a_{i,N-1}$ in the $(N-1)^{\text{st}}$ column of the matrix in \eqref{eq:matrix_after_first_op}, and cancel again, etc. Continuing, we obtain a matrix 
\begin{equation}\label{eq:up_to_r}
    \begin{pmatrix}
    p^{\la_1+\mu_1}\hat{a}_{1,1} & \cdots &  p^{\la_1+\mu_{r-1}}\hat{a}_{1,r-1} &   \\
    \vdots & \ddots & \vdots & \bold{0}_{(N-r+1) \times r} \\
    p^{\la_{r-1}+\mu_1}\hat{a}_{r-1,1} & \cdots & p^{\la_{r-1}+\mu_{r-1}}\hat{a}_{r-1,r-1} &  \\
     & \bold{0}_{r \times (N-r+1)} & & \diag(p^{\mu_r},\ldots,p^{\mu_N})
    \end{pmatrix},
\end{equation}
for some $\hat{a}_{i,j} \in \Z_p$, with the same singular numbers as $A'$. (Note that if $\len(\la)=N-1$ we are done after the first step \eqref{eq:matrix_after_first_op}). Its top left $(r-1) \times (r-1)$ submatrix $\hat{A} = (p^{\la_i + \mu_j}\hat{a}_{i,j})_{1 \leq i,j \leq r-1}$ lies in $p^{\mu_{r-1}}\Mat_{(r-1) \times (r-1)}(\Z_p)$, so all parts of $\SN(\hat{A})$ are at least $\mu_{r-1}$, hence 
\begin{equation}
\SN(A') = (\SN(\hat{A}),\mu_r,\mu_{r+1},\ldots,\mu_N).
\end{equation}
Now for the reverse implication, let us assume that $\mu_{r-1} > \mu_r$ and suppose that the set $\{(a_{i,j})_{\len(\la) < i \leq N}: r \leq j \leq N\}$ is not linearly independent modulo $p$. Let $k'$ be the largest index for which $\{v_{k'},\ldots,v_N\}$ is linearly dependent, and $k\leq k'$ the largest index such that additionally $\mu_k < \mu_{k-1}$. By the assumption $\mu_{r-1} > \mu_r$ it follows that $k \geq r$. We claim $\nu_k > \mu_k$. By definition of $k'$, there must exist a relation 
\begin{equation}
c_{k'}v_{k'} + \ldots + c_N v_N = \vec{0} \in \F_p^{N-\len(\la)}
\end{equation}
with $c_{k'} \neq 0$, so without loss of generality take $c_{k'}=1$. Letting $C_j$ be a lift of $c_j$ to $\Z_p$ for each $j$, we therefore have that 
\begin{equation}
\val_p\left(\col_{k'}(A') + C_{k'+1} \col_{k'+1}(A') +\ldots + C_N \col_N(A')\right) \geq \mu_{k'} + 1.
\end{equation}

Let $A''$ be the matrix obtained from $A'$ via column operations replacing $\col_{k'}(A')$ by $\col_{k'}(A') + C_{k'+1} \col_{k'+1}(A') +\ldots + C_N \col_N(A')$. Then $A''$ has the same singular numbers as $A'$, and furthermore it satisfies $\val_p(\col_j(A'')) \geq \mu_k+1$ for $j=1,\ldots,k-1,k'$. It follows immediately that $\nu$ has at least $k$ parts $\geq \mu_k+1$, and since $\nu_j \geq \mu_j$ for all $j$ this implies $\nu_k \geq \mu_k+1$. This proves the reverse implication.
\end{proof}

The forward direction of \Cref{thm:nonsingular_submatrix} is a corollary of the following inequality, though we are not sure how one would establish the backward direction through the considerations used in the proof below.

\begin{lemma}\label{thm:parts_bound}
Let $\la,\mu \in \sSig_N$, $1 \leq r \leq N$ with $\len(\la) < r$, and $k \geq 0$. Then for any $B = (b_{ij})_{1 \leq i,j \leq N} \in \GL_N(\Z_p)$,
\begin{equation}
    \label{eq:sn_bound_submatrix}
    \left| \SN\left((b_{ij})_{\subalign{\len(\la) &< i \leq N \\ r &\leq j \leq N}}\right)\right| \geq \sum_{j = r}^N \SN(p^\la B p^\mu)_j - \mu_j.
\end{equation}
\end{lemma}

\begin{proof}[Proof of \Cref{thm:parts_bound}]
Let 
\begin{equation}
B' = (b_{ij})_{\subalign{\len(\la) &< i \leq N \\ r &\leq j \leq N} }.
\end{equation}
Since $r > \len(\la)$, the bottom-right $(N-\len(\la)) \times (N-r+1)$ corner of $p^\la B p^\mu$ is just $B' p^{(\mu_r,\ldots,\mu_N)}$. Hence by \Cref{thm:minor_increase_sns},
\begin{equation}
    \label{eq:take_minor_in_proof}
\sum_{j = r}^N \SN(p^\la B p^\mu)_j \leq \sum_{j = 1}^{N-r+1} \SN(B' p^{(\mu_{r},\ldots,\mu_N)})_{j} = |\SN(B' p^{(\mu_{r},\ldots,\mu_N)})|.
\end{equation}
By \Cref{thm:multiply_smaller_dimension}, 
\begin{equation}
    \label{eq:apply_mult}
    |\SN(B' p^{(\mu_{r},\ldots,\mu_N)})| = |\SN(B')| + \sum_{j=r}^N \mu_j.
\end{equation}
Combining \eqref{eq:take_minor_in_proof} with \eqref{eq:apply_mult} and subtracting $\sum_{j=r}^N \mu_j$ from both sides yields \eqref{eq:sn_bound_submatrix}.
\end{proof}

\begin{proof}[Proof of {\Cref{thm:parts_don't_change}}]
In light of \Cref{thm:nonsingular_submatrix}, we must show (in the notation of that result) that
\begin{equation}\label{eq:li_suffices}
    \Pr(\{v_j: r \leq j \leq N\}\text{ is linearly independent}) = \prod_{j=r}^N \frac{p^j - p^{\len(\la)}}{p^j - 1}.
\end{equation}
When $r \leq \len(\la)$ this reduces easily to $0=0$, so suppose $r > \len(\la)$. The columns $\col_r(A),\ldots,\col_N(A)$ are chosen independently from the Haar measure, conditioned to be linearly independent modulo $p$. This implies that the columns $\col_j(A) \pmod{p}$ are chosen from the uniform measure on $\F_p^N$, conditionally on being linearly independent. Therefore
\begin{equation}\label{eq:lin_indep_count_matrices}
    \Pr(\{v_j: r \leq j \leq N\}\text{ is linearly independent}) = \frac{\#S_2}{\#S_1} 
\end{equation}
where 
\begin{align*}
    S_1 &:= \{B = (b_{i,j}) \in \Mat_{N \times (N-r+1)}(\F_p): B \text{ is full rank}\} \\
    S_2 &:= \{B = (b_{i,j}) \in \Mat_{N \times (N-r+1)}(\F_p): (b_{i,j} \bbone_{i > \len(\la)})_{\subalign{&1 \leq i \leq N \\ &1 \leq j \leq N-r+1}} \text{ is full rank}\} \subset S_1. \\
\end{align*}
Computing the number of possible first columns, then second columns, etc. of $B$ yields 
\begin{equation}\label{eq:simple_li}
    \#S_1 = (p^N-1) \cdots (p^N-p^{N-r})
\end{equation}
Since the condition 
\begin{equation}
(b_{i,j} \bbone_{i > \len(\la)})_{\subalign{&1 \leq i \leq N \\& 1 \leq j \leq N-r+1}} \text{ is full rank}
\end{equation}
is independent of the upper submatrix $(b_{i,j})_{\subalign{&1 \leq i \leq \len(\la) \\ &1 \leq j \leq N-r+1}}$, counting the number of possible first, second, etc. columns we have 
\begin{equation}\label{eq:harder_li}
    \#S_2 = (p^N - p^{\len(\la)}) \cdots (p^N - p^{\len(\la) + N-r})
\end{equation}
Computing the RHS of \eqref{eq:lin_indep_count_matrices} via \eqref{eq:simple_li} and \eqref{eq:harder_li} yields \eqref{eq:li_suffices} and hence completes the proof. 
\end{proof}

For the proofs of \Cref{thm:single_box} and \Cref{thm:two_unlikely} we will use the following auxiliary computations over $\F_p$.

\begin{lemma}\label{thm:corank_pt1}
For $0 \leq r \leq k \leq n$,
\begin{equation}\label{eq:rank_comp}
    \#\{B \in \Mat_{n \times k}(\F_q): \rank(B) = r\} = q^{rn+rk-r^2} \frac{(q^{-1};q^{-1})_n (q^{-1};q^{-1})_k}{(q^{-1};q^{-1})_r (q^{-1};q^{-1})_{n-r} (q^{-1};q^{-1})_{k-r}}.
\end{equation}

\end{lemma}
\begin{proof}
The group $\GL_n(\F_q) \times \GL_k(\F_q)$ acts on $\Mat_{n \times k}(\F_q)$ by $(x,y) \cdot B = xBy^{-1}$, and by Smith normal form orbits are parametrized by their coranks. Letting $B_r \in \Mat_{n \times k}(\F_q)$ be the matrix with $ii\tth$ entry $1$ for $1 \leq i \leq r$ and all other entries $0$, we therefore have 
\begin{equation}\label{eq:need_stabilizer}
    \text{LHS\eqref{eq:rank_comp}} = \frac{\# \GL_n(\F_q) \times \GL_k(\F_q)}{\#\Stab(B_r)}.
\end{equation}
Explicitly,
\begin{align*}
\#\Stab(B_r) = 
\left\{\left(\begin{pmatrix}
X & Y \\
0 & Z
\end{pmatrix}, \begin{pmatrix}
X & 0 \\
P & Q
\end{pmatrix}\right): \substack{X \in \GL_r(\F_q),  Z \in \GL_{n-r}(\F_q), Q \in \GL_{k-r}(\F_q) \\ P \in \Mat_{(k-r) \times r}(\F_q),Y \in \Mat_{r \times (n-r)}(\F_q)} \right\},
\end{align*}
therefore 
\begin{equation}
\#\Stab(B_r) = (q^r-1)\cdots(q^r-q^{r-1}) q^{r(n-r)} (q^{n-r}-1) \cdots (q^{n-r} - q^{n-r-1}) q^{r(k-r)} (q^{k-r} - 1) \cdots (q^{k-r} - q^{k-r-1}).
\end{equation}
Combining this with 
\begin{equation}
\# \GL_n(\F_q) \times \GL_k(\F_q) = (q^n-1)\cdots (q^n-q^{n-1}) (q^k-1) \cdots (q^k - q^{k-1})
\end{equation}
and \eqref{eq:need_stabilizer} yields \eqref{eq:rank_comp}. 
\end{proof}

\begin{lemma}\label{thm:corank_pt2}
Let $d$ and $n \geq k$ be three nonnegative integers, let $B \in \Mat_{(n+d) \times k}(\F_q)$ be a uniformly random full-rank matrix, and let $B' \in \Mat_{n \times k}(\F_q)$ be its lower $n \times k$ submatrix. Then for any $0 \leq r \leq k$,
\begin{equation}\label{eq:trunc_rank}
    \Pr(\rank(B') = r) = q^{-(n-r)(k-r)} \frac{\sqbinom{d}{k-r}_{q^{-1}} \sqbinom{n}{r}_{q^{-1}}}{\sqbinom{n+d}{k}_{q^{-1}}}.
\end{equation}
\end{lemma}
\begin{proof}
We first compute 
\begin{equation}
\#\{B \in \Mat_{(n+d) \times k}(\F_q): \rank(B) = k, \rank(B') = r\}
\end{equation}
where $B'$ is the truncated matrix as in the statement. The number of possible $B'$ is computed in \Cref{thm:corank_pt1}, so for each $B'$ we must count the number of $d \times k$ matrices $B''$ such that 
\begin{equation}
\begin{pmatrix}
B'' \\ B'
\end{pmatrix} \in \Mat_{(n+d) \times k}(\F_q)
\end{equation}
is full rank. By change of basis, the number of such $B''$ is the same for any $B'$ of rank $r$, so without loss of generality take $B' = B_r \in \Mat_{n \times k}(\F_q)$, the matrix with $ii\tth$ entry $1$ for $1 \leq i \leq r$ and all other entries $0$. Then the first $r$ columns of $B''$ may be anything, and the last $k-r$ columns must be linearly independent, so there are 
\begin{equation}\label{eq:complete_to_full}
    q^{dr} (q^d-1) \cdots (q^d - q^{k-r-1})
\end{equation}
possibilities for $B'$. The result now follows by combining \eqref{eq:complete_to_full} with \Cref{thm:corank_pt1}, dividing by the number of full rank $(n+d) \times k$ matrices, and cancelling terms.
\end{proof}

\begin{rmk}
We note that \eqref{eq:trunc_rank} is a $q$-analogue of the probability that a uniformly random $k$-element subset $S \subset A \sqcup B$ has $\#S \cap B = r$, when $\#A = d$ and $\#B = n$.
\end{rmk}

\begin{lemma}\label{thm:prob_coker_1}
Let $A \in \GL_N(\Z_p)$ be distributed by the Haar measure and $A'$ be an $n \times m$ submatrix with $n \leq m \leq N$. Then 
\begin{equation}
\Pr(\SN(A') = (1,0,\ldots,0)) = t^{m-n+1} \frac{(t;t)_{N-m} (t;t)_m (t;t)_n (t;t)_{N-n}}{(t;t)_1 (t;t)_{N-m-1} (t;t)_{n-1} (t;t)_{m-n+1} (t;t)_N}.
\end{equation}
\end{lemma}
\begin{proof}
Follows immediately by combining Theorem 1.3(1) and Proposition 2.9 of \cite{van2020limits}. 
\end{proof}

We note that \Cref{thm:prob_coker_1} can also be established by a (longer) direct proof not going through the general results of \cite{van2020limits}.

\begin{proof}[Proof of {\Cref{thm:single_box}}]
First write
\begin{align}\label{eq:reduce_to_cond_prob}
\begin{split}
   &\Pr(\nu_r = \mu_r+1 \text{ and } \nu_j = \mu_j \text{ for all } j > r) \\
   &=\Pr(\nu_r = \mu_r+1 \text{ and } \nu_j = \mu_j \text{ for all }r < j < r+m| \nu_j=\mu_j \text{ for all }j \geq r+m) \\
   &\times \Pr(\nu_j=\mu_j \text{ for all }j \geq r+m).
\end{split}
\end{align}
The second term on the RHS is
\begin{equation}\label{eq:easy_term}
    \Pr(\nu_j=\mu_j \text{ for all }j \geq r+m) =\prod_{j=r+m}^N \frac{1 - t^{j-\len(\la)}}{1 - t^j} = \frac{(t;t)_{N - \len(\la)}}{(t;t)_{r+m-\len(\la)-1}} \frac{(t;t)_{r+m-1}}{(t;t)_N}
\end{equation}
by \Cref{thm:parts_don't_change}, so it suffices to compute the first term on the RHS of \eqref{eq:reduce_to_cond_prob}. By \Cref{thm:nonsingular_submatrix},
\begin{equation}
W(A) := (a_{i,j})_{\subalign{\len(\la)+1 \leq i \leq N \\ r+m \leq j \leq N}}
\end{equation}
is full rank modulo $p$ if and only if $\nu_j=\mu_j \text{ for all }j \geq r+m$. Therefore 
\begin{align} \label{eq:cond_W(A)}
\begin{split}
    &\Pr(\nu_r = \mu_r+1 \text{ and } \nu_j = \mu_j \text{ for all }r < j < r+m| \nu_j=\mu_j \text{ for all }j \geq r+m) \\
    &= \Pr(\nu_r = \mu_r+1 \text{ and } \nu_j = \mu_j \text{ for all }r < j < r+m| W(A) \text{ is full rank modulo $p$}). \\ 
\end{split}
\end{align}
We claim that 
\begin{equation}\label{eq:reduce_to_tI_cond}
\text{RHS\eqref{eq:cond_W(A)}} = \Pr(\nu_r = \mu_r+1 \text{ and } \nu_j = \mu_j \text{ for all }r < j < r+m| W(A) = \tI)
\end{equation}
where
\begin{equation}
\label{eq:def_J}
\tI = \begin{pmatrix} \bold{0}_{(r+m-\len(\la)-1) \times (N-(r+m)+1)} \\ I_{N-(r+m)+1} \end{pmatrix}.
\end{equation}
First note that any matrix $H \in \Mat_{(N-\len(\la)) \times (N-(r+m)+1)}(\Z_p)$ which is full-rank modulo $p$ is in the same $\GL_{N-\len(\la)}(\Z_p)$-orbit as $\tI$ (here we use that $\len(\la) < r$ and simply apply the necessary row operations to $H$). This, together with the explicit description of the Haar measure in \Cref{thm:haar_sampling}, implies that
\begin{equation}\label{eq:law_reduce_tI}
\Law(A|W(A) \text{ is full rank modulo $p$}) = \Law(BA| W(A) = \tI)
\end{equation}
where
\begin{equation}
B =  \begin{pmatrix} I & \bold{0} \\ \bold{0} & \tB \end{pmatrix}   \in \begin{pmatrix} I & \bold{0} \\ \bold{0} & \GL_{N-\len(\la)}(\Z_p)\end{pmatrix}
\end{equation}
and $\tB$ is Haar-distributed independent of $A$, because $B$ mixes $\tI$ to a matrix distributed under the additive Haar measure conditioned on being full rank. By \eqref{eq:law_reduce_tI},
\begin{align}
\label{eq:use_haar_W}
\begin{split}
\Law(\SN(p^\la A p^\mu)|W(A) \text{ is full rank modulo $p$}) &= \Law(\SN(p^\la BA p^\mu)|W(A) = \tI) \\ 
&= \Law(\SN(B p^\la A p^\mu)|W(A) = \tI) \\ 
&= \Law(\SN(p^\la A p^\mu)|W(A) = \tI),
\end{split}
\end{align}
which shows \eqref{eq:reduce_to_tI_cond}.

For convenience define $\tA = (\ta_{i,j})_{1 \leq i,j \leq N}$ to be a random element of $\GL_N(\Z_p)$ distributed by the Haar measure conditioned on $W(\tA) = \tI$, so that 
\begin{equation}
\label{eq:pass_to_tA}
\text{RHS\eqref{eq:reduce_to_tI_cond}} = \Pr(\SN(p^\la \tA p^\mu)_i = \mu_i + \bbone(i=r) \text{ for all }r \leq i \leq N).
\end{equation}
For any deterministic matrix $V = (v_{i,j})_{1 \leq i,j \leq N}$ with $W(V) = \tI$, first note that $\SN(p^\la V p^\mu)_i = \nu_i$ for all $r+m \leq i \leq N$ by \Cref{thm:nonsingular_submatrix} as before. We make the following additional claims, which will be used for our upper and lower bounds:
\begin{enumerate}[label=(\roman*)]
\item If \label{item:sn_to_sn}
\begin{equation}
\label{eq:hyp_suff}
\SN\left((v_{i,j})_{\substack{\len(\la)+1 \leq i \leq r+m-1 \\ r \leq j \leq r+m-1}}\right) = (1,0,\ldots,0)
\end{equation}
then
\begin{equation}
\label{eq:concl_suff}
\SN(p^\la V p^\mu)_i = \mu_i + \bbone(i=r) \text{ for all }r \leq i \leq r+m-1.
\end{equation}
\item \label{item:sn_to_corank} If \eqref{eq:concl_suff} holds, then 
\begin{equation}
\label{eq:concl_nec}
(v_{i,j})_{\substack{\len(\la)+1 \leq i \leq r+m-1 \\ r \leq j \leq r+m-1}} \pmod{p} \text{ has corank $\geq 1$.}
\end{equation}
\end{enumerate}
Let us first show \ref{item:sn_to_sn}, so suppose \eqref{eq:hyp_suff} holds. 

\begin{equation}\label{eq:first_mat_decomp}
p^\la V p^\mu = 
\begin{pmatrix}
p^\la M^{(1)} & p^\la M^{(2)} & p^\la M^{(3)} \\ 
M^{(4)} & M^{(5)} & 0 \\ 
M^{(6)} & M^{(7)} & \diag(p^{\mu_{r+m}},\ldots,p^{\mu_N})
\end{pmatrix},
\end{equation}
where $M^{(i)}$ are the appropriate submatrices of $Vp^\mu$ and $\SN(M^{(5)}) = (1,0[m-1])$. Here the blocks of \eqref{eq:first_mat_decomp} in first, second and third column have widths $r-1$, $m$ and $N-(r+m)+1$ respectively, and the blocks of the first, second, and third rows have heights $\len(\la)$, $r+m-\len(\la)-1$, and $N-(r+m)+1$ respectively.
Hence by further column operations which subtract units times powers $p^{(\mu_i-\mu_\ell)}, 1 \leq i \leq r+m-1,r+m \leq \ell \leq N$ times the $\ell\tth$ column from the $i\tth$ column, we obtain
\begin{equation}
\begin{pmatrix}
p^\la \tM^{(1)} & p^\la \tM^{(2)} & p^\la M^{(3)} \\ 
M^{(4)} & M^{(5)} & 0 \\ 
0 & 0 & \diag(p^{\mu_{r+m}},\ldots,p^{\mu_N})
\end{pmatrix}
\end{equation}
and because of the $p^{(\mu_i-\mu_\ell)}$ powers the matrices $\tM^{(1)}, \tM^{(2)}$ still lie in $\Mat_{\len(\la) \times (r-1)}(\Z_p) \diag(p^{\mu_1},\ldots,p^{\mu_{r-1}})$, $\Mat_{\len(\la) \times m}(\Z_p)\diag(p^{\mu_r},\ldots,p^{\mu_{r+m-1}})$ respectively. We clearly may further cancel to obtain
\begin{equation}\label{eq:final_matrix}
\begin{pmatrix}
p^\la \tM^{(1)} & p^\la \tM^{(2)} & 0 \\ 
M^{(4)} & M^{(5)} & 0 \\ 
0 & 0 & \diag(p^{\mu_{r+m}},\ldots,p^{\mu_N})
\end{pmatrix}.
\end{equation}

Note that since $\mu_r = \ldots = \mu_{r+m-1}$,
\begin{equation}
\label{eq:Ms_and_Vs}
M^{(5)}
= p^{\mu_r}\tV,
\end{equation}
where $\tV = (v_{i,j})_{\substack{\len(\la)+1 \leq i \leq r+m-1 \\ r \leq j \leq r+m-1}}$ is the matrix for which we have assumed $\SN(\tV) = (1,0[m-1])$. Hence 
\begin{equation}
\label{eq:SNM}
\SN(M^{(5)}) = (\mu_r+1,\mu_r[m-1]).
\end{equation}
By \Cref{thm:minor_increase_sns} and \eqref{eq:SNM},
\begin{equation}\label{eq:sns_upper}
\sum_{\ell=1}^m \SN\left((p^{\la_i+\mu_j}v_{i,j})_{\substack{1 \leq i \leq r+m-1 \\ r \leq j \leq r+m-1}}\right)_\ell \leq \sum_{\ell=1}^m \SN(M^{(5)})_\ell = m \mu_r + 1.
\end{equation}
Since the matrix $(p^{\la_i}v_{i,j})_{\substack{1 \leq i \leq r+m-1 \\ r \leq j \leq r+m-1}}$ is not full rank modulo $p$ by \eqref{eq:hyp_suff},
\begin{equation}
\left| \SN\left((p^{\la_i}v_{i,j})_{\substack{1 \leq i \leq r+m-1 \\ r \leq j \leq r+m-1}}\right)\right| \geq 1,
\end{equation}
hence by \Cref{thm:multiply_smaller_dimension} we have 
\begin{equation}
\text{LHS\eqref{eq:sns_upper}} \geq m \mu_r + 1,
\end{equation}
so in fact
\begin{equation}\label{eq:pin_down_sns}
\text{LHS\eqref{eq:sns_upper}} = m \mu_r + 1.
\end{equation}
Since every entry of $(p^{\la_i+\mu_j}v_{i,j})_{\substack{1 \leq i \leq r+m-1 \\ r \leq j \leq r+m-1}}$ is divisible by $p^{\mu_r}$, each singular number is at least $\mu_r$, and combining this with \eqref{eq:pin_down_sns} yields 
\begin{equation}\label{eq:sns_intermediate}
\SN\left((p^{\la_i+\mu_j}v_{i,j})_{\substack{1 \leq i \leq r+m-1 \\ r \leq j \leq r+m-1}}\right) = (\mu_r+1,\mu_r[m-1]).
\end{equation}
Since all entries in columns $1$ through $r-1$ of \eqref{eq:final_matrix} are divisible by $p^{\mu_{r-1}}$, \eqref{eq:sns_intermediate} together with the equivalence of $p^\la V p^\mu$ with \eqref{eq:final_matrix} imply \eqref{eq:concl_suff}. This shows \ref{item:sn_to_sn}.

Now we show \ref{item:sn_to_corank}, so suppose $V$ is such that \eqref{eq:concl_suff} holds. If $(v_{i,j})_{\substack{\len(\la)+1 \leq i \leq r+m-1 \\ r \leq j \leq r+m-1}} \pmod{p}$ were full-rank, then $(v_{i,j})_{\substack{\len(\la)+1 \leq i \leq N \\ r \leq j \leq N}} \pmod{p}$ would be full rank since $W(V) = \tI$ is full-rank. By \Cref{thm:nonsingular_submatrix} this would contradict the fact that $\SN(p^\la V p^\mu)_r = \mu_r+1$. Hence $(v_{i,j})_{\substack{\len(\la)+1 \leq i \leq r+m-1 \\ r \leq j \leq r+m-1}} \pmod{p}$ has corank $k \geq 1$.

Suppose for the sake of contradiction that $k \geq 2$. First bring $p^\la V p^\mu$ into the form \eqref{eq:final_matrix} by row and column operations, and note that none of the operations used affect the block $(v_{i,j})_{\substack{\len(\la)+1 \leq i \leq r+m-1 \\ r \leq j \leq r+m-1}}$. Hence by our assumption, $p^{-\mu_r}M^{(5)} \pmod{p}$ has corank $k$, i.e. rank $m-k$ (using that $r \geq \len(\la)+1$ so the matrix is taller than it is wide). Reducing the matrix in \eqref{eq:final_matrix} modulo $p^{\mu_r}$, all blocks other than $M^{(5)}$ and $\diag(p^{\mu_{r+m}},\ldots,p^{\mu_N})$ are $0$. Putting this block matrix in Smith normal form, there can only be $(m-2)+(N-r-m+1)$ nonzero entries, so there are only $N-r-1$ singular numbers which are $\leq \mu_r$. Hence $\nu_r$ and $\nu_{r+1}$ are both $\geq \mu_r+1$, contradicting \eqref{eq:concl_suff}. Therefore $k=1$ as desired, proving \ref{item:sn_to_corank}.



Using \ref{item:sn_to_sn} and \ref{item:sn_to_corank} for the lower and upper bounds respectively, we have
\begin{align}\label{eq:1box_sandwich_bounds}
\begin{split}
&\Pr\left(\SN\left((\ta_{i,j})_{\substack{\len(\la) < i < r+m \\ r \leq j < r+m}}\right) = (1,0[m-1])\right) \\ 
&\leq \Pr(\nu_r = \mu_r+1 \text{ and } \nu_j = \mu_j \text{ for all }r < j < r+m| W(A) = \tI) \\ 
&\leq \Pr\left(\corank\left((\ta_{i,j})_{\substack{\len(\la) < i < r+m \\ r \leq j < r+m}} \pmod{p}\right) = 1\right).
\end{split}
\end{align}
By applying \Cref{thm:corank_pt2} with $r=m-1,n+d=N,k=m,n=r+m-\len(\la)-1$ we obtain
\begin{align}
\label{eq:corank1prob}
\begin{split}
\Pr\left(\corank\left((\ta_{i,j})_{\substack{\len(\la) < i < r+m \\ r \leq j < r+m}} \pmod{p}\right) = 1\right) &=  t^{r-\len(\la)} \frac{\sqbinom{\len(\la)}{1}_t \sqbinom{r+m-\len(\la)-1}{m-1}_t}{\sqbinom{r+m-1}{m}_t} \\ 
&= t^{r-\len(\la)} \frac{(1-t^{\len(\la)})(1-t^m)}{1-t} \frac{(t;t)_{r+m-\len(\la)-1} (t;t)_{r-1} }{(t;t)_{r-\len(\la)} (t;t)_{r+m-1}} \\
\end{split}
\end{align}
By applying \Cref{thm:prob_coker_1} with $r+m-1, r+m-1-\len(\la), m$ substituted for $N,m,n$ respectively, 
\begin{equation}\label{eq:1boxprob}
    \text{LHS\eqref{eq:1box_sandwich_bounds}} = (t^{r-\len(\la)}-t^r) \frac{1-t^m}{1-t} \frac{(t;t)_{r-1}(t;t)_{r+m-\len(\la)-1}}{(t;t)_{r+m-1}(t;t)_{r-\len(\la)-1}}.
\end{equation}
Hence 
\begin{align}\label{eq:1box_sandwich_bounds_explicit}
\begin{split}
&(t^{r-\len(\la)}-t^r) \frac{1-t^m}{1-t} \frac{(t;t)_{r-1}(t;t)_{r+m-\len(\la)-1}}{(t;t)_{r+m-1}(t;t)_{r-\len(\la)-1}} \\ 
&\leq \Pr(\nu_r = \mu_r+1 \text{ and } \nu_j = \mu_j \text{ for all }r < j < r+m| W(A) = \tI) \\ 
&\leq t^{r-\len(\la)} \frac{(1-t^{\len(\la)})(1-t^m)}{1-t} \frac{(t;t)_{r+m-\len(\la)-1} (t;t)_{r-1} }{(t;t)_{r-\len(\la)} (t;t)_{r+m-1}}.
\end{split}
\end{align}
Combining the reduction \eqref{eq:reduce_to_cond_prob} with the computation of \eqref{eq:easy_term} and the bound on the conditional probability coming from combining \eqref{eq:cond_W(A)}, \eqref{eq:reduce_to_tI_cond}, and \eqref{eq:1box_sandwich_bounds_explicit} completes the proof.
\end{proof}

\begin{proof}[Proof of {\Cref{thm:two_unlikely}}]
For any matrix $B = (b_{i,j})_{1 \leq i,j \leq N} \in \GL_N(\Z_p)$, we define $B' = \left((b_{i,j})_{\subalign{&\len(\la)+1 \leq i \leq N \\ &r \leq j \leq N}}\right)$ as before. Since
\begin{equation}
    \sum_{j=r}^N \nu_j - \mu_j  \leq |\SN(A')|
\end{equation}
by \Cref{thm:parts_bound}, it follows that if $\sum_{j=r}^N \nu_j - \mu_j \geq 2$ then $|\SN(A')| \geq 2$. Hence
\begin{equation}\label{eq:reduce_2box_to_minor}
    \Pr\left(\sum_{j=r}^N \nu_j - \mu_j \geq 2\right) \leq 1 - \Pr(|\SN(A')| \leq 1).
\end{equation}
Since $|\SN(A')|=0$ if and only if $A' \pmod{p}$ is full rank, \Cref{thm:corank_pt2} yields that
\begin{equation}\label{eq:sn0}
    \Pr(|\SN(A')| = 0) = \frac{(t;t)_{r-1}(t;t)_{N-\len(\la)}}{(t;t)_{N}(t;t)_{r-1-\len(\la)}}.
\end{equation}
By \Cref{thm:prob_coker_1}, 
\begin{equation}\label{eq:sn1}
    \Pr(|\SN(A')| = 1) = \frac{(t;t)_{r-1}(t;t)_{N-\len(\la)}}{(t;t)_{N}(t;t)_{r-\len(\la)}} \left(t^{r-\len(\la)}(1-t^{N-r+1}) \frac{1-t^{\len(\la)}}{1-t}\right).
\end{equation}
Combining \eqref{eq:reduce_2box_to_minor} with \eqref{eq:sn0} and \eqref{eq:sn1} completes the proof.
\end{proof}

\section{Asymptotics of matrix product transition probabilities}\label{sec:transition_asymptotics}

In this section, we use the nonasymptotic bounds of the previous section to establish asymptotics for the matrix product process stated earlier as \Cref{thm:kappa=nu_case}, \Cref{thm:kappa-nu=1_case}, and \Cref{thm:kappa-nu_geq_2_case}. The technical work of this section essentially amounts to computing the relevant terms of bounds which were left as prelimit explicit formulas in the previous section, with the additional complication of randomizing those bounds over the singular numbers of one of the matrices; we also phrase everything in terms of truncated signatures $F_d(\nu)$, which was not done in the previous section. As usual, we use $t=1/p$ in formulas. 

\begin{defi}\label{def:unif_o}
In the proofs of \Cref{thm:kappa=nu_case}, \Cref{thm:kappa-nu=1_case} and \Cref{thm:kappa-nu_geq_2_case}, we write $\ou(\cdot)$ to indicate any quantity which, for any fixed $L$ as in those lemmas, is $o(\cdot)$ as $N \to \infty$ with constants uniform over all $\nu^{(N)} \in F_d(\Sig_N^+)$ with $(\nu^{N})'_d \geq L+r_N$.
\end{defi}

\begin{proof}[Proof of \Cref{thm:kappa=nu_case}]
To simplify notation, let
\begin{equation}\label{eq:def:j0}
j_0 = j_0(N) := (\nu^{(N)})_d' -r_N + 1 = \min\{i: \nu^{(N)}_{r_N+i} < d\}.
\end{equation}
By hypothesis, $j_0 \geq L$. Let $U$ be an element of $\GL_N(\Z_p)$ with Haar distribution independent of $A^{(N)}$, and $\la \in \bSig_N^+$. Then
\begin{align}\label{eq:0prob_cases}
\begin{split}
&\Pr(F_d(\SN(A^{(N)}U \diag(p^{\nu^{(N)}}))) = F_d(\nu^{(N)})|\SN(A^{(N)}) = \la) \\ 
&= \Pr(\SN(A^{(N)} \diag(p^{\nu^{(N)}}))_i = \nu^{(N)}_i \text{ for all }j_0 + r_N \leq i \leq N|\SN(A^{(N)}) = \la) \\ 
&= \begin{cases}
\prod_{j=j_0+r_N}^N \frac{1-t^{j-\ell}}{1-t^j} & 0 \leq \ell < j_0+r_N \\ 
0 &  j_0+r_N \leq \ell 
\end{cases}
\end{split}
\end{align}
where $\ell = \len(\la)$. For notational convenience, here and in the rest of the proof we define the random variable $X_N := \len(\SN(A^{(N)}))$. Taking an expectation over $X_N$ in \eqref{eq:0prob_cases} yields
\begin{align}\label{eq:explicit_step_do_nothing}
\begin{split}
    \Pr(F_d(\SN(A^{(N)} U \diag(p^{\nu^{(N)}}))) = F_d(\nu^{(N)})) = \E\left[\bbone(X_N < j_0+r_N)\prod_{j=j_0+r_N}^N \frac{1-t^{j-X_N}}{1-t^j}\right].
\end{split}
\end{align}
Note that \eqref{eq:explicit_step_do_nothing} depends on $\nu^{(N)}$ only through $j_0$, so to establish uniform asymptotics over $\nu^{(N)}$ we simply need them to be uniform over $j_0$. To show \Cref{thm:kappa=nu_case}, we therefore must show 
\begin{equation}
    \label{eq:bound_do_nothing}
    \E\left[\bbone(X_N < j_0+r_N)\prod_{j=j_0+r_N}^N \frac{1-t^{j-X_N}}{1-t^j} \right] = 1 - \frac{t^{j_0}-t^{N-r_N+1}}{1-t} c_N^{-1} + \ou(c_N^{-1}) 
\end{equation}
(recall the notation $\ou$ from \Cref{def:unif_o} and the definition $c_N := (\E[\bbone(X_N \leq r_N)(t^{r_N-X_N} - t^{r_N})])^{-1}$). Since
\begin{equation}
\Pr(X_N \geq j_0+r_N) \leq \Pr(X_N \geq L + r_N) = \ou(c_N^{-1})
\end{equation}
by hypothesis, we may write
\begin{equation}
\label{eq:exp_constant_ind_trick}
\frac{\E[\bbone(X_N < j_0+r_N) (1-t^{j_0+r_N}) \cdots (1-t^N)]}{ (1-t^{j_0+r_N}) \cdots (1-t^N)} = 1 + \ou(c_N^{-1}),
\end{equation}
and using this we rearrange \eqref{eq:bound_do_nothing} to obtain that it is equivalent to show
\begin{multline}\label{eq:bound_do_nothing2}
\frac{ \E[\bbone(X_N < j_0+r_N)\left((1-t^{j_0+r_N}) \cdots (1-t^N) - (1-t^{j_0+r_N-X_N})\cdots(1-t^{N-X_N})\right)]}{(1-t^{j_0+r_N}) \cdots (1-t^N)} \\ 
    =\frac{t^{j_0}-t^{N-r_N+1}}{1-t}c_N^{-1} + \ou(c_N^{-1}).
\end{multline}
This is what we will show.

We write 
\begin{align}\label{eq:qbinom_1box}
\begin{split}
    &\E[\bbone(X_N < j_0+r_N)((1-t^{j_0+r_N}) \cdots (1-t^N) - (1-t^{j_0+r_N-X_N})\cdots(1-t^{N-X_N}))] \\
    &= \E\left[\bbone(X_N < j_0+r_N)\sum_{j=0}^{N-r_N-j_0+1} (-1)^j \sqbinom{N-r_N-j_0+1}{j}_t t^{\binom{j}{2}+j(j_0+r_N)} (1-t^{-j X_N})\right] \\ 
    &=  \sum_{j=1}^{N-r_N-j_0+1} (-1)^{j+1} \sqbinom{N-r_N-j_0+1}{j}_t t^{\binom{j}{2}+j(j_0+r_N)} \E[\bbone(X_N < j_0+r_N)(1-t^{-j X_N})]
\end{split}
\end{align}
by expanding both factors inside the expectation via the $q$-binomial theorem and consolidating term-by-term. 

Note that the $j=0$ term of \eqref{eq:qbinom_1box} is $0$ since $1-t^{-j X_N} = 0$. The contribution of the $j=1$ term of \eqref{eq:qbinom_1box} to \eqref{eq:bound_do_nothing2} is
\begin{equation}
    \label{eq:j=1}
    \frac{1}{(1-t^{j_0+r_N}) \cdots (1-t^N)} \frac{t^{j_0} - t^{N-r_N+1}}{1-t}c_N^{-1}  = 
        c_N^{-1}\left(\frac{t^{j_0}-t^{N-r_N+1}}{1-t} + \ou(1)\right)
\end{equation}
where we use that $r_N \to \infty$ so $(1-t^{j_0+r_N}) \cdots (1-t^N) \to 1$. The asymptotic \eqref{eq:j=1} is uniform over $\nu^{(N)}$ satisfying our hypotheses, since it depends on $\nu^{(N)}$ only through $j_0$, and is uniform over $j_0 \geq L$. 

Hence to prove \eqref{eq:bound_do_nothing2} it now suffices to show 
\begin{equation}
    \label{eq:jgeq2terms}
    \sum_{j=2}^{N-r_N-j_0+1} (-1)^j \sqbinom{N-r_N-j_0+1}{j}_t t^{\binom{j}{2}+j(j_0+r_N)} \E[\bbone(X_N < j_0+r_N)(1-t^{-j X_N})] = \ou(c_N^{-1}),
\end{equation}
where we have used the fact that ${(1-t^{j_0+r_N}) \cdots (1-t^N)} = 1+\ou(1)$ uniformly over $j_0 \geq L$ to remove the denominator of \eqref{eq:bound_do_nothing2}. We rewrite the asymptotic \eqref{eq:jgeq2terms} which we wish to show as 
\begin{equation}
    \label{eq:jgeq2terms2}
    -\sum_{j=2}^{N-r_N-j_0+1} (-1)^j \sqbinom{N-r_N-j_0+1}{j}_t t^{\binom{j}{2}+j j_0} \frac{\E[\bbone(X_N < j_0+r_N)(t^{j(r_N-X_N)} - t^{jr_N})]}{\E[\bbone(X_N \leq r_N)(t^{r_N-X_N} - t^{r_N})]} = \ou(1).
\end{equation}
To show \eqref{eq:jgeq2terms2}, it suffices to show that for all $\delta > 0$,
\begin{equation}\label{eq:j_quotient}
t^{jj_0}\frac{\E[\bbone(X_N < j_0+r_N)(t^{j(r_N-X_N)} - t^{jr_N})]}{\E[\bbone(X_N \leq r_N)(t^{r_N-X_N} - t^{r_N})]} < \delta \quad \quad \text{ for all $j \geq 2$}
\end{equation}
for all $N$ sufficiently large independent of $j$. Since both numerator and denominator in \eqref{eq:j_quotient} are $0$ when $X_N=0$, by clearing factors of $\Pr(X_N > 0)$ we have
\begin{align}\label{eq:pass_to_cond_exp}
\begin{split}
t^{jj_0}\frac{\E[\bbone(X_N < j_0+r_N)(t^{j(r_N-X_N)} - t^{jr_N})]}{\E[\bbone(X_N \leq r_N)(t^{r_N-X_N} - t^{r_N})]} &= t^{jj_0}\frac{\E[\bbone(X_N < j_0+r_N)(t^{j(r_N-X_N)} - t^{jr_N})|X_N > 0]}{\E[\bbone(X_N \leq r_N)(t^{r_N-X_N} - t^{r_N})| X_N > 0]} \\ 
&\leq t^{jj_0}\frac{\E[\bbone(\tX_N < j_0+r_N)t^{j(r_N-\tX_N)}]}{\E[\bbone(\tX_N \leq r_N)(t^{r_N-\tX_N} - t^{r_N})]},
\end{split}
\end{align}
where to simplify notation we let $\tX_N$ be a random variable with
\begin{equation}
\Law(\tX_N) = \Law(X_N | X_N > 0).
\end{equation}
For any $b \geq 0$, 
\begin{multline}
\label{eq:remove_big_Xn}
\text{RHS\eqref{eq:pass_to_cond_exp}} \leq t^{jj_0}\frac{\Pr(r_N + j_0 - b < \tX_N < r_N + j_0)t^{-j j_0}}{\E[\bbone(\tX_N \leq r_N)(t^{r_N-\tX_N} - t^{r_N})]} \\  + t^{jj_0}\frac{\E[\bbone(\tX_N \leq r_N + j_0 - b)t^{j(r_N-\tX_N)}]}{\E[\bbone(\tX_N \leq r_N)(t^{r_N-\tX_N} - t^{r_N})]}.
\end{multline}
The first term in \eqref{eq:remove_big_Xn} is 
\begin{equation}
\label{eq:throw_away_prob_term}
\frac{\Pr(r_N + j_0 - b < \tX_N < r_N + j_0)}{\E[\bbone(\tX_N \leq r_N)(t^{r_N-\tX_N} - t^{r_N})]} \leq \frac{\Pr(\tX_N > r_N + L - b)}{\E[\bbone(\tX_N \leq r_N)(t^{r_N-\tX_N} - t^{r_N})]},
\end{equation}
which by the hypothesis \eqref{eq:lambda_not_large} is $\ou(1)$ (it is uniform over $j_0$, since $j_0$ does not appear). 

Note next that for any $j$ and any function $f: \R \to \R$ with $f([1,r_N+j_0-b]) \subset [0,1]$,
\begin{equation}\label{eq:lbound}
\E[\bbone(\tX_N \leq r_N+j_0 - b)f(\tX_N) t^{j(r_N+j_0-\tX_N)}] \leq t^{(j-1)b}\E[\bbone(\tX_N \leq r_N+j_0 - b)t^{r_N+j_0-\tX_N}]
\end{equation}
because all nonzero terms come from values of $\tX_N$ with $r_N+j_0-\tX_N \geq b$, and so
\begin{equation}
    t^{j(r_N+j_0-\tX_N)} \leq t^{(j-1)b} \cdot t^{r_N+j_0-\tX_N}
\end{equation}
with probability $1$. 

Applying \eqref{eq:lbound} to the numerator and a trivial bound to the denominator of the second term of \eqref{eq:remove_big_Xn} yields
\begin{align}\label{eq:exp-term}
\begin{split}
\frac{\E[\bbone(\tX_N \leq r_N +j_0- b)t^{j(r_N+j_0-\tX_N)}]}{\E[\bbone(\tX_N \leq r_N)(t^{r_N-\tX_N} - t^{r_N})]} &\leq \frac{t^{(j-1)b}\E[\bbone(\tX_N \leq r_N +j_0 - b)t^{r_N+j_0-\tX_N}]}{(1-t)\E[\bbone(\tX_N \leq r_N)t^{r_N-\tX_N}]} \\ 
&\leq \frac{t^{(j-1)b}}{1-t}t^{j_0} .
\end{split}
\end{align}
Hence for any $\delta > 0$, by choosing $b$ so that $\frac{t^{(j-1)b}}{1-t}t^L < \delta/2$, we have that \eqref{eq:j_quotient} holds for all $N$ large enough that the left hand side of \eqref{eq:throw_away_prob_term} is $<\delta/2$. As we had previously reduced to \eqref{eq:j_quotient}, this completes the proof.
\end{proof}

We will prove \Cref{thm:kappa-nu_geq_2_case} before \Cref{thm:kappa-nu=1_case} since the former is needed for the latter.

\begin{proof}[Proof of \Cref{thm:kappa-nu_geq_2_case}]
Let $\tX_N$ be a random variable with $\Law(\tX_N) = \Law(X_N | X_N > 0)$ as before, so 
\begin{equation}\label{eq:remove_prob0}
\E[\bbone(\tX_N \leq r_N)(t^{r_N-\tX_N}-t^{r_N})] = \frac{c_N^{-1}}{\Pr(X_N > 0)} 
\end{equation}
(recalling that $X_N \geq 0$ always, and $X_N > 0$ with positive probability by hypothesis). Using the same notation $j_0=j_0(N)$ defined in \eqref{eq:def:j0}, for the proof it suffices to show
\begin{equation}\label{eq:0prob_2part_j0}
\Pr\left(\left. \sum_{i = j_0 + r_N}^N \SN(A^{(N)} U \diag(p^{\nu^{(N)}}))_i - \nu^{(N)}_i \geq 2 \right| X_N > 0 \right) = \ou(\E[\bbone(\tX_N \leq r_N)(t^{r_N-\tX_N}-t^{r_N})]),
\end{equation}
where $U$ is a Haar element of $\GL_N(\Z_p)$ independent of $A^{(N)}$. For each $0 < x < r_N+j_0$, by applying \Cref{thm:two_unlikely} with $r=j_0+r_N$, $\la = \SN(A^{(N)})$, and $\mu = \nu^{(N)}$ in the notation of that result, we have
\begin{multline} \label{eq:length_cond_2box_bound}
\Pr\left(\left. \sum_{i = j_0 + r_N}^N \SN(A^{(N)} U \diag(p^{\nu^{(N)}}))_i - \nu^{(N)}_i \geq 2 \right| X_N =x \right) \\ 
\leq 
\left(1 - \frac{(t;t)_{j_0+r_N-1}(t;t)_{N-x}}{(t;t)_{N}(t;t)_{j_0+r_N-x}}\left(1 - t^{j_0+r_N-x} + t^{j_0+r_N-x}(1-t^{N-j_0+r_N+1}) \frac{1-t^{x}}{1-t}\right)\right)
\end{multline}
Fix an integer $b \geq 1$ independent of $N$, and let $N$ be large enough so that $r_N + j_0 - b > 0$ holds (this holds for all sufficiently large $N$ since $r_N \to \infty$ and $j_0(N) \geq L$). Then taking a (conditional, given $X_N > 0$) expectation of \eqref{eq:length_cond_2box_bound} when $0<x\leq r_N+j_0-b$ and naively bounding when $x > r_N+j_0-b$ yields
\begin{align}
    \label{eq:0prob_2part_formula}
    \begin{split}
    &\Pr\left(\left.\sum_{i = j_0 + r_N}^N \SN(A^{(N)} U \diag(p^{\nu^{(N)}}))_i - \nu^{(N)}_i \geq 2 \right| X_N > 0  \right) 
    \\
    &\leq \Pr(\tX_N > r_N + j_0 - b ) + 
    \E\left[\bbone(\tX_N \leq r_N+j_0 - b) \left(1 - \frac{(t;t)_{j_0+r_N-1}(t;t)_{N-\tX_N}}{(t;t)_{N}(t;t)_{j_0+r_N-\tX_N}}\right.\right. \\ 
    &\times \left.\left.\left(1 - t^{j_0+r_N-\tX_N} + t^{j_0+r_N-\tX_N}(1-t^{N-j_0+r_N+1}) \frac{1-t^{\tX_N}}{1-t}\right)\right)\right].
    \end{split}
\end{align}
We first rewrite the expression inside the expectation on the right-hand side of \eqref{eq:0prob_2part_formula} (without the indicator function) as 
\begin{align}
\begin{split}\label{eq:qbinom_2box}
 & 1 - \frac{(t;t)_{j_0+r_N-1}(t;t)_{N-\tX_N}}{(t;t)_{N}(t;t)_{j_0+r_N-\tX_N}}\left(1 - t^{j_0+r_N-\tX_N}+ t^{j_0+r_N-\tX_N}(1-t^{N-j_0+r_N+1}) \frac{1-t^{\tX_N}}{1-t}\right) \\ 
&= \frac{(t;t)_{j_0 + r_N- 1}}{(t;t)_{j_0 + r_N  - \tX_N}}\left( \frac{1-t^{j_0+r_N}}{\prod_{i=0}^{\tX_N-1} (1-t^i \cdot t^{j_0 + r_N -\tX_N +1})} \right. \\ 
&\left. - \frac{1}{\prod_{i=0}^{\tX_N-1} (1-t^i \cdot t^{N - \tX_N + 1})}\left(1 - t^{j_0+r_N-\tX_N}+ (t^{j_0 + r_N - \tX_N}-t^{N - \tX_N + 1})\frac{1-t^{\tX_N}}{1-t}\right)\right) \\ 
&=  \frac{(t;t)_{j_0 + r_N- 1}}{(t;t)_{j_0 + r_N  - \tX_N}} \sum_{\ell=0}^\infty \sqbinom{\tX_N-1+\ell}{\ell}_t  \left(t^{\ell(j_0+r_N-\tX_N+1)}(1-t^{j_0+r_N}) \right. \\ 
&\left.-t^{\ell(N-\tX_N+1)}\left(1 - t^{j_0+r_N-\tX_N}+ (t^{j_0 + r_N - \tX_N}-t^{N - \tX_N + 1})\frac{1-t^{\tX_N}}{1-t}\right)\right) 
\end{split}
\end{align}
where in the second equality, we have expanded both of the $1/(\prod \cdots)$ terms into infinite sums by the $q$-binomial theorem (here it is important that $\tX_N \in \Z_{\geq 1}$) and then combined the sums. We further split the sum to write
\begin{equation}
\label{eq:sum_of_S}
\text{RHS\eqref{eq:qbinom_2box}} = S_1+S_2+S_3,
\end{equation}
where we define
\begin{multline}\label{eq:S1}
S_1 = \frac{(t;t)_{j_0 + r_N- 1}}{(t;t)_{j_0 + r_N  - \tX_N}}\left(1-t^{j_0+r_N }- 1 + t^{j_0+r_N-\tX_N} - (t^{j_0+r_N-\tX_N} - t^{N-\tX_N+1})\frac{1-t^{\tX_N}}{1-t} \right. \\
\left. +\sqbinom{\tX_N}{1}_t(t^{j_0+r_N-\tX_N+1} - t^{N-\tX_N+1})\right) = 0
\end{multline}
(the $\ell=0$ term in \eqref{eq:qbinom_2box} together with a part of the $\ell=1$ term chosen so that they exactly cancel),
\begin{multline}\label{eq:S2}
S_2 = \left(-t^{2(j_0+r_N)-\tX_N+1}  -t^{N-\tX_N+1}\left( - t^{j_0+r_N-\tX_N}+ (t^{j_0 + r_N - \tX_N}-t^{N - \tX_N + 1})\frac{1-t^{\tX_N}}{1-t}\right) \right) \\ 
\times \frac{(t;t)_{j_0 + r_N- 1}}{(t;t)_{j_0 + r_N  - \tX_N}}\sqbinom{\tX_N}{1}_t
\end{multline}
(the rest of the $\ell=1$ term), and 
\begin{multline}
\label{eq:S3}
S_3=\frac{(t;t)_{j_0 + r_N- 1}}{(t;t)_{j_0 + r_N - \tX_N}} \sum_{\ell=2}^\infty \sqbinom{\tX_N - 1 + \ell}{\ell}_{t} \Bigg( t^{\ell(j_0 + r_N - \tX_N)}(1-t^{j_0+r_N})  \\ 
- t^{\ell(N-\tX_N +1)}\Bigg(1-t^{j_0+r_N-\tX_N}+t^{j_0 + r_N - \tX_N}(1-t^{N-j_0-r_N + 1})\frac{1-t^{\tX_N}}{1-t}\Bigg)\Bigg)
\end{multline}
(the rest of the sum, i.e. the $\ell \geq 2$ terms). We have observed that $S_1 = 0$, and now argue that $S_2$ and $S_3$ are small asymptotically. The $t$-Pochhammer prefactor 
\begin{equation}\frac{(t;t)_{j_0 + r_N- 1}}{(t;t)_{j_0 + r_N - \tX_N}}\end{equation}
lies in $[0,1]$, and the summands making up $S_3$ satisfy the bound
\begin{multline}\label{eq:3bound}
\left|t^{\ell(j_0 + r_N - \tX_N)}(1-t^{j_0+r_N})- t^{\ell(N-\tX_N +1)}\left(1-t^{j_0+r_N-\tX_N} + (t^{j_0 + r_N - \tX_N}-t^{N - \tX_N + 1})\frac{1-t^{\tX_N}}{1-t}\right)\right|  \\ 
\times \sqbinom{\tX_N - 1 + \ell}{\ell}_{t} \leq 3\sqbinom{\tX_N - 1 + \ell}{\ell}_{t} t^{\ell(j_0 + r_N - \tX_N)},
\end{multline}
hence 
\begin{align}\label{eq:3bound2}
\begin{split}
|S_3| &\leq 3 \sum_{\ell=0}^\infty \sqbinom{\tX_N - 1 + \ell}{\ell}_{t} t^{\ell(j_0 + r_N - \tX_N)} \\ 
&= 3 \frac{1}{\prod_{i=0}^{\tX_N-1}1-t^i \cdot t^{j_0+r_N - \tX_N}} \\ 
&\leq \frac{3}{(t;t)_{\infty}}
\end{split}
\end{align}
for all $\tX_N < j_0 + r_N$ (using that $N+1 \geq j_0+r_N$). Similarly to \eqref{eq:3bound}, we may split $S_2$ into three terms with a power of $t$ at least $2(j_0+r_N-\tX_N)$, yielding
\begin{equation}
\label{eq:S2bound}
|S_2| \leq \frac{3}{1-t}\sqbinom{\tX_N }{1}_{t} t^{2(j_0 + r_N - \tX_N)}.
\end{equation}
Below we use shorthand 
\begin{equation}\label{eq:bbone_b}
\bbone_b := \bbone(\tX_N \leq r_N+j_0 - b)
\end{equation}
to minimize equation overflow. Multiplying \eqref{eq:qbinom_2box} by $\bbone_b$, taking an expectation, and applying Fubini's theorem (the hypotheses of which we checked in \eqref{eq:3bound2}) to pull it inside the sum yields
\begin{align}\label{eq:sum_ell_2}
\begin{split}
&\E\left[\bbone_b \times \left(1 - \frac{(t;t)_{j_0+r_N-1}(t;t)_{N-\tX_N}}{(t;t)_{N}(t;t)_{j_0+r_N-\tX_N}} \left(1 -t^{j_0+r_N} + t^{j_0+r_N-\tX_N}(1-t^{N-j_0+r_N+1}) \frac{1-t^{\tX_N}}{1-t}\right)\right)\right] \\
&= \E[\bbone_b S_2] + \sum_{\ell=2}^\infty \E\left[ \bbone_b\frac{(t;t)_{j_0 + r_N- 1}}{(t;t)_{j_0 + r_N - \tX_N}} \sqbinom{\tX_N - 1 + \ell}{\ell}_{t} \Bigg(t^{\ell(j_0 + r_N - \tX_N)}(1-t^{j_0+r_N})\right. \\ 
&-\left. t^{\ell(N-\tX_N +1)}\Bigg(1-t^{j_0+r_N-\tX_N}+(t^{j_0 + r_N - \tX_N}-t^{N - \tX_N + 1})\frac{1-t^{\tX_N}}{1-t}\Bigg)\Bigg) \right],
\end{split}
\end{align}
where we have also used that $S_1=0$ to throw away those corresponding terms of the sum. To argue that the remaining terms are small, we first note that by the first bound in \eqref{eq:3bound2}, the naive bound
\begin{equation}\label{eq:qbin_bound}
 \sqbinom{\tX_N - 1 + \ell}{\ell}_{t} \leq \frac{1}{(t;t)_\infty},
\end{equation}
the bound \eqref{eq:S2bound} on $S_2$, the nonnegativity of the arguments of all expectations, we have
\begin{equation}\label{eq:lbound_terms_1}
|\text{RHS\eqref{eq:sum_ell_2}}| \leq \frac{3}{(1-t)(t;t)_\infty}\E[\bbone_bt^{2(r_N+j_0-\tX_N)}] + \frac{3}{(1-t)(t;t)_\infty}\sum_{\ell=2}^\infty \E[\bbone_bt^{\ell(j_0+r_N-\tX_N)}].
\end{equation}
Applying \eqref{eq:lbound} and collecting terms yields 
\begin{align}\label{eq:sumbound_l}
\begin{split}
\text{RHS\eqref{eq:lbound_terms_1}} &\leq \E[\bbone_bt^{r_N+j_0-\tX_N}] \frac{3}{(1-t)(t;t)_\infty}\left(t^b+\sum_{\ell=2}^\infty t^{(\ell-1)b} \right) \\ 
&= C't^b \E[\bbone_bt^{r_N+j_0-\tX_N}]
\end{split}
\end{align}
for an explicit constant $C'$ independent of $b$ and $N$. If $j_0-b > 0$ then (recalling the shorthand $\bbone_b$ from \eqref{eq:bbone_b}) we have
\begin{align}\label{eq:j0-b>0}
\begin{split}
\text{RHS\eqref{eq:sumbound_l}} &\leq C't^b\left(t^{L} \E[\bbone(\tX_N \leq r_N)t^{r_N-\tX_N}] + \E[\bbone(r_N < \tX_N \leq r_N+j_0-b)t^{r_N+j_0-\tX_N}]\right) \\ 
& \leq C't^b\left(t^{L} \E[\bbone(\tX_N \leq r_N)t^{r_N-\tX_N}] + t^b \Pr(\tX_N > r_N)\right),
\end{split}
\end{align}
while if $j_0 - b \leq 0$ then $r_N+j_0-b \leq r_N$ and hence $\bbone_b \leq \bbone(\tX_N \leq r_N)$, so
\begin{equation}\label{eq:j0-b<0}
\text{RHS\eqref{eq:sumbound_l}} \leq C' t^b \E[\bbone(\tX_N \leq r_N)t^{r_N-\tX_N+j_0}] \leq C' t^b t^{L} \E[\bbone(\tX_N \leq r_N)t^{r_N-\tX_N}] \leq \text{RHS\eqref{eq:j0-b>0}}.
\end{equation}
Hence the bound \eqref{eq:j0-b>0} actually holds independent of $b$ and $j_0$ (for all $j_0 \geq L$). Since $\tX_N \geq 1$,
\begin{equation}
t^{r_N-\tX_N} \leq \frac{1}{1-t}(t^{r_N-\tX_N} - t^{r_N}),
\end{equation}
and combining with \eqref{eq:j0-b>0} yields
\begin{equation}\label{eq:final_sumbound}
\text{RHS\eqref{eq:sumbound_l}} \leq \frac{C't^L}{1-t}t^b \E[\bbone(\tX_N \leq r_N)(t^{r_N-\tX_N} - t^{r_N})] + C't^{2b}\Pr(\tX_N > r_N).
\end{equation}
Substituting \eqref{eq:final_sumbound} into \eqref{eq:0prob_2part_formula} and multiplying through by $\Pr(X_N > 0)$ to convert the $\tX_N$ back to $X_N$ yields
\begin{align}
    \label{eq:0prob_2part_bound}
    \begin{split}
    &\Pr\left(\sum_{i = j_0 + r_N}^N \SN(A^{(N)} U \diag(p^{\nu^{(N)}}))_i - \nu^{(N)}_i \geq 2   \right) 
    \\
    &\leq \Pr(X_N > r_N + j_0 - b ) + \frac{C't^L}{1-t}t^b c_N^{-1} + C't^{2b}\Pr(\tX_N > r_N). 
    \end{split}
\end{align}
To show the right hand side is small, we let $b$ depend on $N$ as follows. Since 
\begin{equation}
 \Pr(X_N > r_N + j_0 - b ) \leq \Pr(X_N > r_N + L - b) = \ou(c_N^{-1})
\end{equation}
for any fixed $b$, by a diagonalization argument there exists a slowly growing sequence $b=b(N)$ not depending on $j_0$ such that 
\begin{equation}\label{eq:prob_b(N)_small}
\Pr(X_N > r_N + j_0 - b(N)) = \ou(c_N^{-1}).
\end{equation}
Since \eqref{eq:0prob_2part_bound} holds for any $b > 0$, it holds with $b$ replaced by $b(N)$. Then the first term on the right hand side is $\ou(c_N^{-1})$ by \eqref{eq:prob_b(N)_small}, the second term is $\ou(c_N^{-1})$ because $b(N) \to \infty$, and the third term is $\ou(c_N^{-1})$ as well by hypothesis. Hence 
\begin{equation}
\Pr\left(\sum_{i = j_0 + r_N}^N \SN(A^{(N)} U \diag(p^{\nu^{(N)}}))_i - \nu^{(N)}_i \geq 2   \right)  = \ou(c_N^{-1}),
\end{equation}
so we are done.
\end{proof}


\begin{proof}[Proof of \Cref{thm:kappa-nu=1_case}]
First, since $|\kappa^{(N)}/\nu^{(N)}| = 1$ we may rewrite
\begin{align}\label{eq:split_by_2box}
\begin{split}
&\text{LHS\eqref{eq:1box_asymp_wts}} = \Pr\left(\SN(A^{(N)} U \diag(p^{\nu^{(N)}}))_i = \kappa_i^{(N)} \text{ for all } i \geq j+r_N  \right)- \\
&\Pr\left(\SN(A^{(N)} U \diag(p^{\nu^{(N)}}))_i  = \kappa_i^{(N)} \text{ for all } i \geq j+r_N \text{, and } |F_d(\SN(A^{(N)} U \diag(p^{\nu^{(N)}})))/\kappa^{(N)}| \geq 2\right) ,
\end{split}
\end{align}
where again $U$ is a Haar element of $\GL_N(\Z_p)$ independent of $A^{(N)}$.
By trivially bounding the second term in \eqref{eq:split_by_2box} by 
\begin{equation}
\Pr(|F_d(\SN(A^{(N)} U \diag(p^{\nu^{(N)}})))/\kappa^{(N)}| \geq 2)
\end{equation}
and applying \Cref{thm:kappa-nu_geq_2_case}, \eqref{eq:split_by_2box} yields
\begin{align}\label{eq:split_by_2box2}
\begin{split}
&\text{LHS\eqref{eq:1box_asymp_wts}} = \Pr\left(\SN(A^{(N)} U \diag(p^{\nu^{(N)}}))_i = \kappa_i^{(N)} \text{ for all } i \geq j+r_N  \right) + \ou(c_N^{-1})
\end{split}
\end{align} 
uniformly over $\nu^{(N)}$ as in the statement, where here and in the rest of the proof we write $j$ for the index $j(N)$ in the statement. For any integer $b \geq \max(0,L)$, we may therefore write
\begin{align}\label{eq:diff2_3term_split}
\begin{split}
&\text{LHS\eqref{eq:1box_asymp_wts}} =\ou(c_N^{-1}) \\
&+ \Pr\left(\SN(A^{(N)} U \diag(p^{\nu^{(N)}}))_i = \kappa_i^{(N)} \text{ for all } i \geq j+r_N \text{ and }X_N > r_N+L-b  \right) \\ 
&+ \Pr\left(\SN(A^{(N)} U \diag(p^{\nu^{(N)}}))_i = \kappa_i^{(N)} \text{ for all } i \geq j+r_N \text{ and }X_N \leq r_N+L-b  \right) 
\end{split}
\end{align}
For the second summand in \eqref{eq:diff2_3term_split} a naive bound gives
\begin{multline}
\Pr\left(\SN(A^{(N)} U \diag(p^{\nu^{(N)}}))_i = \kappa_i^{(N)} \text{ for all } i \geq j+r_N \text{ and }X_N > r_N+L-b  \right)\\ 
 \leq \Pr(X_N > r_N + L - b).
\end{multline}
Substituting this and applying \Cref{thm:single_box} (with $r=r_N+j, \len(\la)=X_N$) to the last summand in \eqref{eq:diff2_3term_split}, we obtain upper and lower bounds
\begin{align}
\label{eq:use_thm_single_box}
\begin{split}
& \E\left[(1-t^{j+r_N-X_N}) \cdot \bbone(X_N \leq r_N+L-b)t^j\frac{1-t^m}{1-t}(t^{r_N-X_N} - t^{r_N}) \frac{(t;t)_{j+r_N-1} (t;t)_{N-X_N}}{(t;t)_N (t;t)_{j+r_N-X_N}}\right] \\ 
& \quad \quad \quad+\Pr(X_N > r_N + L-b) + \ou(c_N^{-1})\\ 
&\leq \text{LHS\eqref{eq:1box_asymp_wts}} \\ 
&\leq \E\left[\bbone(X_N \leq r_N+L-b)t^j\frac{1-t^m}{1-t}(t^{r_N-X_N} - t^{r_N}) \frac{(t;t)_{j+r_N-1} (t;t)_{N-X_N}}{(t;t)_N (t;t)_{j+r_N-X_N}}\right]\\
& \quad \quad \quad +\Pr(X_N > r_N + L-b) + \ou(c_N^{-1})
\end{split}
\end{align}
for any $b > 0$ (this condition is required since \Cref{thm:single_box} only applies when $X_N < j+r_N$, and we have assumed only that $j \geq L$). We will show both bounds have the same asymptotic to obtain the asymptotic for \eqref{eq:1box_asymp_wts}. The difference between the two bounds in \eqref{eq:use_thm_single_box} is 
\begin{align}\label{eq:bound_discrepency}
\begin{split}
&\E\left[\bbone(X_N \leq r_N+L-b)\frac{1-t^m}{1-t}t^{2(j+r_N-X_N)}(1-t^{X_N}) \frac{(t;t)_{j+r_N-1} (t;t)_{N-X_N}}{(t;t)_N (t;t)_{j+r_N-X_N}}\right] \\
&\leq \frac{1}{(1-t)(t;t)_\infty} \E[\bbone(X_N \leq r_N+L-b)t^{2(j+r_N-X_N)}(1-t^{X_N})] \\ 
&\leq \frac{1}{(1-t)(t;t)_\infty} t^{j+r_N-(r_N+L-b)}\E[\bbone(X_N \leq r_N+L-b)t^{j+r_N-X_N}(1-t^{X_N})] \\ 
&\leq \frac{1}{(1-t)(t;t)_\infty} t^{b+2j-L} \E[\bbone(X_N \leq r_N)(t^{r_N-X_N} - t^{r_N})] \\ 
&\leq \frac{1}{(1-t)(t;t)_\infty} t^{b+L} c_N^{-1}
\end{split}
\end{align}
where we used \eqref{eq:lbound} in the second bound, and the fact that $b \geq L$ and $j \geq L$ in the penultimate and last bounds respectively. Plugging \eqref{eq:bound_discrepency} into \eqref{eq:use_thm_single_box} yields
\begin{align}
\label{eq:second_sandwich}
\begin{split}
&\E\left[\bbone(X_N \leq r_N+L-b)t^j\frac{1-t^m}{1-t}(t^{r_N-X_N} - t^{r_N}) \frac{(t;t)_{j+r_N-1} (t;t)_{N-X_N}}{(t;t)_N (t;t)_{j+r_N-X_N}}\right] \\ 
&- \frac{1}{(1-t)(t;t)_\infty} t^{b+L} c_N^{-1}+\Pr(X_N > r_N + L - b) + \ou(c_N^{-1}) \\
&\leq \text{LHS\eqref{eq:1box_asymp_wts}} \\ 
&\leq\E\left[\bbone(X_N \leq r_N+L-b)t^j\frac{1-t^m}{1-t}(t^{r_N-X_N} - t^{r_N}) \frac{(t;t)_{j+r_N-1} (t;t)_{N-X_N}}{(t;t)_N (t;t)_{j+r_N-X_N}}\right] \\ 
& \quad \quad \quad +\Pr(X_N > r_N + L - b) + \ou(c_N^{-1}).
\end{split}
\end{align}
We now wish to show that the $\E[\cdots]$ in the lower and upper bounds is uniformly asymptotic to $c_N^{-1}t^j(1-t^m)/(1-t)$. Note that the $q$-Pochhammer quotient in \eqref{eq:second_sandwich} is
\begin{equation}\label{eq:product_and_extra}
\frac{(t;t)_{j+r_N-1} (t;t)_{N-X_N}}{(t;t)_N (t;t)_{j+r_N-X_N}} = \frac{1}{1-t^{j+r_N}} \prod_{i=1}^{X_N} \frac{1-t^{j+r_N-X_N+i}}{1-t^{N-X_N+i}}.
\end{equation}
In the upcoming estimate \eqref{eq:product_X_to_b} we will wish to bound the difference between the above \eqref{eq:product_and_extra} and $1$. First write
\begin{equation}\label{eq:remove_extra_t_factor}
    \abs*{1-\frac{1}{1-t^{j+r_N}} \prod_{i=1}^{X_N} \frac{1-t^{j+r_N-X_N+i}}{1-t^{N-X_N+i}}} = \abs*{1 - \frac{1}{1-t^{j+r_N}}} + \frac{1}{1-t^{j+r_N}} \cdot \abs*{1-\prod_{i=1}^{X_N} \frac{1-t^{j+r_N-X_N+i}}{1-t^{N-X_N+i}}}.
\end{equation}
For $X_N \leq r_N+L$, the product in \eqref{eq:product_and_extra} (excluding the $(1-t^{j+r_N})^{-1}$ term) is $\leq 1$ since $r_N + L \leq j+r_N \leq N$. Furthermore, it is a decreasing function of $X_N \in \{0,1,\ldots,r_N+L\}$, since $j+r_N \leq N$ so the numerator decreases more than the denominator when $X_N$ increases. Hence since $b \geq 0$ we have
\begin{align}
\begin{split}\label{eq:prod_sandwich}
0& \leq \bbone(X_N \leq r_N+L-b)\left(1-\prod_{i=1}^{X_N} \frac{1-t^{j+r_N-X_N+i}}{1-t^{N-X_N+i}}\right) \\
 & \leq \bbone(X_N \leq r_N+L-b)\left(1 - \prod_{i=1}^{r_N+j-b} \frac{1-t^{b+i}}{1-t^{b+(N-r_N-j)+i}}\right).
\end{split}
\end{align}

A naive bound using the $q$-binomial theorem gives
\begin{equation}\label{eq:bound_b_product}
1 - \prod_{i=1}^{r_N+j-b} \frac{1-t^{b+i}}{1-t^{b+(N-r_N-j)+i}} \leq 1 - \prod_{i=1}^{r_N+j-b} (1-t^{b+i}) \leq 1-(t^b;t)_\infty \leq Ct^b.
\end{equation}
for some constant $C$ and all large $b$. Hence
\begin{equation}
\label{eq:final_bound_b_product}
0 \leq \bbone(X_N \leq r_N+L-b)\left(1-\frac{(t;t)_{j+r_N-1} (t;t)_{N-X_N}}{(t;t)_N (t;t)_{j+r_N-X_N}}\right) \leq Ct^b \bbone(X_N \leq r_N+L-b).
\end{equation}
The other term in \eqref{eq:remove_extra_t_factor} is easy to bound:
\begin{equation}\label{eq:extra_t_bound}
    \abs*{1 - \frac{1}{1-t^{j+r_N}}}  \leq C'' t^{r_N}
\end{equation}
for some $C''$ depending only on $L$, since $j \geq L$. Substituting \eqref{eq:final_bound_b_product} and \eqref{eq:extra_t_bound} into \eqref{eq:remove_extra_t_factor} yields
\begin{equation}\label{eq:final_b_bound_for_real}
    \bbone(X_N \leq r_N+L-b)\abs*{1-\frac{1}{1-t^{j+r_N}} \prod_{i=1}^{X_N} \frac{1-t^{j+r_N-X_N+i}}{1-t^{N-X_N+i}}} \leq (C''t^{r_N} + Ct^b)\bbone(X_N \leq r_N+L-b).
\end{equation}
We now wish to bound the difference between the expectation (the main term) in \eqref{eq:second_sandwich} and the desired asymptotic. Plugging in the formula for $c_N^{-1}$ and the above bound \eqref{eq:final_b_bound_for_real} yields the bound
\begin{align}
\begin{split}\label{eq:product_X_to_b}
&\left| \E\left[\bbone(X_N \leq r_N+L-b)t^j\frac{1-t^m}{1-t}(t^{r_N-X_N} - t^{r_N}) \frac{(t;t)_{j+r_N-1} (t;t)_{N-X_N}}{(t;t)_N (t;t)_{j+r_N-X_N}} - t^j\frac{1-t^m}{1-t}c_N^{-1}\right] \right| \\ 
\leq& \frac{1-t^m}{1-t}t^j\left|  \E\left[\bbone(X_N \leq r_N+L-b)\left(1-\frac{(t;t)_{j+r_N-1} (t;t)_{N-X_N}}{(t;t)_N (t;t)_{j+r_N-X_N}}\right)(t^{r_N-X_N} - t^{r_N})\right]\right| \\ 
&+ \frac{1-t^m}{1-t}t^j\left| \E\left[\bbone(r_N+L-b < X_N \leq r_N)(t^{r_N-X_N} - t^{r_N})\right] \right| \\ 
\leq& (C t^b+C''t^{r_N})\frac{1-t^m}{1-t} t^j\E\left[\bbone(X_N \leq r_N+L-b)(t^{r_N-X_N} - t^{r_N}) \right] \\ 
&+ \frac{1-t^m}{1-t}t^j\E\left[\bbone(r_N+L-b < X_N \leq r_N)(t^{r_N-X_N} - t^{r_N})\right]
\end{split}
\end{align}
using that $b \geq L$ (otherwise the bounds in the last indicator function would be reversed). Since $r_N+L-b \leq r_N$ and the argument of the expectation is nonnegative,
\begin{equation}
\E\left[\bbone(X_N \leq r_N+L-b)(t^{r_N-X_N} - t^{r_N}) \right] \leq c_N^{-1}.
\end{equation}
Furthermore,
\begin{equation}
\E\left[\bbone(r_N+L-b < X_N \leq r_N)(t^{r_N-X_N} - t^{r_N})\right] \leq (1-t^{r_N})\Pr(X_N > r_N + L - b).
\end{equation}
We thus obtain, bounding $1-t^m \leq 1$, that
\begin{equation}\label{eq:bound_the_bounded_bound}
\text{RHS\eqref{eq:product_X_to_b}} \leq C \frac{t^{b+j}}{1-t} c_N^{-1} + C''\frac{t^{r_N+j}}{1-t} c_N^{-1}  + \frac{t^j}{1-t}\Pr(X_N > r_N + L - b).
\end{equation}
Finally, we let $b$ depend on $N$ as follows. Since 
\begin{equation}
\Pr(X_N > r_N + L - b) = \ou(c_N^{-1})
\end{equation}
for all $b$, by a diagonalization argument there exists a slowly growing sequence $b=b(N)$ such that 
\begin{equation}\label{eq:prob_b(N)_small_2}
\Pr(X_N > r_N + L - b(N)) = \ou(c_N^{-1})
\end{equation}
(the uniformity over $j$ is obvious here but we keep the $\ou$ notation anyway). Substituting \eqref{eq:product_X_to_b}, \eqref{eq:bound_the_bounded_bound}, and \eqref{eq:prob_b(N)_small_2} to simplify the upper and lower bounds in our original inequality \eqref{eq:second_sandwich} thus yields that the inequalities
\begin{align}
\label{eq:shrimp_po'_boy}
\begin{split}
&O_{unif}(t^{b(N)}c_N^{-1}) + O_{unif}(t^{r_N}c_N^{-1}) +  \ou(c_N^{-1}) + \frac{1-t^m}{1-t}t^jc_N^{-1}
\\
&\leq \text{LHS\eqref{eq:1box_asymp_wts}} \\ 
&\leq O_{unif}(t^{b(N)}c_N^{-1}) + O_{unif}(t^{r_N}c_N^{-1}) + \ou(c_N^{-1}) + \frac{1-t^m}{1-t}t^jc_N^{-1},
\end{split}
\end{align}
with implied constants which are uniform over all $j \geq L$, hold for all $N$ sufficiently large that $r_N + L -b(N) \geq 0$. Since $b(N) \to \infty$ and $r_N \to \infty$, this shows that 
\begin{equation}
\text{LHS\eqref{eq:1box_asymp_wts}}  = \frac{1-t^m}{1-t}t^jc_N^{-1} + \ou(c_N^{-1}).
\end{equation}
\end{proof}

\section{Deducing \Cref{thm:uniform_matrix_intro} and \Cref{thm:haar_corner_intro}} \label{sec:deducing}

This section consists in using the dynamical results established in this paper to bootstrap the one-point distribution results of \cite{van2023local} to the multi-time distribution claimed in \Cref{thm:uniform_matrix_intro} and the similar result \Cref{thm:haar_corner_intro}, which we state now. 
\begin{thm}\label{thm:haar_corner_intro}
For each $N \in \Z_{\geq 1}$, let $D_N \in \Z_{\geq 1}$ and let $A^{(N)}_i, i \geq 1$ be iid matrices distributed as the top-left $N \times N$ submatrix of a Haar-distributed element of $\GL_{N+D_N}(\Z_p)$. Let $(r_N)_{N \geq 1}$ be any integer sequence such that $r_N \to \infty$ and $N-r_N \to \infty$ as $N \to \infty$. Define
\begin{equation}
\Lambda^{(N)}(T) := s^{r_N} \circ \iota(\SN(A^{(N)}_{\floor{p^{r_N} T/(1-p^{-D_N})}} \cdots A^{(N)}_1)),T \in \R_{\geq 0}
\end{equation}
Then
\begin{equation}
\La^{(N)}(T) \to \Pois^{(2\infty)}(T)
\end{equation}
in the same sense as \Cref{thm:uniform_matrix_intro}.
\end{thm}

The bootstrapping requires translating the results of \cite{van2023local} into formulas for the single-time distribution of the reflecting Poisson sea and the matrix product process, to check that they agree. The main result of this section is \Cref{thm:matrix_1pt}, and \Cref{thm:uniform_matrix_intro} and \Cref{thm:haar_corner_intro} follow easily from it.

In this section we will use the notation 
\begin{equation}
\Pois^{(2\infty)}(T) = \Pois^{(0)_{i \in \Z},2\infty}(T)
\end{equation}
as we did for the finite analogues in \Cref{def:poisson_walks}.

We first give an explicit name $\cL$ to the joint distributions of conjugate parts of $\Pois^{(2\infty)}$. This distribution was computed explicitly in \cite{van2023local}, which we invite the reader to peruse, and in that work we defined the notation $\cL$ by those explicit formulas (\cite[Theorem 6.1]{van2023local}). Because the formulas require a fair amount of setup, here we simply state the properties we need which were shown in \cite{van2023local}.

\begin{thm}\label{thm:use_other_paper}
For any $k \in \Z_{\geq 0}, t \in (0,1)$, there is a family of $\Sig_k$-valued random variables $\cL_{k,t,\chi}$, parametrized by $\chi \in \R_{>0}$, which appears in the following limits:
\begin{enumerate}[label=(\Roman*)]
\item For any $\zeta \in \R$, letting $\tau_N = t^{\zeta-N}, N \in \Z_{\geq 1}$, we have
\begin{equation}
(\Pois^{(\infty)}(\tau_N)_i' - \log_{t^{-1}} \tau_N - \zeta)_{1 \leq i \leq k} \to \cL_{k,t,t^{\zeta+1}/(1-t)}
\end{equation}
in distribution as $N \to \infty$. \label{item:pois_limit}
\item Fix $p$ prime, and for each $N \in \mathbb{Z}_{\geq 1}$ let $A_{i}^{(N)}, i \geq 1$ be iid matrices with iid entries distributed by the additive Haar measure on $\mathbb{Z}_{p}$. Let $(s_{N})_{N \geq 1}$ be a sequence of natural numbers such that $s_{N}$ and $N-\log _{p} s_{N}$ both go to $\infty$ as $N \rightarrow \infty$. Let $(s_{N_{j}})_{j \geq 1}$ be any subsequence for which $-\log _{p} s_{N_{j}}$ converges in $\mathbb{R} / \mathbb{Z}$, and let $\zeta$ be any preimage in $\mathbb{R}$ of this limit. Then
\begin{equation}\label{eq:rmt_bulk_intro}
(\mathrm{SN}(A_{s_{N_{j}}}^{(N_{j})} \cdots A_{1}^{(N_{j})})_{i}'-[\log _{p}(s_{N_{j}})+\zeta])_{1 \leq i \leq k} \rightarrow \cL_{k,p^{-1}, p^{-\zeta}/(p-1)}
\end{equation}
in distribution as $j \rightarrow \infty$, where $[\cdot]$ is the nearest integer function. \label{item:haar_limit}
\item Fix $p$ prime, and for each $N \in \mathbb{Z}_{\geq 1}$ let $D_N \in \Z_{\geq 1}$ be an integer and $A_{i}^{(N)}, i \geq 1$ be iid $N \times N$ corners of matrices distributed by the Haar probability measure on $\GL_{N+D_N}(\Z_p)$. Let $(s_{N})_{N \geq 1}$ be a sequence of natural numbers such that $s_{N}$ and $N-\log _{p} s_{N}$ both go to $\infty$ as $N \rightarrow \infty$. Let $(s_{N_{j}})_{j \geq 1}$ be any subsequence for which $-\log _{p} ((1-p^{-D_{N_j}})s_{N_{j}})$ converges in $\mathbb{R} / \mathbb{Z}$, and let $\zeta$ be any preimage in $\mathbb{R}$ of this limit. Then
\begin{equation}\label{eq:corner_bulk_intro}
(\mathrm{SN}(A_{s_{N_{j}}}^{(N_{j})} \cdots A_{1}^{(N_{j})})_{i}'-[\log _{p}((1-p^{-D_{N_j}})s_{N_{j}})+\zeta])_{1 \leq i \leq k} \rightarrow \cL_{k,p^{-1}, p^{-\zeta}/(p-1)}
\end{equation}
in distribution as $j \rightarrow \infty$. \label{item:corner_limit}
\end{enumerate}
\end{thm}
\begin{proof}
The explicit definition of $\cL_{k,t,\chi}$ is given in \cite[Theorem 6.1]{van2023local}. The limits \ref{item:pois_limit}, \ref{item:haar_limit} and \ref{item:corner_limit} are given in Proposition 10.1, Theorem 1.2, and Theorem 1.3 of \cite{van2023local} respectively. Note that in all of these cases, the limit in \cite{van2023local} is phrased as joint convergence of all (recentered) conjugate parts to a $\Sig_\infty$-valued random variable $\cL_{t,\chi}$, but this yields the desired claims since $\cL_{t,\chi}$ is defined in \cite[Theorem 6.1]{van2023local} by specifying that its first $k$ coordinates are distributed as $\cL_{k,t,\chi}$.
\end{proof}

\begin{prop}\label{thm:sea_cL}
The marginal distributions of the conjugate parts of the reflecting Poisson sea with parameter $t$ are given by
\begin{equation}
(\Pois^{(2\infty)}(T)_i')_{1 \leq i \leq k} = \cL_{k,t,tT/(1-t)},
\end{equation}
where $\cL$ is as in \Cref{thm:use_other_paper} or \cite[Theorem 6.1]{van2023local}.
\end{prop}
\begin{proof}
Specializing \Cref{def:limit_process} to the case $\mu = (0[2\infty])$ and translating $\tPois^{(0[\infty]),n}(T)$ to $\Pois^{(\infty)}(T)$ via \Cref{def:half_infinite_process}, we have that
\begin{equation}
\Pois^{(2\infty)}(T)_i = \lim_{n \to \infty} \Pois^{(\infty)}(t^{-n}T)_{i+n}
\end{equation}
in joint distribution over any finite collection of indices $i$. Hence 
\begin{equation}\label{eq:conj_part_pois2infty}
(\Pois^{(2\infty)}(T)_i')_{1 \leq i \leq k} = \lim_{n \to \infty} (\Pois^{(\infty)}(t^{-n}T)_i' - n)_{1 \leq i \leq k}
\end{equation}
in distribution. By taking $\zeta = \log_t T$ in \Cref{thm:use_other_paper} \ref{item:pois_limit}, the right hand side of \eqref{eq:conj_part_pois2infty} is $\cL_{k,t,tT/(1-t)}$, which completes the proof.
\end{proof}

For the random matrix side, we need several lemmas. 

\begin{lemma}\label{thm:kernel_size}
Let $A$ be the top-left $N \times N$ corner of uniform element of $\GL_{N+D}(\F_q)$. Then
\begin{equation}\label{eq:kernel_corner}
\E[\#\ker(A)] = \E[q^{\corank(A)}] = \frac{1-q^{-D} + 1-q^{-N}}{1-q^{-N-D}}.
\end{equation}
If instead $A$ is uniformly distributed over $\Mat_N(\F_q)$,
\begin{equation}\label{eq:kernel_haar}
\E[\#\ker(A)] = \E[q^{\corank(A)}] = 2-q^{-N}.
\end{equation}
\end{lemma}

\begin{proof}
The following proof was given to us by an anonymous referee, and is essentially a quotation from the report. We give it here because it seems more transparent than the one given in the original arxiv version of this paper, which we have moved to the \hyperref[sec:appendix]{Appendix}.

We first prove \eqref{eq:kernel_haar}. We note that the left-hand side is just $\frac{1}{q^{N^2}}$ times the number of pairs $(A, v)$
such that $A \in \Mat_N(\mathbb{F}_q)$, $v \in \mathbb{F}_q^N$ and $Av = 0$. We count the number by fixing $v$ first.
If $v = 0$, then there are $q^{N^2}$ choices of $A$. If $v \neq 0$ ($q^N - 1$ choices), then every row of $A$
must be in the $(N - 1)$-dimensional subspace orthogonal to $v$, so there are $q^{N(N-1)}$ choices of $A$.
We thus get
\[
  \E(\#\ker(A)) = \frac{1}{q^{N^2}}\bigl(q^{N^2} + (q^N - 1)q^{N^2-N}\bigr) = 2 - q^{-N}.
\]

We now prove \eqref{eq:kernel_corner}. The left-hand side is equal to $\frac{1}{\#\GL_{N+D}(\mathbb{F}_q)}$ times the number of pairs $(A', v)$,
$A' \in \GL_{N+D}(\mathbb{F}_q)$, $v \in \mathbb{F}_q^N$, such that the first $N$ entries of
$A'\begin{pmatrix} v \\ 0 \end{pmatrix}$ are zero. We again count it by fixing $v$ first. If $v = 0$,
then there are $\#\GL_{N+D}(\mathbb{F}_q)$ choices of $A'$. If $v \neq 0$ ($q^N - 1$ choices), then since all
$\begin{pmatrix} v \\ 0 \end{pmatrix}$ for $v \neq 0$ live in the same left-$\GL_{N+D}(\mathbb{F}_q)$-orbit,
we may assume without loss of generality that $v = (1, 0, \ldots, 0)$. Then the condition for $A'$
becomes that the first $N$ entries of the first column of $A'$ are zero. The remaining $D$ entries of the
first column must form a nonzero vector ($q^D - 1$ choices), and the remaining columns of $A'$ must be obtained
by completing the first column to a basis ($\#\GL_{N+D}(\mathbb{F}_q)/(q^{N+D} - 1)$ choices). Putting these together and
canceling $\#\GL_{N+D}(\mathbb{F}_q)$, we get
\[
  \E[\#\ker(A)] = 1 + \frac{(q^N - 1)(q^D - 1)}{q^{N+D} - 1},
\]
equivalent to the desired formula.
\end{proof}

\begin{rmk}
\Cref{thm:kernel_size} has the amusing consequence that for uniform $A_N \in \Mat_N(\F_q)$, 
\begin{equation}
\lim_{N \to \infty} \E[\#\ker(A_N)] = 2,
\end{equation} 
and in particular is independent of $q$.
\end{rmk}

We may now compute the constant $c_N$ of \Cref{thm:gen_limit} explicitly in the cases of interest.

\begin{lemma}\label{thm:compute_c_N}
Let $t=1/p$, $(r_N)_{N \geq 1}$ be any sequence with $r_N \to \infty$, $(D_N)_{N \geq 1}$ be any sequence of positive integers, $A_N$ be the top-left $N \times N$ corner of Haar-distributed element of $\GL_{N+D_N}(\Z_p)$, $X_N := \corank(A_N)$, and
\begin{equation}
c_N = \frac{t^{-r_N}}{\E[\bbone(X_N \leq r_N)(t^{-X_N}-1)]}
\end{equation}
as in \Cref{thm:gen_limit}. Then
\begin{equation}\label{eq:c_N_corner}
c_N = \frac{t^{-r_N}}{1-t^{D_N}}(1+o(1))
\end{equation}
as $N \to \infty$, where the $o(1)$ is uniform over all choices of the $D_N$. If instead $A_N$ has iid additive Haar entries, then
\begin{equation}\label{eq:c_N_haar}
c_N = t^{-r_N}(1+o(1)).
\end{equation}
\end{lemma}

\begin{proof}
We first claim that for $A_N$ as in the statement and $X_N = \corank(A_N) \pmod{p}$, both $\GL_{N+D}$ corners and additive Haar matrices,
\begin{equation}\label{eq:remove_indicator}
c_N = \frac{t^{-r_N}}{\E[t^{-X_N}-1]}(1+o(1)),
\end{equation}
i.e. the indicator function can be removed. Unpacking, we must show
\begin{equation}
    \label{eq:remove_indicator_unpacked}
    \E[\bbone(X_N > r_N)(t^{r_N-X_N} - t^{r_N})] = o(\E[\bbone(X_N \leq r_N)(t^{r_N-X_N} - t^{r_N})]).
\end{equation}

Naively bounding $(t;t)_\infty \leq (t;t)_n \leq 1$, \eqref{eq:corank_formula_in_proof} yields
\begin{equation}\label{eq:corank_c_upper_bound}
\Pr(\corank(A_N) = c) \leq t^{c^2} (t;t)_\infty^{-5}.
\end{equation}
This yields the bound
\begin{equation}\label{eq:c_N_remainder_bound}
t^{r_N} \E[\bbone(X_N > r_N)(t^{-X_N}-1)] \leq \text{const} \cdot t^{r_N} \sum_{c > r_N} t^{c^2} (t^{-c}-1) = o(t^{r_N})
\end{equation}
on the left hand side of \eqref{eq:remove_indicator_unpacked}.

From \eqref{eq:corank_formula_in_proof} and the naive $t$-Pochhammer bound $(t;t)_\infty \leq (t;t)_n \leq 1$, we also obtain
\begin{equation}
\Pr(\corank(A_N) = 1) \geq t (t;t)^4
\end{equation}
for any $N$ and $D_N$, hence (for $r_N \geq 1$) we have an upper bound
\begin{equation}\label{eq:naive_c_N_bound}
\E[\bbone(\corank(A_N) \leq r_N) (t^{r_N-X_N} - t^{r_N})] \geq t^{r_N} t (t;t)^4 (t^{-1}-1)
\end{equation}
on the right hand side of \eqref{eq:remove_indicator_unpacked}. Combining \eqref{eq:c_N_remainder_bound} with \eqref{eq:naive_c_N_bound} proves \eqref{eq:remove_indicator_unpacked} and hence \eqref{eq:remove_indicator}.

Now, for $\GL_{N+D_N}(\Z_p)$ corners, \Cref{thm:kernel_size} yields
\begin{equation}
\frac{t^{-r_N}}{\E[t^{-X_N}] - 1} = \frac{t^{-r_N}}{\frac{1-t^{D_N} + 1-t^{N}}{1-t^{N+D_N}}-1} = \frac{t^{-r_N}}{1-t^{D_N}}(1+o(1))
\end{equation}
as $N \to \infty$. The case of iid Haar entries similarly follows from \eqref{eq:kernel_haar}.
\end{proof}

Using \Cref{thm:use_other_paper}, \Cref{thm:sea_cL}, and the above estimate \Cref{thm:compute_c_N}, we may now show the following. This gives us the needed 1-point results for \Cref{thm:uniform_matrix_intro} and \Cref{thm:haar_corner_intro}.

\begin{prop}\label{thm:matrix_1pt}
For each $N \in \Z_{\geq 1}$, let $D_N \in \Z_{\geq 1}$ and let $A^{(N)}_i, i \geq 1$ be iid matrices distributed as the top-left $N \times N$ submatrix of a Haar-distributed element of $\GL_{N+D_N}(\Z_p)$. Let $(r_N)_{N \geq 1}$ be any integer sequence such that $r_N \to \infty$ and $N-r_N \to \infty$ as $N \to \infty$, and define
\begin{equation}
\label{eq:iidhaar_tL}
\Lambda^{(N)}(T) := s^{r_N} \circ \iota(\SN(A^{(N)}_{\floor*{p^{r_N} T/(1-p^{-D_N})}} \cdots A^{(N)}_1)),T \in \R_{\geq 0}.
\end{equation}
Then for any fixed $T \in \R_{>0}, k \in \Z_{\geq 1}$,
\begin{equation}
(\L^{(N)}(T)_i')_{1 \leq i \leq k}\to (\Pois^{(2\infty)}(T)_i')_{1 \leq i \leq k}
\end{equation}
in distribution, where the parameter $t$ in $\Pois^{(2\infty)}$ is set to $1/p$. The same result holds if instead the matrices $A^{(N)}_i$ have iid additive Haar entries and one redefines $\La^{(N)}(T)$ to be $s^{r_N} \circ \iota(\SN(A^{(N)}_{\floor*{p^{r_N} T}} \cdots A^{(N)}_1)),T \in \R_{\geq 0}$.
\end{prop}

\begin{proof}
We show the case of $\GL_{N+D_N}$ corners first. Let $s_N = \floor*{p^{r_N}T/(1-p^{-D_N})}$. Then $-\log_p (1-p^{-D_N})s_N$ clearly converges in $\R/\Z$, and $\zeta := -\log_p T$ is a preimage in $\R$ of the limit. For this $\zeta$ it is clear that

\begin{equation}
[\log_p (1-p^{-D_N})s_N + \zeta] = r_N
\end{equation}
for all sufficiently large $N$, so
\begin{equation}\label{eq:L_cL_conv}
(\L^{(N)}(T)_i')_{1 \leq i \leq k} = (\SN(A^{(N)}_{\floor*{p^{r_N} T/(1-p^{-D_N})}} \cdots A^{(N)}_1)_i' - r_N)_{1 \leq i \leq k} \to \cL_{k,p^{-1},p^{-\zeta}/(p-1)}
\end{equation}
by \Cref{thm:use_other_paper} \ref{item:corner_limit}. By \Cref{thm:sea_cL}, we have
\begin{equation}
(\Pois^{(2\infty)}(T)_i')_{1 \leq i \leq k} = \cL_{k,t,tT/(1-t)}
\end{equation}
in distribution, and taking $t=p^{-1}$ this matches the right hand side of \eqref{eq:L_cL_conv}, completing the proof. The case of iid Haar matrices is the same (without the $1-t^{D_N}$ factors), since we have also computed the relevant asymptotics for $c_N$ in \Cref{thm:compute_c_N} in this case, and the relevant limit of $\L^{(N)}(T)_i'$ in \Cref{thm:use_other_paper} \ref{item:haar_limit}.
\end{proof}

Let us record the verification of the main nontrivial hypothesis of \Cref{thm:gen_limit} for our examples:

\begin{lemma}\label{thm:corank_bound}
Let $(r_N)_{N \geq 1}$ be any sequence with $r_N \to \infty$, $(D_N)_{N \geq 1}$ be any sequence of positive integers, $A_N$ be the top-left $N \times N$ corner of uniform element of $\GL_{N+D_N}(\F_q)$. Then
\begin{equation}
\Pr(\corank(A_N) > r_N) = o(\E[\bbone(\corank(A_N) \leq r_N) (t^{r_N-\corank(A_N)} - t^{r_N})])
\end{equation}
as $N \to \infty$, where $t=1/q$. The same result holds if instead $A_N$ is uniformly distributed over $\Mat_N(\F_q)$.
\end{lemma}
\begin{proof}
This expression \eqref{eq:corank_formula_in_proof} decays like $t^{c^2}$, and so it is easy to check that there exists $C$ such that
\begin{equation}
    \Pr(\corank(A_N) \geq k | \corank(A_N) > 0) < C t^{1.0001}
\end{equation}
(one can also get stronger bounds from the manipulations of the previous proof, but these are unnecessary). \Cref{thm:prob_hypothesis} then concludes the proof.


In the case where $A_N$ is uniformly random over $\Mat_N(\F_q)$, $\Pr(\corank(A) = c)$ is given by the $D \to \infty$ limit of the formula on the right hand side of \eqref{eq:corank_formula_in_proof}, by \Cref{thm:corank_pt1}, and the same bounds given above still apply.
\end{proof}

\begin{proof}[Proof of \Cref{thm:haar_corner_intro}]
We must show that for any times $0 < T_1 < T_2 < \ldots < T_\ell$, 
\begin{equation}
(\La^{(N)}(T_1),\ldots,\La^{(N)}(T_\ell)) \to (\Pois^{(2\infty)}(T_1),\ldots,\Pois^{(2\infty)}(T_\ell)).
\end{equation}
Since $\La^{(N)}(T)_0' = N-r_N \to \infty$, the hypothesis of convergence of finite sets of coordinates from \Cref{thm:weak_eq} holds for $\La^{(N)}(T)$ if and only if it holds for the modified version $\tLa^{(N)}(T)$ where all parts $-\infty$ are replaced by $0$.  Since $\tLa^{(N)}(T)$ is in $\bSig_{2\infty}^+$, by \Cref{thm:finite_sets_to_conj_parts} it suffices to show that 
\begin{equation}
(\La^{(N)}(T_i)_j')_{\substack{1 \leq i \leq \ell \\ 1 \leq j \leq d}} \to (\Pois^{(2\infty)}(T_i)_j')_{\substack{1 \leq i \leq \ell \\ 1 \leq j \leq d}}
\end{equation}
in distribution as $N \to \infty$ (where we used that $\La^{(N)}(T)_j' = \tLa^{(N)}(T)_j'$ for $j \geq 1$). Since both sides are Markov processes on $\Sig_d$, it suffices to show for every $\eta \in \Sig_d$ that
\begin{multline}\label{eq:uniform_bootstrap}
\Law((\La^{(N)}(T_i)_j')_{\substack{2 \leq i \leq \ell \\ 1 \leq j \leq d}} | (\La^{(N)}(T_1)_j')_{1 \leq j \leq d} = \eta) \\ 
\to \Law((\Pois^{(2\infty)}(T_i)_j')_{\substack{2 \leq i \leq \ell \\ 1 \leq j \leq d}} | (\Pois^{(2\infty)}(T_1)_j')_{1 \leq j \leq d} = \eta)
\end{multline}
and
\begin{equation}\label{eq:uniform_1pt_inproof}
\Pr((\La^{(N)}(T_1)_j')_{1 \leq j \leq d} = \eta) \to \Pr((\Pois^{(2\infty)}(T_1)_j')_{1 \leq j \leq d} = \eta)
\end{equation}
as $N \to \infty$. The one-point equality \eqref{eq:uniform_1pt_inproof} is exactly the limit of \Cref{thm:matrix_1pt}, so we turn to \eqref{eq:uniform_bootstrap}. Let $X_N = \corank(A^{(N)}_i \pmod{p})$ and
\begin{equation}
c_N = \frac{t^{-r_N}}{\E[\bbone(X_N \leq r_N)(t^{-X_N}-1)]}
\end{equation}
as usual. We define
\begin{equation}
\hLa^{(N)}(T) := s^{r_N} \circ \iota(\SN(A^{(N)}_{\floor*{c_N T}} \cdots A^{(N)}_1)),T \in \R_{\geq 0}
\end{equation}
which is just $\La^{(N)}$ with a slightly different time-change which corresponds to the one in \Cref{thm:gen_limit}. Then we claim that
\begin{multline}\label{eq:La_to_hLa}
| \Pr((\La^{(N)}(T_{i+1})_j')_{1 \leq j \leq d} = \eta^{(i+1)}| (\La^{(N)}(T_{i})_j')_{1 \leq j \leq d} = \eta^{(i)}) \\ 
- \Pr((\hLa^{(N)}(T_{i+1})_j')_{1 \leq j \leq d} = \eta^{(i+1)}| (\hLa^{(N)}(T_{i})_j')_{1 \leq j \leq d} = \eta^{(i)})| \to 0
\end{multline}
as $N \to \infty$, for any $\eta^{(i)},\eta^{(i+1)} \in \Sig_d$. This follows because (a)
\begin{equation}
c_N = \frac{t^{-r_N}}{1-t^{D_N}}(1+o(1))
\end{equation}
by \Cref{thm:compute_c_N}, and (b) the transition rates of $(\La^{(N)}(T_{i+1})_j')_{1 \leq j \leq d}$ and $(\hLa^{(N)}(T_{i+1})_j')_{1 \leq j \leq d}$ out of any state reachable from $\eta^{(i)}$ are bounded above by $\text{const} \cdot t^{(\eta^{(i)})'_d}$ by \Cref{thm:kappa=nu_case}, so the multiplicative $1+o(1)$ factor does not matter. From \eqref{eq:La_to_hLa} it follows that to show \eqref{eq:uniform_bootstrap}, it suffices to show the same limit with $\hLa$ in place of $\La$, namely
\begin{multline}\label{eq:uniform_bootstrap_hLa}
\Law((\hLa^{(N)}(T_i)_j')_{\substack{2 \leq i \leq \ell \\ 1 \leq j \leq d}} | (\hLa^{(N)}(T_1)_j')_{1 \leq j \leq d} = \eta) \\ 
\to \Law((\Pois^{(2\infty)}(T_i)_j')_{\substack{2 \leq i \leq \ell \\ 1 \leq j \leq d}} | (\Pois^{(2\infty)}(T_1)_j')_{1 \leq j \leq d} = \eta).
\end{multline}
This now follows directly\footnote{This is where we need \Cref{thm:corank_bound}.} from \Cref{thm:gen_limit_F_d}, as our $\hLa^{(N)}(T-T_1)$ corresponds to $\La^{(N)}(T)$ in that result with initial condition specified by $\eta$. Hence we have shown \eqref{eq:uniform_bootstrap} and \eqref{eq:uniform_1pt_inproof}, which completes the proof.
\end{proof}

\begin{proof}[Proof of \Cref{thm:uniform_matrix_intro}]
Same as the proof of \Cref{thm:haar_corner_intro} after removing the $1-t^{D_N}$ factors everywhere above, since each specific lemma we use was also proven for the additive Haar case.
\end{proof}

\begin{appendix}
\section*{Original proof of Lemma~{\ref{thm:kernel_size}}}\label{sec:appendix}

As mentioned in the previous section, our original proof of \Cref{thm:kernel_size} in the arxiv version was by $q$-series identities, but the referee suggested a more transparent proof, which is the one given above. Here we record the original.

\begin{proof}
    
Set $t=q^{-1}$ throughout. We will show \eqref{eq:kernel_corner} first.

Let $\tA$ be a uniform element of $\GL_{N+D}(\F_q)$. Then the matrix formed by the top $N$ rows of $\tA$ is uniformly distributed among full-rank $N \times (N+D)$ matrices, so \Cref{thm:corank_pt2} yields
\begin{equation}\label{eq:corank_formula_in_proof}
\Pr(\corank(A) = c) = t^{c^2} \frac{\sqbinom{D}{c}_t \sqbinom{N}{N-c}_t}{\sqbinom{N+D}{N}_t}
\end{equation}
(with $d,n,k,r$ in the notation of \Cref{thm:corank_pt2} corresponding to $D,N,N,N-c$ above). 

Since $t^{\corank(A)} = \#\ker(A)$ and 
\begin{equation}
\sqbinom{N+D}{N}_t \frac{1-t^D + 1-t^N}{1-t^{N+D}} = \sqbinom{N+D-1}{D} + \sqbinom{N+D-1}{D-1},
\end{equation}
to show the desired
\begin{equation}
    \sum_{c=0}^N t^c \Pr(\corank(A)=c) = \frac{1-t^D+1-t^N}{1-t^{N+D}}
\end{equation}
it suffices to show
\begin{equation}\label{eq:corner_qseries_wts}
\sum_{c=0}^N t^{c^2-c} \sqbinom{D}{c}_t \sqbinom{N}{N-c}_t = \sqbinom{N+D-1}{D} + \sqbinom{N+D-1}{D-1}.
\end{equation}
The $q$-Chu-Vandermonde identity (see e.g. \cite[II.6]{gasper_rahman_2004}) yields
\begin{equation}
\sum_{j=0}^m q^{(m-j)(k-j)} \sqbinom{m}{j}_t \sqbinom{n}{k-j}_t = \sqbinom{m+n}{k}.
\end{equation}
Applying this twice with $(k,m,n) = (D-1,D,N-1)$ and $(k,m,n) = (D,D,N-1)$
\begin{align}
\begin{split}
\text{RHS\eqref{eq:corner_qseries_wts}} &= \sum_{j=0}^{D-1} t^{(D-j)(D-1-j)} \sqbinom{D}{j}_t \sqbinom{N-1}{D-1-j}_t + \sum_{j=0}^{D} t^{(D-j)(D-j)} \sqbinom{D}{j}_t \sqbinom{N-1}{D-j}_t \\ 
&= \sum_{j=0}^D t^{(D-j)(D-j-1)}\sqbinom{D}{j}_t \left(\sqbinom{N-1}{D-1-j}_t + t^{D-j}\sqbinom{N-1}{D-j}_t\right) \\ 
&= \sum_{j=0}^D t^{(D-j)(D-j-1)} \sqbinom{D}{j}_t \sqbinom{N}{D-j}_t
\end{split}
\end{align}
by the $q$-Pascal identity, and this is equal to the left hand side of \eqref{eq:corner_qseries_wts} after relabeling $c=D-j$. This shows \eqref{eq:kernel_corner}.

For $A$ as in \eqref{eq:kernel_haar}, $\Pr(\corank(A) = c)$ is given by the $D \to \infty$ limit of the formula on the right hand side of \eqref{eq:corank_formula_in_proof}, by \Cref{thm:corank_pt1}. The proof then similarly reduces to showing the $D \to \infty$ limit of the equation \eqref{eq:corner_qseries_wts} which we have just established, so we are done.
\end{proof}
\end{appendix}

\begin{acks}[Acknowledgments]
I am very grateful to Alexei Borodin for helpful conversations and advice throughout this project, and feedback on the exposition. I wish also to thank Ivan Corwin for exchanges regarding the Gibbs property and line ensemble perspective, Vadim Gorin for helpful comments on the text, and Jimmy He and Matthew Nicoletti for useful discussions on Markov processes and particle systems. Finally, I thank the referees for helpful corrections and comments, in particular for finding a mis-statement in \Cref{thm:universality_intro} and finding a simpler proof of \Cref{thm:kernel_size}.
\end{acks}
\begin{funding}
This paper is based on PhD thesis work \cite{van2023asymptotics} which was supported by an NSF Graduate Research Fellowship under grant no. 1745302, and the paper itself was written (and Theorems \ref{thm:uniform_matrix_intro} and \ref{thm:haar_corner_intro} added) while supported by the European Research Council (ERC), Grant Agreement No. 101002013.
\end{funding}



\bibliographystyle{imsart-number.bst} 
\bibliography{references.bib}       


\end{document}